\documentclass[11pt, a4paper]{article}
\usepackage[T1]{fontenc}
\usepackage{url}
\usepackage[margin=1in]{geometry} 
\usepackage[noadjust]{cite} 
\usepackage[colorlinks=true,citecolor=red,linkcolor=blue,urlcolor=blue]{hyperref}
\usepackage{amsmath}                              
\usepackage{graphicx}         
\usepackage{booktabs}         
\usepackage{paper}
\usepackage{amsmath,amssymb}
\usepackage{color}
\usepackage{authblk}
\usepackage{changes}
\usepackage{cancel}
\begin{document}

\title{On the best approximation by finite Gaussian mixtures} 
\author{Yun Ma, Yihong Wu, 
Pengkun Yang\thanks{Y.\ Ma and P.\ Yang are with Department of Statistics and Data Science, Tsinghua University. Y.\ Wu is with Department of Statistics and Data Science, Yale University. 
P.\ Yang is supported in part by the NSFC Grant 12101353, Tsinghua University Initiative Scientific Research Program, and National Key R\&D Program of China 2024YFA1015800. 
Y.\ Wu is supported in part by the NSF Grant CCF-1900507.
This paper was presented in part at the 2023 IEEE International
Symposium on Information Theory \cite{MWY23}.
}}

\maketitle
\begin{abstract}
We consider the problem of approximating a general Gaussian location mixture by finite mixtures. The minimum order of finite mixtures that achieve a prescribed accuracy
is determined within constant factors for the family of mixing distributions with {compact} support or appropriate assumptions on the tail probability including subgaussian and subexponential. While the upper bound is achieved using the technique of local moment matching, the lower bound is established by relating the best approximation error to the low-rank approximation of certain trigonometric moment matrices, followed by a refined spectral analysis of their minimum eigenvalue. In the case of Gaussian mixing distributions, this result corrects a previous lower bound in \cite{WV10}.

\end{abstract}

\tableofcontents

\section{Introduction}
\label{sec:intro}
Gaussian mixtures are widely applied in statistical modeling of heterogeneous populations.
Let $\phi$ denote the standard Gaussian density.
For a probability distribution $P$ on the real line, denote
by $f_P$ the marginal density of the Gaussian convolution $P*\phi$, that is
\begin{equation*}
  f_P(x)=\int \phi(x-\theta) dP(\theta).
\end{equation*}
We refer to $P$ and $P*\phi$ as the \textit{mixing distribution} and the \textit{mixture}, respectively. 
 Given a general mixture $P*\phi$, the problem of interest is how to best  
approximate it by a finite mixture $P_m*\phi$, where
the support size of $P_m$ is at most $m$ (\ie, $m$-atomic). 

Let $d(f,g)$ denote a loss function that measures the approximation error of $g$ by $f$. Concrete examples include $L_p$ distances or $f$-divergences \cite{csiszar1975divergence}, the latter of which, including the total variation $\TV(f,g)$, squared Hellinger distance $H^2(f,g)$, 
the Kullback-Leibler (KL) divergence $\KL(f\|g)$, and the $\chi^2$-divergence $\chi^2(f\|g)$, are the focus of the present paper\footnote{If $f$ and $g$ are densities for distribution $P$ and $Q$, respectively, we also write $d(P,Q)=d(f,g)$.} (see Appendix \ref{app:f-div} for a quick review).
The best approximation error of $f_P$ 
by an $m$-component mixture is defined as
\begin{equation}
  \label{eq:def-appr}
\appr(m,P,d)\triangleq\inf_{P_m\in \calP_m}d\left(f_{P_{m}} , f_{P}\right),
\end{equation}
where $\calP_m$ denotes the collection of all distributions that are at most $m$-atomic. Considering the worst instance of this pointwise quantity, we define
\begin{equation}
\appr(m,\calP,d)\triangleq\sup_{P\in\calP} \appr(m,P,d)
\end{equation}
as the worst-case approximation error over a family $\calP$ of mixing distributions by $m$-component mixtures.
It is well-known that the optimization problem \eqref{eq:def-appr} is \emph{nonconvex} (in the location parameters) and is generally hard to solve. This shares the essential difficulty of approximation by neural nets with one hidden layer \cite{Barron1993}.

In information theory, the Gaussian convolution structure arises in the context of Gaussian channels~\cite{Cover2006ElementsOI}, where the input and output distributions correspond to $P$ and $P*\phi$, respectively. 
The channel capacity determines the maximal rate at which information can be reliably transmitted, which, under the second moment constraint, {is defined as
\[
    C\triangleq \max_{P_X:\Expect[X^2]\leq \sigma^2} I(X;X+Z),
\]
where the mutual information between two random variables $X$ and $Y$ is defined as $I(X;Y) \triangleq \Expect \log\frac{P_{XY}}{P_XP_Y}$.
It is well-known that the capacity \( C = \frac{1}{2} \log(1 + \sigma^2) \), and it is achieved by a Gaussian input distribution \( N(0, \sigma^2)\).
While Gaussian inputs achieve the theoretical capacity, practical communication systems typically employ modulation schemes that restrict signals to finite and discrete constellations. This constraint introduces the problem of finite-constellation design, which aims to approach the theoretical capacity limit and improve system performance. Such design challenges frequently occur in diverse application domains, including wireless communication~\cite{Wu2018MIMO}, quantum cryptography~\cite{Kaur2021, Lin19Quantum, Ghorai19Quantum}, and computational neuroscience~\cite{kostal_information_2013}.}

To address this issue, \cite{WV10} studied the Gaussian channel capacity under input cardinality constraints (finite-constellation capacity), in particular, the rate of convergence to the Gaussian channel capacity when the cardinality grows. {Define   
\[
    C_m\triangleq \max_{P_X \in \calP_m: \Expect[X^2]\leq \sigma^2} I(X;X+Z)
\]
as the capacity subject to an input cardinality $m$. 
It turns out that the capacity gap is precisely characterized by $\appr$ under the KL divergence for the unconstrained capacity-achieving input distribution (Gaussian):
\begin{equation}
\appr(m,N(0,\sigma^2),\KL) = C-C_m.
\label{eq:Cm}
\end{equation}
Consequently, the rate of the capacity gap and the corresponding constellation design are direct applications of our main results.
We construct a rate-optimal approximation using a refined moment-matching approach, which serves as a capacity-achieving constellation scheme for Gaussian channel inputs. Notably, we apply a \textit{local} moment matching scheme that apply moment matching to the conditional distribution of Gaussian distribution on a sequence of intervals, which provably improves the global moment matching scheme in \cite{WV10}.
(For further discussions and numerical comparison, see Section~\ref{sec:constellation}.)
Furthermore, to understand the fundamental limit of the capacity gap, we establish a tight lower bound through a novel analytical framework in Section~\ref{sec:LB}. Our analysis reduces the finite mixture approximation problem to a low-rank matrix approximation problem, which is resolved using spectral analysis. Notably, this lower bound corrects the previous result in \cite[Theorem 8]{WV10}.
}

The problem of approximation by finite mixtures also naturally arises in nonparametric statistics and empirical process theory. Classical results show that the complexity of a class of distributions, as manifested by their 
metric entropy, plays a crucial role in determining the rate of convergence of nonparametric density estimation \cite{YangBarron99,Vaart1996WeakCA}. If the distribution family is parametric, its entropy is often determined by the dimension of the parameter space. However, nonparametric families are infinite-dimensional and determining its entropy entails more delicate analysis including finite-dimensional approximation, which, for the Gaussian mixture family, requires approximation by finite mixtures. 
To describe the most economical approximation by finite mixtures, let us define
\begin{equation*}
  \comp (\epsilon,P,d)\triangleq \min\{m\in\naturals: \exists P_m\in \calP_m, d(f_{P_m},f_P) \leq  \epsilon\},
\end{equation*}
\ie, the smallest order of a finite mixture that approximates a given mixture $f_P$ within a prescribed accuracy $\epsilon$. 
For uniform approximation over $\calP$, define
\begin{equation*}
  \comp (\epsilon,\calP,d) \triangleq \sup_{P\in\calP} \comp (\epsilon,P,d),
\end{equation*}
which offers a meaningful \textit{complexity measure} for the class $\{f_P:P\in\calP\}$ and is closely related to more classical complexity notions such as the metric entropy. In fact, most of the existing constructive bounds on the metric or bracketing entropy for general Gaussian mixtures are obtained by first finding a discrete approximation then quantizing the weights and atoms, and the resulting upper bounds increase with $\comp$ \cite{GV07,GV01,Zhang08,Saha19}.
Hence, tightened upper bounds for metric entropy immediately follow.  
See Section~\ref{sec:mle-rate} for a detailed discussion.

Clearly, determining $\comp$ and that of $\appr$ are equivalent by the formula 
\begin{equation*}
\comp(\epsilon,\calP,d)=\inf\sth{m:\appr(m,\calP,d)\leq \epsilon}.
\end{equation*}
In the following, we state our main results in terms of $\comp$, which follow from bounds on $\appr$ in Sections~\ref{sec:UB} and~\ref{sec:LB}.

\subsection{Main results}
Our main results give non-asymptotic rates of $\comp$.
First, we explore the family of compactly supported distributions and provide a tight convergence rate in Theorem~\ref{thm:main-bdd}, which improves upon previous results \cite{GV01, Zhang08}. Next, we turn our attention to distribution families with exponential tail decay (\eg, subgaussian and subexponential families) in Theorem~\ref{thm:main-tail}.
In both regimes, we give  tight lower and upper bounds. 
An overview of the proof strategy is shown at the end of this subsection.

To begin with, consider the
distributions supported on $[-M,M]$ for $M>0$: 
\begin{equation*}
   \Pbdd{M}\triangleq \{P:P[-M,M]=1\}.
\end{equation*}
We consider a loss function $d$ satisfying the following inequality. 
This assumption encompasses the $\TV^2,H^2,\KL,\chi^2$ divergences as special cases by the inequalities between $f$-divergences (see Appendix~\ref{app:f-div}). 

\begin{assumption}
\label{as:f-div}
  There exists absolute constants  $c, c'>0$  such that 
$$c \TV^2\left(P, Q\right) \leq d\left(P \| Q\right) \leq {c' \chi^{2}\left(P \| Q\right)},\quad \forall P,Q.$$
\end{assumption} 

The following optimal rate of complexity level is established:
\begin{theorem}
\label{thm:main-bdd}
  Suppose $M\leq \epsilon^{-c_1}$ for some universal constant $0<c_1<\frac{1}{2}$.
  Then, for 
  $\epsilon\in(0,\frac{1}{2}]$ and $d$ satisfying Assumption~\ref{as:f-div},
\begin{equation}
    \label{eq:main01}
\comp(\epsilon,\Pbdd{M},d)\asymp
\frac{\log\frac{1}{\epsilon}}{\log\left(
1+ \frac{1}{M} \sqrt{\log \frac{1}{\epsilon}}\right)}\vee 1.
\footnote{For $x, y \in \reals$, $x \vee y \triangleq \max\{x, y\}$ and 
$x \wedge y \triangleq \min\{x, y\}$. For two positive sequences ${a_n}$ and ${b_n}$, write $a_n \lesssim b_n$ or $a_n = O(b_n)$ when $a_n \leq Cb_n$ for some absolute constant $C > 0$, $a_n \gtrsim b_n$ or $a_n = \Omega(b_n)$ if $b_n \lesssim a_n$, and $a_n \asymp b_n$ or $a_n =\Theta (b_n)$ if both $b_n \gtrsim a_n$ and $a_n \gtrsim b_n$ hold. 
We write $a_n=O_\alpha(b_n)$ and $a_n \lesssim_\alpha b_n$ if $C$ may depend on parameter $\alpha$.}
\end{equation}
\end{theorem}

We provide some interpretations of Theorem~\ref{thm:main-bdd}. 
By definition, $\comp$ increases with $M$ and decreases with $\epsilon$.
In fact, \eqref{eq:main01} captures an ``elbow-effect'' depending on the relationship between $M$ and $\epsilon$. 
If the support of mixing distributions is not too wide, \eg, $M\lesssim \pth{\log \frac{1}{\epsilon}}^{\frac{1}{2}-\delta}$ for a constant $\delta>0$, 
the complexity $\comp$ has a slower growth with respect to $\epsilon$ as $\frac{\log\epsilon^{-1}}{\log\log\epsilon^{-1}}$.
When $\sqrt{\log \frac{1}{\epsilon}}\lesssim M\lesssim \epsilon^{-c_1}$,
the finite mixture needs to cover a wider range, and consequently, $\comp$ grows faster as $M\sqrt{\log \frac{1}{\epsilon}}$. 
Finally, when 
$M\gtrsim \epsilon^{-c_1}$ for any $c_1\in(0,\frac{1}{2})$, our Theorem~\ref{thm:lb-unif} shows that we need at least $\Omega(M)$
 components to encompass an extensive range.

Next, we turn to the families of distributions supported on the whole real line. 
The following notions are introduced to characterize the tail conditions.
Given a function $\psi:[0,\infty)\mapsto [0,\infty)$
that is non-decreasing and
satisfies that $\psi(0)=0$ and 
{$\lim_{x\to\infty} \psi(x)=\infty$}, define
\begin{equation*}
    \|X\|_{\psi}\triangleq\inf \sth{t>0: \mathbb{E} \psi\pth{\frac{|X|}{t}} \leq 1}.
\end{equation*} 
This concept is commonly referred to as the $\psi$-Orlicz norm, originally introduced in  \cite{Orlicz1961}. It also implies the following tail bound \cite{Kuchibhotla22}:
\begin{equation}
    \label{eq:orlicz_tail}
    \mathbb{P}[|X| \geq t] \leq \frac{2}{\psi\left(t /\|X\|_{\psi}\right)+1},\quad \forall t \geq 0.
\end{equation} 
We focus on  $\psi(x)=e^{x^\alpha}-1$ for  $\alpha>0$. Define 
\begin{equation}
\calP_\alpha(\beta)\triangleq\sth{P : X\sim P, \|X\|_\psi \leq \beta,  \psi(x)=e^{x^\alpha}-1},
    \label{eq:Pab}
\end{equation}
    which is also referred to as sub-Weibull family (see  \cite{SubWeibull2020,SubWeibull2022,Kuchibhotla22} for various equivalent definitions). 
    For instance, the special cases of $\alpha=2$ and $\alpha=1$
    correspond to the
    $\beta^2$-subgaussian and $\beta$-subexponential families, respectively \cite{HDP}.  
    The tail probability bound \eqref{eq:orlicz_tail} becomes \cite[Eq. (2.3)]{Kuchibhotla22}
    \begin{equation}
        \label{eq:subweibull-tail}
        P[|X| \geq t]\leq 2\exp\qth{-\pth{\frac{t}{\beta}}^\alpha},\quad P\in \calP_\alpha(\beta).
    \end{equation}
We assume that $\psi$ (and thus $\alpha$ therein) is fixed, while  $\beta$ could vary with $m$.
\begin{theorem}
    \label{thm:main-tail}
    Suppose that $\beta \leq \epsilon^{-c_1}$ for some universal constant $0<c_1<\frac{1}{2}$. Then, for $\epsilon\in(0,\frac{1}{2}]$ and {$d$ satisfying Assumption~\ref{as:f-div}},
    \begin{equation*}
        \beta\pth{\log\frac{1}{\epsilon}}^{\frac{2+\alpha}{2\alpha}}\lesssim_\alpha\comp(\epsilon,\calP_\alpha(\beta),d)\lesssim_\alpha \frac{\log\frac{1}{\epsilon}}{\log\pth{1+\frac{1}{\beta}\pth{\log\frac{1}{\epsilon}}^{\frac{\alpha-2}{2\alpha}}}} \vee 1.
    \end{equation*}
    In particular,
    when $\pth{\log\frac{1}{\epsilon}}^{\frac{\alpha-2}{2\alpha}}\lesssim \beta \leq \epsilon^{-c_1}$, we have 
    \begin{equation*}
        \comp(\epsilon,\calP_\alpha(\beta),d) \asymp_\alpha \beta\pth{\log\frac{1}{\epsilon}}^{\frac{2+\alpha}{2\alpha}}.
    \end{equation*}
\end{theorem}
Theorem~\ref{thm:main-tail} shows that the complexity level grows in a logarithm rate, which equivalently implies an exponential convergence rate of the approximation error $\appr$.
The tight result is applicable when $\alpha\in(0,2]$ and $\beta$ is a fixed constant, such as in the case of the subgaussian and subexponential family. 
Specifically, for the $\sigma^2$-subgaussian family $\calP_2(\sigma)$ 
with $c_0\leq\sigma\leq \epsilon^{-c_1}$,
we have
    \begin{equation}
    \label{eq:comp_subG}
    \comp(\epsilon,\calP_2(\sigma),d)\asymp {\sigma\log\frac{1}{\epsilon}}.
    \end{equation}
{In these theorems, the conditions $M, \beta\leq \epsilon^{-c_1}$ are required only for the lower bound of $\comp$, and the upper bound continues to hold without this condition. For explicit details, see Theorems~\ref{thm:ub-bdd} and~\ref{thm:lb-unif} for the corresponding upper and lower bounds. }

One way to reconcile Theorems~\ref{thm:main-bdd} and~\ref{thm:main-tail}  is to notice that each distribution $P\in\calP_\alpha(\beta)$ is effectively supported 
(except for a total mass that is polynomially small in $\epsilon$) on the interval 
$[-C\beta \pth{\log \frac{1}{\epsilon}}^{\frac{1}{\alpha}},
C\beta \pth{\log \frac{1}{\epsilon}}^{\frac{1}{\alpha}}]$ 
for some large constant $C$, so that the approximation complexity of the class $\calP_\alpha(\beta)$ coincides with that of $\Pbdd{M}$ with 
$M \asymp \beta \pth{\log \frac{1}{\epsilon}}^{\frac{1}{\alpha}}$.
As will be shown next,  our upper bound essentially
pursues this idea, while the lower bound is based on a different approach.  

We now briefly discuss the proof strategies for the main results. We prove the upper bound under the $\chi^2$-divergence and the lower bound under the $\TV$ distance
{as $d$ satisfies Assumption~\ref{as:f-div}}. 
For the upper bound, we extend
the local moment matching argument in previous work \cite{Zhang08}, which constructs a discrete approximation by matching the moments for the mixing distribution conditioned on each subinterval in a partition of the effective support of the mixture. This approach can be further generalized to distribution families characterized by various tail conditions. The matching lower bound is the major contribution of this paper, where we relate the best approximation error to the \textit{low-rank approximation} of  \textit{trigonometric moment matrices}, and then conduct a refined analysis based on the classical spectrum theory. The application of orthogonal polynomials also plays a crucial role in our analysis. 

\begin{remark}[Gaussian location mixture with general variance] 
\label{remark:var}
Consider the problem of approximating $P*\phi_\sigma$ by an $m$-component mixture $P_m*\phi_\sigma$, where  $\phi_\sigma$ denotes the density of $N(0,\sigma^2)$. This problem appears in heteroscedastic settings of nonparametric density estimation \cite{Jiang20,SGS21}.
By the scale invariance of the $f$-divergences, we have $d(P_\sigma*\phi,Q_\sigma*\phi)=d(P*\phi_\sigma,Q*\phi_\sigma)$ for 
{$d\in\{\TV,H^2,\KL,\chi^2\}$}, where $P_\sigma$ denotes the distribution of $X/\sigma$ for $X\sim P$. Consequently, the minimum number of components is equal to $\comp (\epsilon,P_\sigma,d)$. 
\end{remark}

{

\begin{remark}
The loss function \( d \) may not be symmetric, as in the case with \( \KL \) and \( \chi^2 \) divergences. In general, the best approximation with respect to  \( d(f_{P_m}, f_P) \) and  \( d(f_P, f_{P_m}) \) are not equivalent. For instance, if $P=N(0,\sigma^2)$ with $\sigma>1$, then \( \chi^2(f_P \| f_{P_m}) = \infty \) for any $m$ and any $m$-atomic \( P_m \), but \( \chi^2( f_{P_m}\|f_P) \) can be made exponentially small.

In certain special cases, such as when \( P \) is compactly supported or sub-Gaussian, the rate of \( \KL(f_P \| f_{P_m}) \) can be derived using the divergence comparison inequalities for Gaussian mixtures~\cite{JPW23}. These inequalities provide upper bounds on the \( \KL \) divergence between Gaussian mixtures in terms of the symmetric squared Hellinger distance (\( H^2 \)), where our main results can be applied.  
\end{remark}  

}

\subsection{Comparison with previous results}
\label{sec:comparison}
Below we give an overview of previous results.
The upper bound for the compact support case is 
discussed in \cite{GV01,GV07,Zhang08}. Among them, the strongest result
\cite[Lemma 1]{Zhang08} gives an upper bound of
{$\comp(\epsilon,\Pbdd{M},\TV)\lesssim \pth{M\sqrt{\log{\frac{1}{\epsilon}}}}\vee \log{\frac{1}{\epsilon}}$. }
Theorem~\ref{thm:main-bdd} strengthens this result by bounding the $\chi^2$-divergence and establishing the optimal rate. 
Correspondingly, for general $\sigma$, it follows from Remark~\ref{remark:var} that
 the {$\pth{\frac{M}{\sigma}\sqrt{\log{\frac{1}{\epsilon}}}}\vee \log{\frac{1}{\epsilon}}$} upper bound from \cite[Lemma 3]{Jiang20} can be improved to {$\frac{\log{\frac{1}{\epsilon}}}{\log\left(1+ \frac{\sigma}{M} \sqrt{\log{\frac{1}{\epsilon}}}\right)}\vee 1.$}
 In addition, for Hellinger distance, \cite{pw20} conjectures that $\comp(\epsilon,\Pbdd{1},H)\asymp \frac{\log \frac{1}{\epsilon}}{\log \log \frac{1}{\epsilon}}$, which is proved by our Theorem~\ref{thm:main-bdd}.
 
 For the subgaussian case, \cite[Lemma 7]{pw20} gives a $\log{\frac{1}{\epsilon}}$ upper bound 
for 1--subgaussian family, which is further extended to sub-Weibull families by Theorem~\ref{thm:main-tail}.
The specific problem of approximating $P=N(0,\sigma^2)$ is studied in \cite{Wu2010FunctionalPO,WV10} in the context of the finite-constellation capacity.
While quantized Gaussian only achieves an error that is polynomial in $m$, an exponential upper bound is 
$\appr(m,N(0,\sigma^2),\KL) \lesssim \sigma^2 (\frac{\sigma^2}{1+\sigma^2})^{2m}$ is shown in \cite[Theorem 8]{WV10} using  the Gauss quadrature. As a corollary of our main result, \eqref{eq:comp_subG} generalizes this result to subgaussian family with an improved exponent for large $\sigma$.

Compared with these constructive upper bounds, the lower bound is far less understood. For the Gaussian distributions, 
\cite[Eq.\ 66]{WV10} claims that 
$\appr(m,N(0,\sigma^2),\KL) \geq (\frac{\sigma^2}{2+\sigma^2}+o(1))^{2m}$;
however, the sketched proof turns out to be flawed, which results in a wrong dependency of the exponent on large $\sigma$. This is now corrected in Theorem~\ref{thm:main-tail} (see also Theorems \ref{thm:ub--subW} and \ref{thm:lb-tail-1}), which shows that the exponential convergence is indeed tight but the optimal exponent (for fixed $\sigma$) is $\Theta(\frac{1}{\sigma})$ as opposed to 
$\Theta(\frac{1}{\sigma^2})$.
The exact optimal exponent, as a function of $\sigma$, however, remains open.

\subsection{Related work}
The problem of approximation 
by location mixtures is first addressed by the celebrated Tauberian theorem of Wiener \cite{Wiener32}, which gives a general characterization of whether the translation family 
of a given function is dense in $L_1(\reals^d)$ or $L_2(\reals^d)$ in terms of its Fourier transform. 
Convergence rates have been studied over the past few decades,
with a wide range of applications in 
approximation theory, machine learning, and information theory \cite{Barron1993,ferreira1997neural,cuesta2002shape,WV10,Wu2010FunctionalPO}.
For example, for the location and scale $m$-mixture class of sigmoidal functions, Barron \cite{Barron1993} obtained dimension-free
convergence rate for approximating functions whose gradient has integrable Fourier transform, a fundamental result in the theory of neural networks.
For Gaussian models, Wu and Verd\'u \cite{WV10} linked this problem to the Gaussian channel capacity under input cardinality constraint (cf.~\eqref{eq:Cm}). 
More recently, \cite{HDNguyen2019,TTNguyen2020,TTNguyen2021} showed the consistent approximation over various families with general location-scale mixtures.
{Another related problem is the Gaussian mixture reduction, which requires approximating a high-order Gaussian mixture with a low-order one. This problem broadly arises in applications including belief propagation  and Bayesian filtering. Although there are many numerical algorithms by means of clustering, optimization, or the greedy algorithm, convergence rates and  optimal approximators are still left to be discovered. See \cite{Zhang2020GaussianMR,Sajedi2023ANP,O23GMR} for some recent works and \cite{lookgmr11} for a review.}

In the statistics literature, understanding the complexity of a distribution class plays an important role in nonparametric density and functional estimation \cite{Vaart1996WeakCA}. 
Information-theoretic risk bounds are obtained on the basis of metric entropy for a variety of loss functions (see \cite[Chapter 32]{PW-it} for an overview). 
In addition, metric entropy of the Gaussian mixture family is crucial in the statistical analysis of 
 the sieve and  nonparametric maximum likelihood estimator (MLE) in  mixture models as well as posterior concentration \cite{SW94,WS95,Genovese00, GV01,GV07,Zhang08,Saha19,Jiang20,SGS21}. These results all rely on metric entropy of the Gaussian mixture class 
 obtained via approximation by finite mixtures. {In this vein,  the quantity $m^\star(n^{-1/2},\calP,d)$ is referred to as the \textit{statistical degree} of $\calP$ \cite{pw20}, which represents the smallest $m$ so that any
density in $\calP$ can be made statistically indistinguishable  (on the basis of $n$ observations) from some
density in $\calP_m$.}

A related quantity to finite mixture approximation, known as \textit{smoothed empirical distribution}, has been studied in recent literature \cite{GGNWP20,CNW22,BJPR22}. Given a sample \((X_1, \ldots, X_m)\iiddistr P\), we approximate  \(P * \phi\) by \(P_m * \phi\), where \(P_m=\frac{1}{m}\sum_{i=1}^m\delta_{X_i}\) is the empirical measure. For mixing distributions with bounded support or subgaussian tails, the convergence rate of this approximation is shown to be dimension-free and polynomial in \(m\), measured by different Wasserstein distances and \(f\)-divergences (\eg, $O(m^{-1/2})$ under $\TV$ and $O(m^{-1})$ under $\chi^2$). 
Note that $P_m$ is a special $m$-atomic distribution. 
As noted in 
\cite[Sec.~VII]{WV10}, the achieved approximation error is consistent with CLT-type of rates which are markedly slower than the optimal rates that are typically exponentially.

\subsection{Organization}
The rest of this paper is organized as follows.
In Section~\ref{sec:prelim}, we provide the background on trigonometric moment matrices and orthogonal polynomials on the real line and on the unit circle. Section \ref{sec:UB} and \ref{sec:LB} contain  upper and lower bounds on the best approximation error  $\appr$ and present the main proof ideas. 
Section~\ref{sec:discussion} discusses  applications to  the convergence rate of nonparametric maximum likelihood estimator and the extensions to other classes of mixing distributions and mixture models.
Additional backgrounds and proof details are presented in the Appendices.

\subsection{Notations}
Let $[k]= \{1, \dots , k\}$ for $k \in \naturals$.
 Denote by $\delta_{jk}=\indc{j=k}$  the Kronecker's
delta notation. We use bold symbols to represent vectors and matrices. For a vector $\bfx$, denote $\bfx^\top$ and $\bfx^\star$ as the transpose and the Hermitian transpose, respectively, and $\operatorname{Diag}(\bfx)$ as the corresponding diagonal matrix. Denote $\|\cdot\|$ as the Euclidean norm for vectors or spectral norm for matrices, and let $\|\cdot\|_F$ be the Frobenius norm. Write $\lambda_{\min}(\bfA)$ as the smallest eigenvalue of a Hermitian matrix $\bfA$.

\section{Preliminaries}
\label{sec:prelim}

\subsection{Trigonometric moment matrices}
\label{sec:momentmatrix}
The theory of moments is fundamental in many areas of  probability, statistics, and approximation theory \cite{usp37}. Given a distribution $P$ and $X\sim P$,
denote its $k\Th$ moment by $m_k=m_k(P)=m_k(X)=\Expect_P[X^k]$. 
The moment matrix associated with $P$ of order $n+1$ is the following \textit{Hankel} matrix:
\begin{equation}    
\label{eq:eq2-M}
\bfM_{n}(X)=\left[\begin{array}{cccc}
m_0 & m_{1} & \cdots & m_{n} \\
m_{1} & m_{2} & \cdots & m_{n+1} \\
\vdots & \vdots & \ddots & \vdots \\
m_{n} & m_{n+1} & \cdots & m_{2n}
\end{array}\right]_{(n+1)\times (n+1)}.
\end{equation}
Denote the vector of monomials as $\bfX_{n}=(1,X,\ldots,X^n)^\top$. The moment matrix of $P$ can be equivalently represented as 
$\bfM_{n}(X)=\Expect_P[\bfX_{n}\bfX_{n}^\top]$. Consequently, if $P$ is discrete with no more than $m$ atoms, the moment matrix of any order is of rank at most $m$, and $P$ can be uniquely determined by its first $2m-1$ moments \cite{usp37}. 

The above formulation can also be adapted to the trigonometric moments. 
For $k\in\integers$, denote $t_k=t_k(P)=t_k(X)=\Expect_P[e^{ikX}]$ as the $k\Th$ order Fourier coefficients (or characteristic functions) of $P$. Define the \textit{Toeplitz} matrix 
\begin{equation}
\label{eq:eq2-T}
\bfT_{n}(X)=\left[\begin{array}{cccc}
t_0 & t_{1} & \cdots & t_{n} \\
t_{-1} & t_0 & \cdots & t_{n-1} \\
\vdots & \vdots & \ddots & \vdots \\
t_{-n} & t_{-(n-1)} & \cdots & t_0
\end{array}\right]_{(n+1)\times (n+1)}
\end{equation}
as the trigonometric moment matrix associated with $P$ of order $n+1$.  
$\bfT_{n}(X)$ is equivalently the ordinary moment matrix of $Z=e^{iX}$ in the sense that $\bfT_{n}(X)=\Expect_P[\bfZ_{n}\bfZ_{n}^\star]$ for 
$\bfZ_{n}=(1,Z,\ldots,Z^n)^\top$.
Note that both $\bfM_n$ and $\bfT_n$ are positive semidefinite matrices.

Our proof of the converse results relies on classical theory of moment matrices. For Hankel moment matrices, the seminal work \cite{Szeg1936OnSH} studied the asymptotic behavior of the eigenvalues for Gaussian and exponential weights.   Systematical treatments for general distribution classes are given by a series of work \cite{WidomWilf66,Chen1999SmallEO,Berg1999SmallEO,Chen2004SmallEO}. \cite{STAMPACH2019483} proposes a generalized result for certain forms of weighted Hankel matrices. 
\cite{szego1953} gives characterization for eigenvalues of Toeplitz forms, which applies in particular to the trigonometric moment matrices. 

The moment matrices are Hermitian and positive definite, provided that the corresponding distribution has infinite support \cite{Szeg1936OnSH}. It turns out that the smallest eigenvalue of the moment matrix plays an important role in the derivation of the lower bound. \cite{PASUPATHY1992} gives bounds on eigenvalues for Gaussian Toeplitz matrices, but the lower bound is suboptimal. \cite{Chen1999SmallEO} introduces a framework of bounding the smallest eigenvalue, 
extending the method of \cite{Szeg1936OnSH}. Specifically, the ordinary moment matrix \eqref{eq:eq2-M} is related to the orthogonal polynomials on the real line, and the trigonometric moment matrix \eqref{eq:eq2-T}  is related to the orthogonal polynomials on the unit
circle (see Section~\ref{sec:LB} for further discussions). 

\subsection{Orthogonal polynomials}
\label{sec:ortho}
  Orthogonal polynomials on the real line and on the unit circle serve as useful tools in our proofs. 
 Given a weight function  $w:\reals\mapsto [0,\infty)$,  
 denote by $\{p_n(x)\}$ the set of monic (leading coefficient equal to one) orthogonal polynomials on the real line associated with the weight $w(x)$, such that the degree of $p_n(x)$ is 
 $n$ and
\begin{equation}
    \label{eq:oprl}
    \int p_j(x)p_k(x) w(x) \diff x = h_j \delta_{jk}, \quad \forall j,k\geq 0.
\end{equation}
When $w(x)$ is the density of a probability distribution $P$,  $\{p_n(x)\}$ is said to be the set of orthogonal polynomials associated with $P$.

For a density function $f(\theta)$ supported on a subset of $\reals$, 
let $\{\varPhi_n(z)\}$
be the associated set of monic orthogonal polynomials on the unit circle, that is, $\varPhi_n$ is of degree $n$, and 
\begin{equation}
    \label{eq:opuc}
    \int \varPhi_j(e^{i\theta})\overline{\varPhi_k(e^{i\theta})} f(\theta) \diff\theta =\kappa_j^{-2} \delta_{jk}, \quad \forall j,k\geq 0.
\end{equation}

The formula of Szeg\"o \cite[Theorem 11.5]{SzegoOrthopolys}
shows the relation between orthogonal polynomials on the unit circle and on the real line, provided that the weight functions of the two orthogonal systems are related. This result is further developed in the follow-up work \cite{ZHEDANOV1998,Krasovsky2003}. See Section~\ref{sec:lb-unif} for a detailed exposition and applications.

In Appendix~\ref{app:ortho}, we describe several concrete examples of orthogonal polynomials associated with certain distributions that will be used in our proof. We refer the readers to \cite{SzegoOrthopolys} for a comprehensive review of orthogonal polynomials and \cite{BarrySimon2005a} for a special treatment of orthogonal polynomials on the unit circle.

\section{Achievability via moment matching}
\label{sec:UB}
In this section, we give upper bounds of finite mixture approximation of the distribution family $\calP$. For any $P\in\calP$, we need to construct an $m$-atomic distribution $P_m$ achieving a small approximation error measured by $d(f_{P_m},f_P)$. 
Previous results have shown that comparing moments is useful in determining the approximation accuracy. The Gauss quadrature, which is briefly explained in Appendix~\ref{app:GQ}, serves as a classical and effective approach to the global discrete approximation in the sense of matching moments.
For example, \cite[Theorem 8]{WV10} considers the $m$-point Gaussian quadrature with a scale parameter $\sigma$ that matches the first $2m-1$ moments of $N(0,\sigma^2)$. For compactly supported distributions, \cite[Lemma 1]{Zhang08} provides another moment matching approximation.
 
In our non-asymptotic analysis framework, when the relative scale of the parameter of the distribution family compared with $m$ changes, different treatments are needed. 
Consider the family of compactly supported distributions as an example. The previous result \cite{Zhang08} becomes suboptimal unless $M$ grows with $m$ at a certain rate, as shown in \eqref{eq:ub-bdd}. This problem also arises for distribution families with tail conditions. 

Aiming at tight upper bounds, we carry out a refined analysis based on the technique of moment matching. The following is a road map for this section. We start from the compactly supported case in Section~\ref{sec:ub-bdd}, 
where we bound the $\chi^2$-divergence by moment differences and construct moment matching approximations both globally and locally depending on the relationship between parameters. 
Then, in Section~\ref{sec:ub-tail}, convergence rates for distribution families with tail conditions are obtained via an extra truncation argument. We derive rate upper bounds for the $\calP_\alpha(\beta)$ 
 family, and the proof idea also works for general choices of $\psi$ by plugging in the tail conditions (see Section~\ref{sec:extension-moment} for an extension).

\subsection{Compactly supported distributions}
\label{sec:ub-bdd}
The following lemma is useful for upper bounding the  approximation error in $\chi^2$-divergence by comparing moments.
\begin{lemma}
    \label{lem:mm1}
    Suppose that $P,Q\in\Pbdd{M}$. For any $J>4M^2$, if $m_j(P)=m_j(Q)$ for $j\in[J-1]$,
\begin{equation*}
  \chi^{2}\left(f_P \| f_Q\right) \leq 4\exp\pth{\frac{M^2}{2}}\pth{\frac{4eM^2}{J}}^{J}.
\end{equation*}
\end{lemma}

\begin{proof}
Denote $\mu=\Expect[Q]$ and  $\sigma^2=\var[Q]$. Let $U\sim P$, $V\sim Q$, $U-\mu\sim {P}^\prime_\mu$, and $V-\mu\sim {Q}^\prime_\mu$. Since $m_j(P)=m_j(Q)$ for $j\in[J-1]$, it holds that $m_j({P}^\prime_\mu)=m_j({Q}^\prime_\mu)$ for $j\in[J-1]$. 
Applying \cite[Lemma 9]{WY19} with $\Expect_{{Q}^\prime_\mu}[X]=0$, we have 
$$
\chi^{2}\left(f_{P} \| f_{Q}\right)
= \chi^{2}\left(f_{{P}^\prime_\mu} \| f_{{Q}^\prime_\mu}\right) \leq e^{\frac{\sigma^2}{2}} \sum_{j \geq J} 
\frac{\left(m_j({P}^\prime_\mu)-m_j({Q}^\prime_\mu)\right)^{2}}{j !}.$$
Since ${P}^\prime_\mu,{Q}^\prime_{\mu}\in\Pbdd{2M}$, we have $|m_j({P}^\prime_\mu)|, |m_j({Q}^\prime_\mu)|\le (2M)^j$.
Since $\sigma^2\le M^2$, it follows that
\begin{align*}
    \chi^{2}\left(f_{P} \| f_{Q}\right)
\le \exp\pth{\frac{M^2}{2}} \sum_{j \geq J} 
\frac{[2\left(2M\right)^j]^{2}}{j !}
\le 4\exp\pth{\frac{M^2}{2}}\pth{\frac{4eM^2}{J}}^{J},
\end{align*}
where the last inequality follows from 
{the Chernoff bound for Poisson distribution $\pbb[X\geq J]\leq e^{-4M^2}\left(\frac{4e M^2}{J}\right)^{J}$ for $X\sim\operatorname{Poisson}(4M^2)$ and $J>4M^2$ (see, \eg, \cite[Theorem 4.4]{mitzenmacher2005probability})}.
\end{proof}

Now we are ready to show the $\chi^2$ approximation error bound for compactly supported mixing distributions. 
When $M$ is small, we construct a discrete distribution that matches the first few moments.
However, this approach is loose for large $M$ and can be remedied using the idea of \textit{local approximation}. Specifically,  we partition the support into subintervals and approximate each conditional distribution via moment matching; similar construction of local moment matching has appeared in the statistics literature \cite{Zhang08,SW22} by applying Carathe\'odory's theorem to the truncated moments. 
{This approach determines the allocation of atoms across subintervals, and tight upper bounds then follow from the trade-off between the number of atoms per interval and the interval size.} 
The overall approximation error can be bounded combining that in each subinterval.  

\begin{theorem}
  \label{thm:ub-bdd}
  There exists a universal constant $\kappa>0$ such that for any $m\in \naturals$ and $M>0$:
  \begin{equation}
    \label{eq:ub-bdd}
\appr(m,\Pbdd{M},\chi^2)\leq
  \begin{cases}
      \exp\left(-m\log\frac{m}{M^2}\right), & m\geq \kappa M^2;\\
      \exp\pth{-\frac{\log\kappa }{4\kappa}\frac{m^2}{M^2} }, & 3\sqrt{\kappa}M \leq m\leq  \kappa M^2.\\
  \end{cases}
  \end{equation}
  \end{theorem}
\begin{proof}
Suppose that $m\geq \kappa M^2$ for some universal constant $\kappa\geq 16e^3$.
For $P\in\Pbdd{M}$, by the Gauss quadrature rule (see Appendix~\ref{app:GQ}), there exists $P_m\in\Pbdd{M}\cap \calP_m$ that matches the first $2m-1$ moments of $P$. By Lemma~\ref{lem:mm1}, we have
\begin{align}
      \chi^{2}\left(f_{P_m} \| f_{P}\right)
      &\le 4\exp\left(-2m\log\frac{m}{M^2}+\frac{M^2}{2}
      +\log(4e^2)m\right)\nonumber\\
      &\leq \exp\left(-m\log\frac{m}{M^2}\right).
       \label{eq:ub-bdd-1}
    \end{align}

Suppose that $3\sqrt{\kappa}M\leq m\leq \kappa M^2$ holds. Set $K=\floor{\frac{3\kappa M^2}{m}}\geq 3$.
We partition the interval $[-M,M]$ into $K$ subintervals $I_j=[-M+(j-1)\frac{2M}{K},-M+j\frac{2M}{K}]$ for $j\in[K]$. Let $P_{(j)}$ be the conditional distribution of $P$ on $I_j$, that is, for any Borel set $A$, $P_{(j)}(A)=P(A|I_j).$ Denote  $\tilde{m}=\floor{m/K}$.
Note that the condition $\tilde{m} \ge \kappa \pth{\frac{M}{K}}^2$ holds by $$K^2\floor{\frac{m}{K}}\geq 
\floor{\frac{3\kappa M^2}{m}}^2
\floor{\frac{m^2}{3\kappa M^2}} \geq \pth{\frac{3}{4}\frac{3\kappa M^2}{m}}^2
\pth{\frac{3}{4}\frac{m^2}{3\kappa M^2}} 
\geq 
\kappa M^2,$$ 
where the second inequality follows from the facts that 
$\floor{x}\geq \frac{c}{c+1} x$ for all $x\geq c, c\in\naturals$, and $\min\sth{\frac{m^2}{3\kappa M^2},\frac{3\kappa M^2}{m}}\geq 3$.
Applying \eqref{eq:ub-bdd-1} to each $P_{(j)}$ with a translation, there exists $\tilde P_{j}$ supported on at most $\tilde m$ atoms such that
\begin{align*}
       \chi^2(f_{\tilde P_{j}}\|f_{P_{(j)}})
\leq \exp\left(-\floor{\frac{m}{K}}\log\frac{\floor{\frac{m}{K}}}{\left(\frac{M}{K}\right)^2}\right)
\leq \exp\left(-\frac{3}{4}\frac{m^2}{3\kappa M^2}\log\kappa\right)
= \exp\left(-\frac{\log\kappa}{4\kappa}\frac{m^2}{ M^2}\right).
\end{align*}

Define $P_m = \sum_{j=1}^K  P(I_j) \tilde P_{j}$ supported on at most $m$ atoms.
Since $P=\sum_{j=1}^K P(I_j) P_{(j)}$, by Jensen's inequality and the convexity of $f$-divergences,
\begin{equation*}
    \chi^{2}\left(f_{P_{m}} \| f_{P}\right)
    \leq \sum_{j=1}^K P(I_j) \chi^2(f_{\tilde P_{j}}\|f_{P_{(j)}})
    \leq \exp\pth{-\frac{\log\kappa }{4\kappa}\frac{m^2}{M^2} }.
    \qedhere
\end{equation*}
\end{proof}

\subsection{Distribution families under tail conditions}
\label{sec:ub-tail}
In this subsection, we extend our analysis to the sub-Weibull family $\calP_\alpha(\beta)$ using a truncation argument. 
For each $P \in \calP_\alpha(\beta)$, we reduce the problem to
approximating the conditional distribution $P_t$ of $P$ on $[-t,t]$, where the approximate error is  bounded by Theorem~\ref{thm:ub-bdd} above.
The final result follows by optimizing $t$. 
Moreover, this approach is applicable for general distribution families by incorporating the corresponding tail probability bounds.

The following lemma upper bounds the $\chi^2$-divergence between two (general) mixtures by truncating one of the mixing distributions:
\begin{lemma}
\label{lem:chi2-truncate}
    Let  $A$ be a Borel set such that $P(A),P(A^c)>0$.    
    Denote by $P_A$ the conditional distribution of $P$ on $A$, 
    i.e.,    $P_A(\cdot)=P(\cdot \cap A)/P(A)$.
    For any distribution $Q$, we have 
    \begin{equation*}
        \chi^2(f_{Q}\|f_{P})\leq \frac{2}{P(A)}\pth{\chi^2(f_{Q}\|f_{P_A})+P(A^c)}.
    \end{equation*}
\end{lemma}
\begin{proof}
Define $P_{A^c}$ as $P$ conditioned on the complement $A^c$.
By linearity, the mixture can be decomposed as $f_P=P(A)f_{P_A} + P(A^c)f_{P_{A^c}}$. Then
\begin{align*}
    \chi^2(f_{Q}\|f_{P})
    &= \int \frac{(f_{Q}-P(A)f_{P_A} - P(A^c)f_{P_{A^c}})^2}{f_P}  \\
    &\leq 2\int \frac{(f_{Q}-f_{P_A})^2+ [P(A^c)(f_{P_A} -f_{P_{A^c}})]^2}{f_P}  \\
    &\leq \frac{2\chi^2(f_{Q}\|f_{P_A})}{P(A)}+2P(A^c)^2\int \frac{ (f_{P_A} -f_{P_{A^c}})^2}{f_P} .
\end{align*}
Since $f_P\ge P(A) f_{P_A}$ and $f_P\ge P(A^c) f_{P_{A^c}}$, it follows that
\begin{equation*}
    \int \frac{ (f_{P_A} -f_{P_{A^c}})^2}{f_P}  \leq
    \int \frac{ f_{P_A}^2 +f_{P_{A^c}}^2}{f_P} 
    \leq \frac{1}{P(A)}+\frac{1}{P(A^c)} =
    \frac{1}{P(A)P(A^c)}.
\end{equation*}
Consequently,
\begin{align*}
        \chi^2(f_{Q}\|f_{P})
        &\leq \frac{2}{P(A)}\pth{\chi^2(f_{Q}\|f_{P_A})+P(A^c)}. \qedhere
\end{align*}
\end{proof}
The next proposition directly follows from Lemma~\ref{lem:chi2-truncate} with $A=I_t\triangleq [-t,t]$  and the definition of $\appr$. 
\begin{proposition}
\label{prop:ub-tail}
For any distribution family $\calP$, with the same notation as in Lemma~\ref{lem:chi2-truncate}, we have
\begin{align*}
    \label{eq:prop-ub-tail-chi2}
    \appr(m,\calP,\chi^2)\leq & \inf_{t>0} \sup_{P\in\calP}
    \frac{2}{P(I_t)}\pth{\appr(m,\Pbdd{t},\chi^2)+P(I_t^c)}.
\end{align*}
\end{proposition}

Now we can derive upper bounds for distribution classes under tail conditions. In the following, we state the universal approximation rate for the family $\calP_\alpha(\beta)$. Extensions to other families are discussed in Section~\ref{sec:extension-moment}. 
\begin{theorem}
    \label{thm:ub--subW}
    There are constants $c_\alpha,C_\alpha$, such that for all 
    $m\geq C_\alpha \beta$,     
    \begin{equation}
        \label{eq:ub--subW}
        \appr(m,\calP_\alpha(\beta),\chi^2) \leq 
        \exp\sth{-c_\alpha m\log \pth{1+\frac{m^{\frac{\alpha-2}{\alpha+2}}}{\beta^{\frac{2\alpha}{\alpha+2}}}}}.
    \end{equation}
\end{theorem}
\begin{proof}
Denote $\calP=\calP_\alpha(\beta)$. 
By Theorem~\ref{thm:ub-bdd},
when $m\geq 3\sqrt{\kappa}t$, we have that
\begin{align}
        \appr(m,\Pbdd{t},\chi^2)
    \le 
    \exp\qth{-\Theta_\alpha\pth{m\log\pth{1+\frac{m}{t^2}}}}.
    \label{eq:chi2-ub-ptbdd}
\end{align} 
Set $t =c_\alpha  \beta \qth{m\log \pth{1+\frac{m^{\frac{\alpha-2}{\alpha+2}}}{\beta^{\frac{2\alpha}{\alpha+2}}}}}^\frac{1}{\alpha}$ for some $c_\alpha$ only depending on $\alpha$.
From the assumption $m\geq C_\alpha \beta$ for sufficiently large $C_\alpha$, it can be verified that there exists $c_\alpha$ such that $m\geq 3\sqrt{\kappa}t$ and $t\geq \Omega_\alpha(\beta)$ hold. 
By the tail condition \eqref{eq:subweibull-tail}, we have  $\sup_{P\in\calP}
    P(I_t^c)\leq 2\exp(-(\frac{t}{\beta})^\alpha)$, and hence $P(I_t)=\Omega_\alpha(1)$.
    The desired result follows from applying Proposition~\ref{prop:ub-tail} with our choice of $t$ and the upper bounds \eqref{eq:chi2-ub-ptbdd} and \eqref{eq:subweibull-tail}. 
\end{proof}
As a special case, consider the family $\calP_2(\sigma)$ of all 
$\sigma^2$-subgaussian distributions. Then for some universal constant $C$,
  \begin{equation}
  \label{eq:ub-subg}
\appr(m,\calP_2(\sigma),\chi^2) \leq
    \exp\pth{-Cm\log\pth{1+\frac{1}{\sigma}} }.
  \end{equation}

In Theorems~\ref{thm:ub-bdd} and~\ref{thm:ub--subW}, we assume that the order of $m$ is at least proportional to the scaling parameter (\ie, $M$ and $\beta$,  respectively). Otherwise, if the scaling parameter of the mixing distribution gets larger,
 the distribution becomes too heavy-tailed to be well approximated by any $m$-order finite mixture. This intuition is made precise by the lower bound in Theorem~\ref{thm:inapprox}, which implies that consistent approximation is impossible unless $m$ grows proportionally to the scale parameters.

\section{Converse via spectrum of moment matrices}
\label{sec:LB}

In this section, we give lower bounds of the best approximation error by finite mixtures. 
In Section~\ref{sec:lb-general}, we first propose a general framework applicable to any finite mixture in terms of trigonometric moment matrices. This strategy is then applied to the class of sub-Weibull distributions $\calP_\alpha(\beta)$ and compactly supported distributions $\Pbdd{M}$ in Sections~\ref{sec:lb-tail}
 and \ref{sec:lb-unif}, respectively.
For each distribution family $\calP$ of interest, we exhibit a \textit{test distribution} $P\in\calP$ and obtain lower bound of $\appr(\epsilon,P,d)$ and hence also $\appr (m,\calP,d)$.
These test distributions are essentially the hardest to approximate in their respective distribution classes and are amenable to analysis.

Our results indicate that distributions with heavier tails are generally more difficult to approximate using finite mixtures. 
Specifically, the tail behavior itself directly yields a lower bound on the approximation error (see Corollary~\ref{cor:TV-eigen-f} for an explicit result).
In addition, without imposing constraints on the distribution family, no finite mixture of fixed order can achieve consistent approximations. 
To identify the fundamental limit, we consider the regime of large $\beta$ or $M$ and prove lower bound on $m$ such that consistent approximation is possible (see  Theorem~\ref{thm:inapprox} and \eqref{eq:lb-unif-inapprox} of Theorem~\ref{thm:lb-unif}). 

\subsection{General framework}
\label{sec:lb-general}
We state a general procedure to lower bound the approximation error $\appr (m,P,d)$, where $P$ is not finitely supported. At the high level, we relate the approximation error in $\TV$ to that of the trigonometric moments. 
As mentioned in Section \ref{sec:momentmatrix}, 
 the trigonometric moment matrix of $P$ of any degree is full-rank, while that of any $m$-atomic distribution has rank at most $m$. Thus, the question boils down to the smallest eigenvalue of the trigonometric moment matrix of $P$, which can be lower bounded  using classical results in spectral analysis. For this step, we provide two different methods, both of which are needed for proving the main results.

\subsubsection{Lower bound via trigonometric moments} 
The following lemma states a classical result of the best low-rank approximation, known as the Eckart-Young-Mirsky theorem (see, \eg, \cite{mirsky60}).
\begin{lemma}
Let $A\in\complex^{n\times n}$ be a Hermitian matrix with eigenvalues $\sigma_1\geq \ldots\geq\sigma_n\geq 0$. Then for any $k\in[n]$,
\begin{equation*}
    \min _{B\in {\complex}^{n\times n},\ \operatorname{rank}(B)=k}\|A-B\|_{F}=\sqrt{\sigma_{k+1}^{2}+\cdots+\sigma_{n}^{2}}.
\end{equation*}
    \label{lem:eym}
\end{lemma}
Recall that $\bfT_m(X)=(t_{k-j}(X))_{j,k=0}^m$ denotes the trigonometric moment matrix for a random variable $X$ of order $m+1$ as in \eqref{eq:eq2-T}.
The following proposition relates the approximation error in TV to the smallest eigenvalue of trigonometric moment matrices. 
\begin{proposition}
\label{prop:TV-eigen}
For any random variable $X\in\reals$ with distribution $P$, 
\begin{equation*}
    \appr(m,P,\TV)\geq  \sup_{\delta>0} \frac{\lambda_{\min}(\bfT_m(\delta X))}{2(m+1)\exp(m^2\delta^2/2)}.
\end{equation*} 
\end{proposition}
\begin{proof}
    Denote $Z\sim N(0,1)$ and $X_m\sim P_m$, where $P_m\in\calP_m$ is an arbitrary $m$-atomic distribution.      
Set the test function as $f_\omega(x)=\exp(i\omega x)$. 
By the variational representation formula in Lemma~\ref{lem:variation}, we obtain that
    \begin{align*}
    \TV\left(f_{P_{m}} , f_{P}\right)
&\geq \frac{1}{2} \sup _{\omega} \left|\mathbb{E}[f_\omega(X+Z)]-\mathbb{E}[f_\omega(X_m+Z)]\right|\\
&= \frac{1}{2} \sup _{\omega}  e^{-\frac{\omega^2}{2}}\left|\mathbb{E}[f_\omega(X)]-\mathbb{E}[f_\omega(X_m)]\right|.
\end{align*}
Set $\omega\in\delta\mathbb{Z}\triangleq\{k\delta: k\in\mathbb{Z}\}$, where $\delta>0$ is a frequency parameter to be determined. Then  
\begin{align*}
  \TV\left(f_{P_{m}}, f_{P}\right)
&\geq\max_{k\in\mathbb{Z},|k|\leq m}\frac{\left|\mathbb{E}[\exp(ik \delta X_{m})]-\mathbb{E}[\exp(ik\delta X)]\right|}{2\exp(k^2\delta^2/2)}
=\max_{|k|\leq m} \frac{\left|t_k(\delta X_m)-t_k(\delta X)\right|}{2\exp(k^2\delta^2/2)}.
\end{align*}
Note that $t_k(\delta X_m)$ and $t_k(\delta X)$ are entries of $(m+1)\times (m+1)$
Toeplitz matrices $\bfT_m(\delta X_m)$ and $\bfT_m(\delta X)$, respectively.
Since $\operatorname{rank}(\bfT_m(\delta X_m))\leq m$, applying Lemma~\ref{lem:eym} yields
\begin{equation*}
\TV\left(f_{P_{m}}, f_{P}\right)
\geq\frac{\|\bfT_m(\delta X_m)-\bfT_m(\delta X)\|_{F}}{2(m+1)\exp(m^2\delta^2/2)}
\geq  \frac{\lambda_{\min}(\bfT_m(\delta X))}{2(m+1)\exp(m^2\delta^2/2)}. \qedhere
\end{equation*}
\end{proof}

{
\begin{remark}
Proposition~\ref{prop:TV-eigen} is proved via lower bounding \( \TV(f_{P}, f_{P_m}) \) by the difference of trigonometric moments. A related result \cite[Theorem 4.2]{Doss20} lower bounds the divergence by moment differences as \( H^2(f_{P}, f_{Q}) \geq C_k \|\bfm_{2k-1}(P)-\bfm_{2k-1}(Q)\|_2^2 \) for any \( P, Q \in \Pbdd{1} \cap \calP_k \), where \( \bfm_{2k-1}(P) \triangleq (m_1(P), \dots, m_{2k-1}(P)) \) and \( C_k = \exp(-\Theta(k \log k)) \). However, this result require both $f_P$ and $f_Q$ to be finite mixtures. In contrast, the lower bound on \(\TV\) in Proposition \ref{prop:TV-eigen} applies to continuous mixture $f_P$.
\end{remark}

}
\subsubsection{Spectral analysis of trigonometric moment matrices}
\label{sec:spectral}
To lower bound the smallest eigenvalue of $\bfT_m(\delta X)$, we propose two parallel approaches, by applying the techniques of  \textit{wrapped density} and \textit{orthogonal expansion}, respectively.

\paragraph{Wrapped density.} 
    Given a density function $g$ on  $\reals$,  define the corresponding \textit{wrapped density} $g^{\wrap}$ on $[0,2\pi]$ as 
\begin{equation*}
    g^{\wrap}(\theta)\triangleq\begin{cases}
        \sum_{j=-\infty}^\infty g(\theta-2\pi j), & \theta\in[0,2\pi];\\
        0, &\text{otherwise}.
    \end{cases}
\end{equation*}
Thus, if $Y$ has density $g$, then $Y^{\wrap} \triangleq Y \mod 2 \pi \integers$ has density $g^{\wrap}$.  By definition, we have that $t_k(Y)=t_k(Y^{\wrap})$ for any $k\in\naturals$, and then $\bfT_m(Y^{\wrap})=\bfT_m(Y)$. Furthermore, the classical theory for the spectrum of Toeplitz matrices shows that the eigenvalues of $\bfT_m(Y^{\wrap})$ are bounded by the minimum and maximum of the wrapped density (cf.~e.g.~\cite[Eq.~(1.11)]{szego1953}), that is,
\begin{align}
    \lambda_{\min}(\bfT_m(Y))  = \lambda_{\min}(\bfT_m(Y^{\wrap})) 
    \geq \inf \{2\pi g^{\wrap}(\theta) : \theta\in [0,2\pi]\}
    \label{eq:eigen-lb-0}.
\end{align}
To see this, writing  $\bfT\equiv \bfT_m(Y^{\wrap})$, the smallest eigenvalue is expressed as the minimum  Rayleigh quotient 
   \begin{equation}
    \label{eq:rayleigh}
        \lambda_{\min}(\bfT)=\min_{\bfx\in\complex^{m+1}\setminus\{\mathbf{0}\}} \frac{ \bfx^\star \bfT \bfx}{\|\bfx\|^2},
    \end{equation}
For any $\bfx=(x_0,\ldots,x_m)^\top \in \complex^{m+1}$, let $\pi_m(w)=\sum_{j=0}^m x_jw^j$. 
Let $\kappa \triangleq \inf_{\theta\in[0,2\pi]} g^{\wrap}(\theta)$. Then
\begin{equation}
    \bfx^\star \bfT \bfx 
=  \Expect[|\pi_m(e^{i Y^{\wrap}})|^{2}] =  
\int_0^{2\pi} g^{\wrap}(\theta)  |\pi_m(e^{i \theta})|^{2} \diff\theta\geq \kappa \int_0^{2\pi}  |\pi_m(e^{i \theta})|^{2} \diff\theta  = 2\pi\kappa \cdot \bfx^\star  \bfx,
\label{eq:eigen-lb-0why}
\end{equation}
where the last step applies the 
orthogonality of the monomials on the unit circle $\int_0^{2\pi} d\theta e^{i(j-k) \theta} = 2\pi \delta_{jk}$.

{
As a corollary of Proposition~\ref{prop:TV-eigen}, we provide an explicit lower bound based on the tail behavior of the mixing distribution.
This result highlights that  distributions with heavier tails are more difficult to approximate.
We will formalize this intuition in the following subsections.
\begin{cor}
\label{cor:TV-eigen-f}
Let $X\sim P$ with density function $f$ on $\reals$. Then,
\begin{equation}
    \label{eq:TV-eigen-f}
    \appr(m,P,\TV) \geq \sup_{\delta>0} \frac{\pi\inf_{0\leq\theta\leq 2\pi/\delta}f(\theta)}{(m+1)\delta\exp(m^2\delta^2/2)} .
\end{equation}
Moreover, if $f$ is symmetric around zero, then the infimum in  the numerator may be take over 
$0\leq\theta\leq \pi/\delta$.
\end{cor}
\begin{proof}
    Denote $Y=\delta X \sim g$ with density function $g(\theta)=\frac{1}{\delta}f(\frac{\theta}{\delta})$. Applying \eqref{eq:eigen-lb-0} with the fact that $g^\wrap(\theta)\geq \inf_{\theta\in[0,2\pi)}g(\theta)$, we have $\lambda_{\min}(\bfT_m(Y))\geq 2\pi\inf_{\theta\in[0,2\pi)}g(\theta)$. 
    Particularly, for symmetric $f$, $\lambda_{\min}(\bfT_m(Y))\geq 2\pi\inf_{\theta\in[0,\pi)}g(\theta)$.
    Applying Proposition~\ref{prop:TV-eigen}, \eqref{eq:TV-eigen-f} then follows.
\end{proof}
}

\paragraph{Orthogonal expansion.} 
Orthogonal expansion serves as an alternative approach to lower bound the smallest eigenvalue of $\bfT_m(\delta X)$.
Let $\{\varphi_n\}$ be the orthonormal polynomials on the unit circle associated with $Y=\delta X$, that is,
\begin{equation*}
    \Expect\qth{ \varphi_j(e^{iY})\overline{\varphi_k(e^{iY})} }= \delta_{jk}, \quad j,k\geq 0.
\end{equation*}
 Let $\bfR_{m,\delta}=(R_{nj})_{n,j=0}^m$ be the $(m+1)\times(m+1)$ lower triangular matrix that encodes the coefficients of $\{\varphi_k\}_{k=0}^{m}$, namely, $\varphi_n(t) = \sum_{j=0}^n R_{nj} t^j$ and $R_{nj}=0$ if $j> n$. We call
 $\bfR_{m,\delta}$ the \textit{associated coefficient matrix} of $\{\varphi_k\}$ of order $m+1$.

  The next proposition gives a lower bound of the minimum eigenvalue based on the method in \cite{Chen1999SmallEO}. In a nutshell, we expand the quadratic form in the Rayleigh quotient  under the basis of orthogonal polynomials, and bound the norm of the coefficient matrix. 
\begin{proposition}
    \label{prop:ortho}
For any $X\sim P$ and  $\delta>0$, with the above notations,
\begin{equation*}
    \lambda_{\min}(\bfT_m(\delta X)) \geq \frac{1}{\|\bfR_{m,\delta}\|_F^2}.
\end{equation*}
   
\end{proposition}

\begin{proof} 
   For any $\bfx=(x_0,\ldots,x_m)^\top \in \complex^{m+1} \backslash \{0\}$, 
 let $\pi_m(w)=\sum_{j=0}^m x_jw^j$ with $w=w(z)=e^{i\delta z}$.
Expand $\pi_m$ under the orthogonal basis $\{\varphi_n\}$ as 
$\pi_m(w)=\sum_{k=0}^m c_k\varphi_k(w)$.
Letting $\bfc=(c_0,\ldots,c_m)^\top$ and 
$\bfT\equiv  \bfT_m(\delta X)$, the orthogonality of $\{\varphi_k\}$ implies that 
\begin{equation*}
\|\bfc\|^2=\Expect\left[|\pi_m(e^{i\delta X})|^2\right] =\bfx^\star \bfT \bfx.
\end{equation*}
Denote $\bfR\equiv\bfR_{m,\delta}$
as the coefficient matrix of $\{\varphi_k\}$ of order $m+1$.
By comparing the coefficients, we have that $\bfx^\top = \bfc^\top \bfR$, which implies that  $\|\bfx\|\leq \|\bfc\| \|\bfR\|\leq \|\bfc\|\|\bfR\|_F$.
Finally, applying \eqref{eq:rayleigh}, we have
\begin{equation*}
    \lambda_{\min}(\bfT)\geq  \frac{\|\bfc\|^2}{\|\bfc\|^2\|\bfR\|_F^2} =\frac{1}{\|\bfR\|_F^2}.
    \qedhere
\end{equation*}
\end{proof}

Proposition~\ref{prop:ortho} reduces the task of lower bounding the smallest eigenvalue to upper bounding  ${\|\bfR\|_F^2}$,
which is feasible if the orthogonal polynomial under 
the test distribution $P$ has an explicit representation. Alternatively, as will be shown in Section~\ref{sec:lb-unif},  we can  control $\|\bfR\|_F^2$ by the Szeg\"o recurrence formula \cite[Sec.1.5]{BarrySimon2005a}. 
Finally, the error lower bound is obtained by optimizing the frequency parameter $\delta$.

In the next two subsections, we sketch the application of this approach to uniform and Gaussian distributions, the associated orthogonal polynomials of which are the  standard monomials and the Rogers-Szeg\"o polynomials (see Appendix~\ref{app:ortho}), respectively. 
The wrapped density approach only uses the density lower bound, while the orthogonal expansion approach requires understanding the coefficients of the  orthogonal polynomials; nevertheless, as we will see in the case of compactly supported distributions, the latter approach yields optimal results while the former yields a trivial lower bound. 

\subsection{Distribution families under tail conditions}
\label{sec:lb-tail}
In this subsection, we establish the lower bound for the class $\calP_\alpha(\beta)$ of
sub-Weibull distributions (recall \eqref{eq:Pab})  by applying the wrapped density approach in Section~\ref{sec:spectral}.
We focus on the case where $\alpha$ is fixed and the scale parameter $\beta$ may vary with $m$. 
Let $X_{\alpha,\beta} \sim P_{\alpha,\beta}$ with the density function 
\begin{equation}
     \label{eq:f_ab}
    f_{\alpha,\beta}(x)=\frac{C_\alpha}{\beta}\exp\qth{-\left|{\frac{x}{\beta}}\right|^\alpha},
\end{equation}
where $C_\alpha=\pth{\int_{-\infty}^{\infty} \exp\pth{-|x|^\alpha} \diff x}^{-1}$ is the normalization constant. We establish the following lower bound of the approximation error measured by total variation distance.
\begin{theorem}
    \label{thm:lb-tail-1}
    For any $m\in\naturals$ and  $\alpha,\beta>0$,
    \begin{equation*}
        \appr(m,P_{\alpha,\beta},\TV) \geq \beta^{-\frac{2}{2+\alpha}}m^{-\frac{\alpha}{2+\alpha}}
        \exp\qth{-\Theta_\alpha\pth{\pth{\frac{m}{\beta}}^{\frac{2\alpha}{2+\alpha}}}}.
    \end{equation*}
\end{theorem}
\begin{proof}
Denote $X \sim P_{\alpha,\beta}$ and $Y=\delta X \sim g$,
where $g(\theta)=\frac{1}{\delta}f_{\alpha,\beta}(\frac{\theta}{\delta})$. 
Note that $f_{\alpha,\beta}$ is symmetric and decreasing in $(0,\infty)$. 
Applying \eqref{eq:TV-eigen-f} with $\delta=\pth{\frac{1}{m}}^{\frac{2}{2+\alpha}} \pth{\frac{2\pi}{\beta}}^{\frac{\alpha}{2+\alpha}}$, we have
\begin{align*}
    \appr(m,P_{\alpha,\beta},\TV) &\geq \frac{\pi f_{\alpha,\beta}\pth{\frac{2\pi}{\delta}}}{2m\delta\exp(m^2\delta^2/2)}
    \geq  C_{1} \beta^{-\frac{2}{2+\alpha}}m^{-\frac{\alpha}{2+\alpha}}
    \exp\qth{ -C_{2}  \pth{\frac{m}{\beta}}^{\frac{2\alpha}{2+\alpha}} },
\end{align*}
for some constants $C_{1},C_{2}$ depending only on $\alpha$.
\end{proof}
As concrete examples, we consider the case where $P$ is the Laplace  or Gaussian distribution. Consider the $\operatorname{Laplace}(\lambda)$ distribution with density $f(x)=(2\lambda)^{-1} e^{-|x|/\lambda}$. Applying Theorem~\ref{thm:lb-tail-1} with $(\alpha,\beta)=(1,\lambda)$, it follows that
    \begin{equation*}
     \appr(m,\operatorname{Laplace}(\lambda),\TV) 
     \geq\frac{\pi}{4\sqrt[3]{2\pi m\lambda^2}}\exp\qth{-\pth{\frac{2\pi m}{\lambda}}^{\frac{2}{3}}},
    \end{equation*}    
    where the explicit constants are obtained from \eqref{eq:TV-eigen-f}. 
    Similarly, for the $N(0,\sigma^2)$ distribution with $(\alpha,\beta)=(2,\sqrt{2}\sigma)$,
   \begin{equation}
    \label{eq:lb-gaussian}
    \appr(m,N(0,\sigma^2),\TV) 
    \geq\frac{1}{2\sqrt{2m\sigma}}\exp\pth{-\frac{\pi m}{\sigma}}.
   \end{equation} 
   We finally remark that Theorem~\ref{thm:lb-tail-1} also applies to lower bound  $\appr(m,\calP_{\alpha}(\beta),\TV)$ 
   by noting that $P_{\alpha,\beta}\in\calP_{\alpha}(k_\alpha\beta)$ for $k_\alpha=(1-2^{-\alpha})^{-\frac{1}{\alpha}}$, $\alpha>0$.
\begin{remark}
\label{rem:lb-1}
Alternatively, \eqref{eq:lb-gaussian} can be shown 
using Proposition \ref{prop:ortho} and orthogonal polynomials for the wrapped Gaussian distribution. This second proof is given in Appendix~\ref{app:lb--subg}. 
\end{remark}

Next, we consider the regime when the scale parameter $\beta$ diverges. In this case, more components are needed to achieve a prescribed  approximation accuracy. Hence, it is natural to conjecture that consistent approximation is impossible unless $m$ is at least proportional to the scale parameter. This however is not captured by Theorems~\ref{thm:lb-tail-1} where the lower bound still converges to zero when $m\asymp\beta$ due to the  polynomial term in $m$.

The next result makes the intuition precise in a strong sense for the given test distributions. 
Instead of applying Proposition \ref{prop:TV-eigen} and bounding the minimum eigenvalue, we execute a different but more straightforward proof strategy: 
If the mixing distribution $P$ is smooth enough with a large $\beta$, 
then $f_P$ can be flat over a large region. 
In comparison, when $m$ is too small, $f_{P_m}$  inevitably has many modes and thus cannot approximate $f_P$ consistently.
\begin{theorem}
    \label{thm:inapprox}
    Let $P_{\alpha,\beta}$ and $C_\alpha$   
        be defined in \eqref{eq:f_ab}. 
Denote $\tilde{C}_\alpha={\sqrt{2\pi }C_\alpha}$. For any $\beta\geq 2\tilde{C}_\alpha m$, 
    \begin{equation}
        \label{eq:inapprox-tail}
        \appr(m,P_{\alpha,\beta},\TV)\geq 1- 5 \frac{\tilde{C}_\alpha m}{\beta} \sqrt{\log \frac{\beta}{\tilde{C}_\alpha m}}.
    \end{equation}
\end{theorem}
\begin{proof}
For $P=P_{\alpha,\beta}$, note that 
\begin{equation*}
    f_{P}(x)=\int \phi(x-\theta) f_{\alpha,\beta}(\theta)\diff\theta \leq \frac{C_\alpha}{\beta}\int \phi(x-\theta) \diff\theta = \frac{C_\alpha}{\beta}.
\end{equation*}
Fix any $m$-atomic distribution $P_m=\sum_{i=1}^m w_i \delta_{\theta_i}$. Define 
\begin{equation*}
R_i \triangleq \sth{x: w_i \phi(x-\theta_i)\ge \frac{C_\alpha}{\beta} }
\end{equation*}
as the peak region of the $i\Th$ mixture component. 
Denote $\eta_i=\eta(w_i)\triangleq\sqrt{2\log\frac{\beta w_i}{\tilde{C}_\alpha}}$ if $w_i\ge \frac{\tilde{C}_\alpha}{\beta}$.
Then, $R_i = [\theta_i - \eta_i, \theta_i + \eta_i ]$ 
if $w_i\ge \frac{\tilde{C}_\alpha}{\beta}$, and $R_i=\emptyset$ otherwise.
Let $F_{m}$, $F$, and $G_\theta$ be the probability measures corresponding to densities $f_{P_{m}}$, $f_P$, and $\phi(\cdot - \theta)$, respectively.
Define $R\triangleq \cup_{i=1}^m R_i$. By the union bound, 
\begin{equation*}
    \TV\left(f_{P_{m}} , f_{P}\right)\ge  F_{m}(R) - F(R) \ge \sum_{i=1}^m w_i G_{\theta_i}(R_i) - \sum_{i=1}^m F(R_i).
\end{equation*}
Since $f_{P} \leq \frac{C_\alpha}{\beta}$, we have 
$R_i\subseteq \{x: w_i \phi(x-\theta_i) \ge f_{P}(x) \}$.
Consequently, $w_i G_{\theta_i}(R_i) - F(R_i)\ge 0$ for all $i\in [m]$. Consider the following set of indices of mixture components with large weights
\begin{equation*}
\calI \triangleq \sth{i\in[m]:  w_i \ge \frac{2\tilde{C}_\alpha}{\beta}}.
\end{equation*}
Note that our assumption $\beta\geq 2\tilde{C}_\alpha m$ implies that $\{i\in[m]:  m w_i \ge 1\}\subseteq\calI$. Hence, by the pigeonhole principle, $\calI$ is nonempty.
For each $i\in \calI$, we have $R_i = [\theta_i - \eta_i, \theta_i + \eta_i ]$
and $\eta_i\geq \sqrt{2\log 2}$. It follows that
\begin{align}
    \label{eq:inapprox-prf-2}
    w_i G_{\theta_i}(R_i) \!-\! F(R_i)
\!\ge\! w_i\pth{1 \!-\! 2\Phi_c\pth{\eta_i}} \!-\! \frac{2\eta_i C_\alpha}{\beta}
\!\ge\! w_i\!\qth{1 \!-\! \frac{2C_\alpha}{\beta w_i} \! \pth{\frac{1}{\eta_i}\!+\!\eta_i\!}\!}, 
\end{align}
where the second inequality used $\Phi_c(\eta_i)\triangleq \int_{\eta_i}^\infty \phi(u)du\le \frac{1}{\eta_i}\phi(\eta_i)$.
Then,
\begin{align*}
\TV\left(f_{P_{m}} , f_{P}\right)
& \ge \sum_{i\in\calI} w_i G_{\theta_i}(R_i) - F(R_i)\\
& \stepa{\ge} 1- \sum_{i\not\in\calI} w_i - \sum_{i\in\calI} \frac{4C_\alpha}{\beta}\eta(w_i)\\
& \stepb{\ge} 1 - m\frac{2\tilde{C}_\alpha }{\beta} - \frac{4C_\alpha}{\beta}|\calI|\eta(|\calI|^{-1})\\
& = 1 - \frac{2\tilde{C}_\alpha m}{\beta} - \frac{4}{\sqrt{\pi}}\frac{\tilde{C}_\alpha|\calI|}{\beta} \sqrt{\log\frac{\beta}{\tilde{C}_\alpha |\calI|}}\\
& \stepc{\ge} 1 - \frac{2\tilde{C}_\alpha m}{\beta} - \frac{4}{\sqrt{\pi}}\frac{\tilde{C}_\alpha m}{\beta} \sqrt{\log\frac{\beta}{\tilde{C}_\alpha m}},
\end{align*}
where (a) follows from \eqref{eq:inapprox-prf-2}, (b) uses the definition of $\calI$
and the concavity of $x\mapsto \sqrt{\log x}$ on $x\geq 1$, 
and (c) holds since $\frac{\beta}{\tilde{C}_\alpha |\calI| }\geq \frac{\beta}{\tilde{C}_\alpha m }\geq 2$ and
$\frac{\sqlog{x}}{x}$ is decreasing in $x\geq \sqrt{e}$.
Consequently, 
\begin{equation*}
    \TV\left(f_{P_{m}} , f_{P}\right) \ge 1- \pth{\frac{2}{\sqlog 2}+\frac{4}{\sqrt{\pi}}} \frac{\tilde{C}_\alpha m}{\beta} \sqrt{\log \frac{\beta}{\tilde{C}_\alpha m}}\geq 1 - 5 \frac{\tilde{C}_\alpha m}{\beta} \sqrt{\log \frac{\beta}{\tilde{C}_\alpha m}}.
\end{equation*}
The lower bound holds uniformly for all $P_m\in\calP_m$, and hence the result \eqref{eq:inapprox-tail} follows.
\end{proof}

\subsection{Compactly supported distributions}
\label{sec:lb-unif}
We focus on the class $\Pbdd{M}$ of all distributions supported on $[-M,M]$. 
Theorem~\ref{thm:lb-unif} shows that the desired lower bound can be obtained by considering the uniform distribution.

\begin{theorem}
  \label{thm:lb-unif}
  There exists universal constants $C_1$ and $C^\prime$ such that for any $m\geq C^\prime (M\vee 1)^2$,
\begin{equation}
    \label{eq:lb-unif-1}
    \appr(m,\Unif[-M,M],\TV)\geq
    \exp\pth{-C_1m\log\pth{\frac{m}{M^2}}},
\end{equation}
and it holds for any $M>0$, $m\in\naturals$ that 
\begin{equation}
    \label{eq:lb-unif-2}
\appr(m,\Unif[-M,M],\TV) \geq
\frac{1}{4m}\exp\left(-\frac{\pi^2 m^2}{2M^2}\right).
  \end{equation}
Moreover, for $C=\sqrt{\frac{\pi}{2}}$ and any $M\geq \sqrt{2\pi}m$,  
\begin{equation}
    \label{eq:lb-unif-inapprox}
    \appr(m,\Unif[-M,M],\TV)\geq 1- 5 \frac{Cm}{M} \sqrt{\log \frac{M}{C m}}.
\end{equation}
\end{theorem}

We give some interpretation of Theorem~\ref{thm:lb-unif} as follows. 
\begin{itemize}
    \item \emph{Non-vanishing error.}
When $m\leq o(M)$, \eqref{eq:lb-unif-inapprox} gives a compact support analogue of Theorem~\ref{thm:inapprox}, demonstrating that $m$-component mixtures cannot yield any non-trivial approximation as the TV error is $1-o(1)$. The derivation of \eqref{eq:lb-unif-inapprox} is essentially the same as those used in Theorem~\ref{thm:inapprox}. 
\item \emph{Polynomially small error.} When $M\lesssim m\lesssim M^2$, \eqref{eq:lb-unif-2} dominates the lower bound. The lower bound can be directly obtained by applying Proposition \ref{prop:TV-eigen} to  the trigonometric moment matrix of the uniform distribution. 

\item \emph{Exponentially small error.} When $m\gtrsim M^2$, \eqref{eq:lb-unif-1} becomes effective. 
The proof is more challenging,
which requires a small frequency parameter $\delta$ in Proposition \ref{prop:TV-eigen}.
For $X$ supported on $[-M,M]$, $\delta X\mod 2\pi \integers$ is the same as $\delta X$ for small $\delta$.
As such, the minimum value of the wrapped density function is zero, which only yields a trivial lower bound. 

For this result, we first lower bound the approximation error in $\chi^2$, which can be related to, via the variational representation of the $\chi^2$-divergence, the smallest eigenvalue of a certain \textit{weighted moment matrix}. 
Then, we convert the $\chi^2$-divergence lower bound to a $\TV$ lower bound using a truncation argument for Gaussian mixtures.
\end{itemize}

\begin{proof}[Proof of Theorem~\ref{thm:lb-unif}]
The proof of \eqref{eq:lb-unif-1} is provided in Appendix~\ref{app:lb-unif-1} -- see Theorem \ref{thm:lb-unif-1} therein.
    Let $X\sim P=\Unif[-M,M]$
    and $\delta= \frac{\pi}{M}$. By Corollary~\ref{cor:TV-eigen-f},
\begin{align*}
\appr(m,P,\TV) \geq  \frac{1}{2(m+1)\exp(m^2\delta^2/2)}\geq \frac{1}{4m}\exp\left(-\frac{\pi^2 m^2}{2M^2}\right).
\end{align*}

The proof of \eqref{eq:lb-unif-inapprox} is essentially the same as the proof of Theorem~\ref{thm:inapprox}.
Note that $f_P(x)=\frac{1}{2M}\int_{-M}^M \phi(x-\theta)\diff\theta\leq \frac{1}{2M}. $ 
For any 
 $P_m=\sum_{i=1}^m w_i \delta_{\theta_i}\in\calP_m$, define 
\[
R_i = \{x: w_i \phi(x-\theta_i)\ge (2M)^{-1} \},\quad 
\calI \triangleq \{i\in[m]: w_i \ge \sqrt{2\pi}M^{-1} \}.
\]
 The result is proved
 via a similar argument to that of Theorem~\ref{thm:inapprox}.
\end{proof}

For \( m \gtrsim M^2 \), 
the spectral analysis framework remains applicable if we consider certain compactly supported distribution other than the uniform distribution. 
This framework will be extended to multiple dimensions in Section~\ref{sec:extension-other}.
To illustrate this, let $b\in (0,\pi]$ and $\tilde{P}$ be the distribution supported on $[-b,b]$ with the density function  
\begin{equation}
    \label{eq:f_test_unif}
    \tilde{f}(\theta)=
    \begin{cases}
    \frac{\sqrt{\sin^2(b/2)-\sin^2(\theta/2)}\cos(\theta/2)}{\pi \sin^2(b/2)},& \theta\in[-b,b],\\
    0,& \text{otherwise.}
\end{cases}
\end{equation}
Let $X\sim \tilde P$ and $\delta=\frac{b}{M}$. Then $\frac{X}{\delta}$ is supported on $[-M,M]$.
Let $P$ denote the distribution of $\frac{X}{\delta}$. 
The following result holds:
\begin{proposition}
    \label{prop:lb-unif-arc}
    Let $P$ be defined as above.
    Then, there exists $b\in (0,\pi]$ and universal constants $C_1$, $C^\prime$ such that when $m\geq C^\prime M^2$,
\begin{equation}
    \label{eq:lb-unif-arc}
    \appr(m,P,\TV)\geq
    \exp\pth{-C_1m\log\pth{\frac{m}{M^2}}}.
\end{equation}
\end{proposition}
To prove Proposition~\ref{prop:lb-unif-arc}, we connect the orthogonal system on an arc of the unit circle to the orthogonal system on an
interval of the real line (see Lemma~\ref{lem:arc}). Specifically, 
the orthonormal polynomials associated with $\tilde{P}$ can be related to the (scaled) Chebyshev polynomials of the second kind. Propositions~\ref{prop:TV-eigen} and~\ref{prop:ortho} can then be applied, where the explicit form of the Chebyshev polynomials is utilized to bound the norm of the associated coefficient matrix.
 The complete proof is presented in Appendix~\ref{app:lb-unif-arc}.

\section{Applications and extensions}
\label{sec:discussion}
In this section, we discuss the applications of our results in information theory and nonparametric statistics. Additionally, we show how our findings can be extended to a wider range of distribution families and mixture models.  

\subsection{Constellation design for Gaussian channels}
\label{sec:constellation}
The finite mixture approximation problem is closely related to the input constellation design for Gaussian channels. Recall that the  Gaussian distribution is capacity-achieving for Gaussian channels under average power constraint. Under an additional input cardinality constraint of at most \( m \), the capacity gap is characterized by \( \appr(m, N(0, \sigma^2), \KL) \)  (see \eqref{eq:Cm}). Consequently, our proposed approximator serves as an \( m \)-point constellation that approaches the limit at the optimal rate.

We implement our finite mixture approximation in Algorithm~\ref{algo:mm}, utilizing the moment-matching approach for truncated distributions as described in Section~\ref{sec:ub-tail}. 
Theorem~\ref{thm:ub--subW} implies that setting \( t \asymp \sigma \sqrt{m \log(1 + \sigma^{-1})} \) and \( K \asymp \sigma^2 \log(1 + \sigma^{-1}) \) yields a rate-optimal approximator. 
For the Gauss quadrature, we compute the moments of conditional distributions and apply the moment-based Golub-Welsch algorithm~\cite{GW69}, as detailed in Algorithm~\ref{algo:quad} in Appendix~\ref{app:GQ}.

\begin{algorithm}
    \caption{TruncQuad constellation for Gaussian channels} 
    \label{algo:mm}
    \begin{algorithmic}
    \STATE \textbf{Input:} Number of atoms $m$, signal-noise ratio $\sigma^2$, $c_1,c_2,c_3>0$
    \STATE \textbf{Output:} $m$-atomic distribution $P_m\in\calP_m$
    \STATE Define {$P=N(0,\sigma^2)$, }
\(
t = c_1 \sigma \sqrt{m \log(c_2 + \sigma^{-1})}
\), 
\(
K = \min\left(\lceil c_3  t^2 / m \rceil, m\right)
\), and $m_0=\lfloor m / K \rfloor$
\STATE Partition the intervals:
\(
I_j = [-t + (j - 1)\frac{2t}{K}, -t + j\frac{2t}{K}] , j\in[K]
\) 
\STATE Compute the conditional moment $m_{j}(P_{(k)})= \int_{I_k} x^j \diff P/P(I_k)$ for $j\in [2m_0-1]$ and $k\in[K]$
    \STATE For each $k\in[K]$, compute $P_{k}$ by Algorithm~\ref{algo:quad} with input $\bfm_k=\pth{m_{j}(P_{(k)})}_{j=1}^{2m_0-1}$
\STATE Compute $P_m = \sum_{k=1}^K P(I_k)P_{k}$
    \end{algorithmic}
\end{algorithm}

We conduct numerical experiments to evaluate the capacity gap of the truncated quadrature method (TruncQuad) described in Algorithm~\ref{algo:mm} with parameters \( (c_1, c_2, c_3) = (2, 5, 0.2) \). 
Fixing \( \sigma^2 \in \{1, 10, 50\} \), we measure the capacity gap as the number of atoms \( m \) increases. 
We compare its performance against the following constellation schemes considered in \cite{WV10}:

\begin{enumerate}[label=(\arabic*)]
    \item Equilattice: the constellation uniformly distributed on the grid \( E_m \), defined as 
    \begin{equation*}
    E_{m}=\begin{cases}
    \left\{2 i \Delta_{m}: i=\frac{1-m}{2}, \ldots, 0, \ldots, \frac{m-1}{2}\right\} & m \text { odd } \\
    \left\{(2 i+1) \Delta_{m}: i=-\frac{m}{2}, \ldots, 0, \ldots, \frac{m-2}{2}\right\} & m \text { even, }
    \end{cases} \quad \Delta_{m}=\sqrt{\frac{3}{m^{2}-1}}.
    \end{equation*}
    \item Quadrature: the Gauss quadrature rule based on the first $2m-1$ moments of $P$.
    \item Quantized: the real line is partitioned into \( m \) subintervals using quantile sequences \( \alpha_j = \Phi^{-1} \left( \frac{j-1}{m} \right) \), \( j = 0, \ldots, m \). The \( m \)-point uniform distribution is then supported on the conditional expectations \( x_j = \mathbb{E}_P [ X \mid X \in [\alpha_{j-1}, \alpha_j]] \) for \( j \in [m] \), with the constellation scaled to achieve variance \( \sigma^2 \). 
    \item Random walk: sample $B_{m} \sim \operatorname{Binomial}(m-1,1 / 2)$ and define $X_m = \frac{2\sigma}{\sqrt{m-1}}\left(B_{m}-\frac{m-1}{2}\right)$ as the input. By the central limit theorem,  $\hat{X}_{m}$ asymptotically converges to $N(0,\sigma^2)$.
\end{enumerate}

Figure~\ref{fig:m_capacity_gap} presents the results, showing that both TruncQuad and Gauss-Hermite quadrature converges exponentially, whereas the quantized Gaussian and random walk schemes converge more slowly at a polynomial rate. Meanwhile, the equilattice input achieves only a constant error, as it converges weakly to \( \Unif([- \sqrt{3}, \sqrt{3}]) \) as \( m \to \infty \).  
Moreover, TruncQuad yields smaller errors than Gauss-Hermite quadrature and significantly so for high SNR (large \( \sigma \)), which aligns with our theoretical analysis. 
Notably, TruncQuad converges at a rate of \( \exp\left(-\Omega\left(m \log\left(1 + \frac{1}{\sigma}\right)\right)\right) \) (see \eqref{eq:ub-subg}), outperforming the \( \sigma^2 \exp\left(-\Omega\left(m \log\left(1 + \frac{1}{\sigma^2}\right)\right)\right) \) rate of Gauss-Hermite quadrature~\cite[Theorem 8]{WV10} when \( \sigma \) is large.   
\begin{figure}[H]
    \centering
    \includegraphics[width=\linewidth]{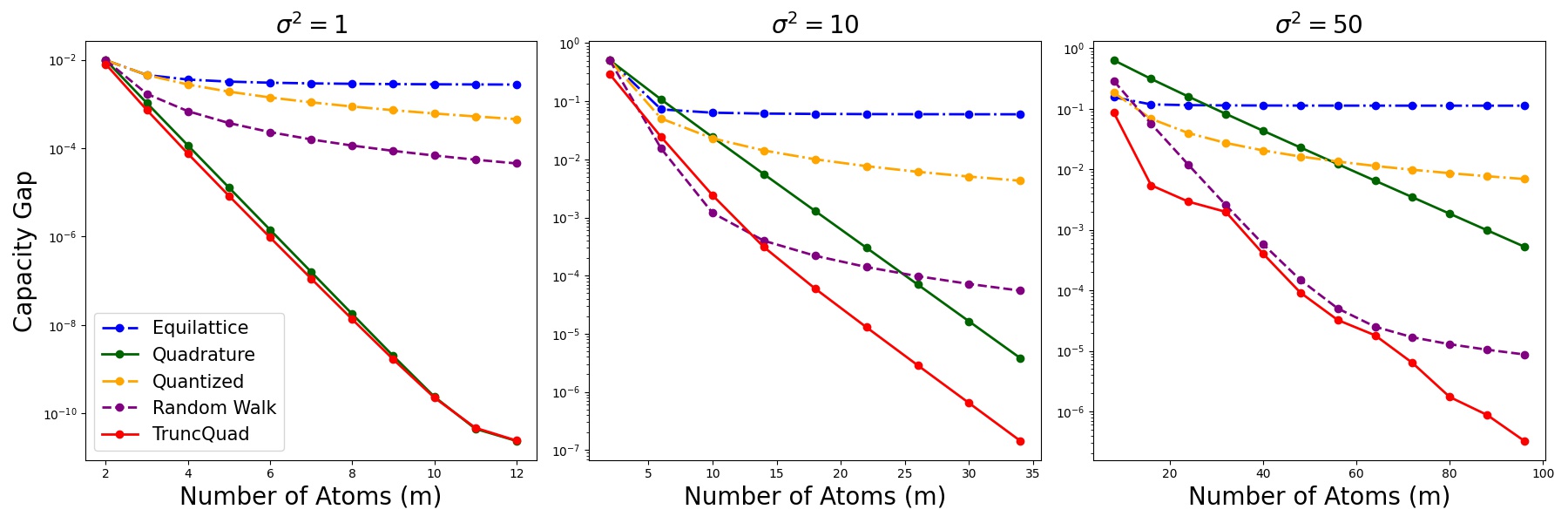}
    \caption{Capacity gap  with $\sigma^2=1,10,50$ as $m$ increases.}
    \label{fig:m_capacity_gap}
\end{figure}

\subsection{Convergence rates of nonparametric maximum likelihood estimator}
\label{sec:mle-rate}
Given a sample of $n$ independent observations $X=(X_1,\ldots,X_n)$, where $X_i\iiddistr f_{P^\star}$,
the nonparametric maximum likelihood estimator (NPMLE) is defined as \cite{KW56}
\begin{equation}
    \label{eq:npmle-def}
   {\hat{P}(\calQ) \in \argmax_{Q\in\calQ} \frac{1}{n}\sum_{i=1}^n \log f_Q(X_i),}
\end{equation}
where the optimization is over a class $\calQ$ of priors. 
The finite-order approximation of a family of distributions naturally yields upper bounds of the metric entropy of that family, which further determines the convergence rate of the NPMLE \cite{WS95,Vaart1996WeakCA,Genovese00,GV01,JZ09}.
Particularly, \cite{Zhang08} studies the \textit{unconstrained} NPMLE where $\calQ$ is the class of all distributions on $\reals$. 
Following the analysis in \cite{Zhang08}, the Hellinger distance $H(f_{\hat P}, f_{P^\star})$ can be upper bounded in terms of $m^\star(\epsilon,\calQ,L_\infty[-M,M])$.
In this regard, our analysis recovers the convergence rates established in \cite{Zhang08} by deriving an upper bound for \( m^\star(\epsilon, \mathcal{Q}, L_\infty[-M, M]) \) as an application of Theorem~\ref{thm:ub-bdd}. This result is detailed in Lemma~\ref{lem:npmle_uc}.

Moreover, suppose $P^\star$ belongs to a given family of distributions $\calP$, Theorem~\ref{thm:rate-families} below in fact establishes a faster rate 
for the \textit{constrained} NPMLE with $\calQ=\calP$, which is 
characterized by 
{$\comp(n^{-1/2},\calP, H)\frac{\log n}{n}$} with respect to the Hellinger distance\footnote{
We note that existing minimax lower bounds in \cite[Theorem 20]{PW21} agree with 
{$\comp(n^{-1/2},\calP, H)/n$}. 
However, bridging this gap remains an open problem.}.
In contrast to \cite[Theorem 1]{Zhang08} which yield the same convergence rates {of the unconstrained NPMLE} for the families $\Pbdd{M}$ and $\calP_\alpha(\beta)$, our result gives a strictly faster rate for the former.
{The improvement stems from the additional structural information \( P^\star \in \mathcal{P} \), which enables the application of our tight approximation rates.}
In addition, another technical contribution is a simplified proof: Compared with that of \cite[Theorem 1]{Zhang08}, 
we simplify the argument by first proving a $L_\infty$-entropy estimate then directly bounding the square root of the mixture density.  The proofs are deferred to Appendix~\ref{app:proof-mle-rate}.

\begin{theorem}
    \label{thm:rate-families}
    \begin{enumerate}
        \item Let $P^\star\in \Pbdd{M}$, where $M \leq n^{c}$ for some constant $c>0$.
        {Denote $\hat P=\hat P(\Pbdd{M})$. }
        There exist constants $s^\star,c^\prime>0$ such that
        for any $s \geq s^\star$, 
    \begin{equation}
    \label{eq:npmle-rate}
        \pbb[H(f_{\hat{P}}, f_{P^\star} )\geq s\epsilon_n]\leq n^{-c^\prime s^2},
    \end{equation} 
    where
        \begin{equation}
            \label{eq:rate-bdd}
            \epsilon_n= \frac{\log n}{\sqrt{n\log\left(1+ \sqlog{n}/M\right)}}.
        \end{equation}
    \item Let $P^\star\in \calP_\alpha(\beta)$, where $\alpha>0$ is fixed and $\beta \leq n^{c}$ for some constant $c>0$.
    {Denote $\hat P=\hat P(\calP_\alpha(\beta))$. }
    There exist $s^\star,c^\prime >0$ which depend on $\alpha$ such that
    for any $s \geq s^\star$,   
    \eqref{eq:npmle-rate} holds with
        \begin{equation}
            \label{eq:rate-tail}
            \epsilon_n= \frac{\log n}{\sqrt{n\log\pth{1+\pth{\log n}^{\frac{\alpha-2}{2\alpha}}/\beta}}}.
        \end{equation}
    \end{enumerate}
\end{theorem}

\begin{remark}
    Theorem~\ref{thm:rate-families} can be extended to other divergences using the divergence comparison inequalities \cite{JPW23}. For example, for \( d=\KL \) and \( \chi^2 \) and any $P,Q \in \Pbdd{M}$, the inequalities  take the form 
\(
d(f_{P} , f_{Q}) \leq C_M H^2(f_{P}, f_{Q})
\) with some \( C_M \) depending on \( M \). By \eqref{eq:npmle-rate}, we have
\[
\mathbb{P}[d(f_{\hat{P}} , f_{P^\star}) \geq s^2 C_M \epsilon_n^2] \leq \mathbb{P}[H(f_{\hat{P}}, f_{P^\star}) \geq s \epsilon_n] \leq n^{-c^\prime s^2},
\]
which implies an \( O(C_M \epsilon_n^2) \) rate in $d(f_{\hat{P}} , f_{P^\star})$ for the constrained NPMLE on $\Pbdd{M}$.  
The tight dependency on the parameters $M$ remains open. 
\end{remark}

\subsection{Mixing distributions under moment constraints}
\label{sec:extension-moment}
So far, our results deal with mixing distributions that exhibit exponential decay in their tails. 
In fact, our analysis framework also allows for extension to distributions with heavier tails.
Let us consider mixing distributions under moment conditions, which were well-studied in mixture models \cite{GV07,Zhang08,Saha19,SW22}
 Let $\psi(x)=x^\alpha$ for $\alpha>0$. Then the corresponding $\psi$-Orlicz norm is simply the usual
    $L_\alpha$-norm 
 $\|X\|_\psi = \|X\|_\alpha \triangleq \Expect[|X|^\alpha]^{\frac{1}{\alpha}}$. Define
    $$\calM_\alpha(\beta)\triangleq\{P: X\sim P, \|X\|_\alpha\leq \beta \},$$ 
    that is, the family of distributions whose  $\alpha\Th$ moment is no larger than $\beta^\alpha$. For $P\in \calM_\alpha(\beta)$, \eqref{eq:orlicz_tail} implies the polynomial-type tail bounds $P[|X| \geq t]\leq 2\pth{\frac{\beta}{t}}^\alpha$.
\begin{theorem}
    \label{thm:main-moment}
    There exists $c_1,c_2$ (which may depend on $\alpha$) such that the following holds. When $m\geq c_1\beta$ and $\beta \geq c_2$,
    \begin{equation*}
        \frac{1}{m\log\frac{m}{\beta}}\pth{\frac{\beta}{m}}^\alpha\lesssim_\alpha\appr(m,\calM_\alpha(\beta),\TV)\lesssim_\alpha 
       \pth{\frac{\beta}{m}\sqlog{\frac{m}{\beta}}}^\alpha.
    \end{equation*}
\end{theorem}

In contrast to the sub-Weibull family, distributions in $\calM_\alpha(\beta)$ have much heavier tails and thus require more components to approximate. As such, the best approximation error is no longer exponential but polynomial in $m$ for light-tailed, as shown in  Theorem~\ref{thm:main-moment}.
The derivation of the upper and lower bounds essentially follows the same steps for Theorems~\ref{thm:ub--subW} and~\ref{thm:lb-tail-1}, respectively, and the proofs are deferred to Appendix~\ref{app:main-moment}.

There is a gap between the lower and upper bound polynomial in $m$. Following our derivation, the gap appears 
due to the extra $m^{-1}$ term in the Proposition~\ref{prop:TV-eigen}. In Theorems~\ref{thm:main-bdd} and~\ref{thm:main-tail}, the extra polynomial term does not affect the tightness of the results since the exponential term dominates the rate. On the other hand, Theorem~\ref{thm:main-moment}
is limited to the $\TV$ distance, since other divergences (\eg, the $\KL$ and $\chi^2$ divergence) may vary within an exponent from the $\TV$ distance and thus widen the gap of the polynomial rate. A potential remedy is to characterize the information geometry of the Gaussian mixture families, as carried out in \cite{JPW23} for compactly supported and subgaussian mixing distributions.
The tight rate is still left to be discovered.

\subsection{Generalization to other mixture models}
\label{sec:extension-other}
So far we have focused on approximating general one-dimensional Gaussian location mixture by finite mixtures. Our techniques of finite-order approximation can also adapt to other models, including the Gaussian location-scale mixtures and multi-dimensional mixtures. In these contexts, most of the existing approximation results \cite{Saha19,Jiang20,SGS21,SW22}  in the empirical Bayes literature rely  on Taylor expansion and local moment matching. 

\paragraph{Location-scale Gaussian mixtures.}
Consider the location-scale Gaussian mixture with density $$f_H(x)=\int \frac{1}{\sigma}\phi\pth{\frac{x-\theta}{\sigma}} \diff H(\theta,\sigma),$$ where $H$ is the joint mixing distribution of location and scale parameters. 
In other words, 
if $(X,S)\sim H$ and $Z\sim N(0,1)$ are independent, then $X+SZ\sim f_H$.
Let $H_m$ denote an $m$-atomic approximation with $X_m+S_mZ\sim f_{H_m}$. 

Previous works applied moment-based approach to construct approximators. 
{Suppose that $H$ is supported on $[-a, a] \times[\underline{\sigma}, \bar{\sigma}]$ for some positive constants $a$ and $\underline{\sigma}\leq \bar{\sigma}$.} 
By applying Taylor expansion and Carath\'eodory's theorem, \cite[Lemma 3.4]{GV01} constructs an $H_m$ with $m\leq O(\log^2\frac{1}{\epsilon})$ such that $\Norm{f_H-f_{H_m}}_\infty \le \epsilon$. 
In addition, if $X$ and $S$ are independent, both $X_m$ and $S_m$ are supported on $O(\log\frac{1}{\epsilon})$ points.

Under the above compact support condition, a location-scale Gaussian mixture can be reduced to a location mixture. 
Specifically, define $X'=X+\sqrt{S^2-\underline{\sigma}^2}Z \sim P'$ and $Z_{\underline{\sigma}}\sim N(0,\underline{\sigma}^2)$.
Then 
\[
X'+Z_{\underline{\sigma}}\sim f_H,
\]
and $X'$ is $\sigma^2$-subgaussian for some $\sigma$ depending on $a$ and $\bar{\sigma}$.
Applying \eqref{eq:comp_subG} and Remark~\ref{remark:var}, there exists an $m$-atomic $P_m$ on $\reals$ with $m\leq O(\log\frac{1}{\epsilon})$ such that
$\TV(f_{P_m},f_{H})\leq \epsilon$.
Furthermore, since $\|f_H'\|_\infty\leq O(1)$, $\TV(f_{P_m},f_{H})\lesssim\epsilon$ implies $\Norm{f_{P_m}-f_{H}}_\infty\lesssim\sqrt{\epsilon}$. Consequently, the $O(\log\frac{1}{\epsilon})$ upper bound also holds for $d=L_\infty$, thereby improving the previous results in \cite{GV01}.

For the  lower bound, applying Lemma~\ref{lem:variation} yields that
\begin{align}
    \TV\left(f_{H_{m}} , f_{H}\right)
&\geq \frac{1}{2} \sup _{\omega} \left|\mathbb{E}\qth{\exp\pth{i\omega X-\frac{1}{2}\omega^2S^2}}-\mathbb{E}\qth{\exp\pth{i\omega X_m-\frac{1}{2}\omega^2S_m^2}}\right|.
\label{eq:TVlb-locationscale}
\end{align}
It is unclear how to lower bound this uniformly over $(X_m,H_m)$ that is $m$-atomic.

A key challenge arises from the distinct effects of the location and scale parameters on the mixture distribution.
Following our framework, the corresponding low-rank matrix over a grid \( \omega = \{\omega_j\}_{j=0}^{m} \subseteq \reals \) can be constructed as
\[
\mathbf{V}_m \triangleq \mathbb{E} \left[ \mathbf{v}_m(X_m, S_m) \mathbf{v}_m^\star(X_m, S_m) \right] = \left( \mathbb{E} \left[ e^{-i(\omega_j - \omega_k) X_m - \frac{1}{2} (\omega_j^2 + \omega_k^2) S_m^2} \right] \right)_{j,k=0}^{m},
\]
where \( \mathbf{v}_m(x,s) \triangleq \left( 1, e^{-i\omega_1 x - \frac{1}{2} \omega_1^2 s^2}, \dots, e^{-i\omega_m x - \frac{1}{2} \omega_m^2 s^2} \right) \).
However, the entries in $\mathbf{V}_m$ differ from the terms in the variational lower bound \eqref{eq:TVlb-locationscale}.

Nevertheless, for the special case of scale mixtures with the additional constraint \( X_m = X \equiv 0 \), we can further lower bound \eqref{eq:TVlb-locationscale} as follows.
First, by choosing $\omega\in \{\sqrt{k}: k\in \{0,\ldots,2m\}\}$, using the same argument as in 
Proposition \ref{prop:TV-eigen}, \eqref{eq:TVlb-locationscale} can be related to the minimum eigenvalue $\lambda_{\min}$ of the Hankel matrix $\pth{\Expect \exp\pth{-\frac{(j+k)S^2}{2}} }_{j,k=0}^{m}$. Applying the method based on orthogonal expansion in Proposition \ref{prop:ortho} (cf.~Appendix \ref{app:lb-unif-1}), one can show that $\lambda_{\min} \geq e^{-\Omega(m)}$, leading to the matching lower bound that at least $\Omega(\log\frac{1}{\epsilon})$ components are needed to achieve an approximation error $\epsilon$ under $\TV$ in the scale mixture model.
Establishing lower bounds for general location-scale mixtures remains an outstanding challenging.

\paragraph{Multiple dimensions.}
Fix $d\in\naturals$.
Consider the Gaussian location mixture on $\reals^d$: $$f_P(\bfx)=\int \phi_{(d)}(\bfx-\bstheta) \diff P(\bstheta),$$
where $\phi_{(d)}(z)\triangleq (2\pi)^{-d/2}\exp(-\|z\|^2/2)$ denotes the $d$-dimensional standard Gaussian density. 

Let $\Pbdd{M,d}$ denote the family of distributions supported on the $d$-dimensional Euclidean ball $B_d(M)\triangleq \{\bfx\in \reals^d: \|\bfx\| \leq M\}$.
When $P\in \Pbdd{M,d}$ is a product measure, our one-dimensional results can be applied to approximate each coordinate separately.
For general $P\in \Pbdd{M,d}$, we need to extend Lemma~\ref{lem:mm1} to the multidimensional case by comparing \emph{moment tensors} \cite{Doss20}. Then, we establish the following upper bound, which improves the previous result  \cite[Lemma D.3]{Saha19} by bounding the $\chi^2$-divergence and deriving a tighter bound than the  $\exp(-\Theta_d(m^{\frac{1}{d}}))$ result in \cite[Lemma D.3]{Saha19} under the case $m^{\frac{1}{d}}\gtrsim_d M^2$.
\begin{proposition}
 \label{prop:highdim-ub}
  There exists $\kappa_d>0$ which only depends on $d$ such that for any $m\in \naturals$ and $M>0$:
  \begin{align}
  \label{eq:ub-bdd-hd}
\appr(m,\Pbdd{M,d},\chi^2)\leq
  \begin{cases}
      \exp\left(-\frac{m^{\frac{1}{d}}}{2}\log\frac{m^{\frac{1}{d}}}{M^2}\right), & m^{\frac{1}{d}}\geq \kappa_d M^2; \\ 
      \exp\pth{-\frac{\log\kappa_d }{42\kappa_d}\frac{m^{\frac{2}{d}}}{M^2} }, & 6\sqrt{3\kappa_d}M \leq m^{\frac{1}{d}}\leq  \kappa_d M^2. \\ 
  \end{cases}
\end{align}
\end{proposition}

The lower bound is derived following the roadmap in Section~\ref{sec:lb-general}. For the multidimensional regime, the trigonometric moment matrices are replaced by multilevel Toeplitz matrices \cite{TYRTYSHNIKOV1998multiToeplitz,Pestana2019}, and 
{both the wrapped density approach and the orthogonal expansion approach are} then applied to yield the following result. 
\begin{proposition}
 For any $m\in\naturals$,
    \label{prop:highD-lb}      
    \begin{equation}
    \label{eq:highD-lb-bdd-1} 
    \appr(m,\Unif([-M/\sqrt{d},M/\sqrt{d}]^d),\TV)\geq \frac{1}{2^{d+1}m}\exp\pth{-\frac{d^2\pi^2 m^{\frac{2}{d}}}{2M^2} }.
    \end{equation}
    {
    Furthermore, there exists a universal constant $C_1$ and $C^\prime$ such that when $m^{\frac{1}{d}}\geq C^\prime M^2/d$,
    \begin{equation}
    \label{eq:highD-lb-bdd-2} 
    \appr(m,\Pbdd{M,d},\TV)\geq \frac{1}{2^{d+1}m}\exp\pth{-C_1dm^{\frac{1}{d}}\log\pth{\frac{dm^{\frac{1}{d}}}{M^2}}}.
    \end{equation}
    }
 \end{proposition}
{Consequently, Propositions~\ref{prop:highdim-ub} and \ref{prop:highD-lb} give the tight approximation rate of $m$ and $M$ for the multi-dimensional distributions with compact support.}
The approximation rates for the distribution classes under tail conditions can be similarly extended. 
The additional results and proofs are presented in Appendix~\ref{app:highD-ub-lb}.

\paragraph{General mixtures.}
Our analytical framework extends beyond Gaussian mixtures.  
For the lower bounds, the convolution \( P * \nu \) is the distribution of \( X + Z \) with \( X \sim P \) and \( Z \sim \nu \) being independent. Analogous to Proposition~\ref{prop:TV-eigen}, we obtain 
\[
\inf_{P_m\in \calP_m} \TV(P_m*\nu,P*\nu) \geq
\sup_{\delta>0} \frac{\lambda_{\min}(\bfT_m(\delta X))}{2(m+1) \max_{k\in\mathbb{Z},|k|\leq m} |t_{k}(\delta Z)|}.
\] 
The lower bound for approximating \( P * \nu \) thus follows from bounding \( \lambda_{\min}(\mathbf{T}_m(\delta X)) \), as discussed in Section~\ref{sec:spectral}, together with the formula for \( t_k(\delta Z) \). 

The upper bounds can also be extended if the mixture can be identified by sufficiently many moments of its mixing distribution, akin to Lemma~\ref{lem:mm1}.
For instance, consider approximating the Poisson mixture 
$$\pi_{P}(y)\triangleq\int \operatorname{poi}(y,\lambda) \mathrm{d} P(\lambda)$$
by some finite mixture $\pi_{P_m}$, 
where $\operatorname{poi}(y,\lambda)= \frac{e^{-\lambda}\theta^y}{y!}$ and $P_m\in \calP_m$.  
As an Poisson analogue of Lemma~\ref{lem:mm1}, \cite[Theorem 3.3.4]{WY20Book} shows that for any $P,Q$ supported on $[0,M]$, if $m_j(P)=m_j(Q)$ for all $j\in[L]$, then
\[
\TV \left(\pi_{P}, \pi_{Q}\right) \leq\left(\frac{e M}{2 L}\right)^{L}.
\]
For $P\in\Pbdd{M}$ with $M\lesssim \log \frac{1}{\epsilon}$, 
there exists $P_m\in\calP_m$ with $m \leq O\pth{\frac{\log\frac{1}{\epsilon}}{\log\left(
1+ \frac{1}{M}{\log \frac{1}{\epsilon}} \right)}}$ such that $\TV \left(\pi_{P_m}, \pi_{P}\right) \leq \epsilon$. 
 For larger $M$, \cite[Lemma 6]{SW22} suggests that the result may be further improved by matching the moments of conditional distributions on quadratically partitioned subintervals, and the optimal rate remains open.

\appendix
\section{Preliminaries}
\subsection{Distances between probability measures}
\label{app:f-div}
The loss function $d(\cdot,\cdot)$ for measuring the approximation error is chosen among $f$-divergences,
including
\begin{itemize}
  \item $\chi^2$-divergence $\chi^{2}(P\|Q)=\int\frac{(p-q)^2}{q}  \diff \mu$;
  \item Total variation distance $\TV(P, Q)=\frac{1}{2}\int|p-q|\diff \mu$;
  \item Kullback-Leibler (KL) divergence $\KL(P, Q)=\int p \log \frac{p}{q} \diff \mu$;
  \item Squared Hellinger distance $H^2(P, Q)={\int(\sqrt{p}-\sqrt{q})^{2} \diff \mu}$,
\end{itemize}
where $p= \frac{\diff P}{\diff\mu}$ and 
$q= \frac{\diff Q}{\diff\mu}$ for some measure $\mu$ dominating $P$ and $Q$.
We refer the reader to  \cite[Chapter 7]{PW-it} for a comprehensive over of $f$-divergences, including  the following inequalities \cite[Sec.~7.6]{PW-it}:
\begin{lemma}
    \label{lem:f-divs}
  \begin{equation*}
     \frac{1}{2}H^2(P, Q)\leq \TV(P, Q) 
     \leq \sqrt{\frac{1}{2}D(P, Q)} \leq \sqrt{\frac{1}{2}\chi^{2}(P, Q)}, \end{equation*}
     \begin{equation*}
         \TV(P, Q)\leq H(P, Q) 
     \leq \sqrt{D(P, Q)}.
     \end{equation*}
\end{lemma}
We also need the following variational representation of  $f$-divergences (see, e.g, \cite[Sec.~7.13]{PW-it}):
\begin{lemma}
\label{lem:variation}
  For two probability measures $P$ and $Q$ supported on a space $\calX$, 
   \begin{align}
   \label{eq:vr-2}
   \chi^{2}(P \| Q)=&~\sup_{g:\calX \to \reals} \frac{\left(\mathbb{E}_{P}g -\mathbb{E}_{Q} g  \right)^{2}}{\var_{Q}[g ]},\\
    \TV(P, Q)=&~\frac{1}{2} \sup _{\|g\|_{\infty} \leq 1}\left|\mathbb{E}_{P} g -\mathbb{E}_{Q} g \right|,
      \label{eq:vr-1}
   \end{align} 
   where  $g: \calX\to\complex$ in \eqref{eq:vr-1} and    $\|g\|_{\infty}\triangleq\sup_{x\in\calX} |g(x)|$ denotes the $L_\infty$-norm.
\end{lemma}

\subsection{Gauss quadrature} 
\label{app:GQ}
Gauss quadrature is the best discrete approximation of a given distribution in the sense of moments; cf.~\cite[Section 3.6]{sb02}. Given a distribution $P$ supported on an interval $[a, b] \in\reals$, an $m$-point Gauss quadrature is an
$m$-atomic distribution $P_m\in\calP_m$, also supported on $[a, b]$, such that, for any polynomial $g$ of
degree no more than $2m-1$,
\begin{equation}
\label{eq:GQ}
\Expect_P\qth{g(X)}=\Expect_{P_m}\qth{g(X)}.
\end{equation}
For  a basic algorithm to compute the Gauss quadrature, see \cite{GW69}. Note that the generic result of Carathe\'odory's theorem \cite[Lemma 1]{GV01} provides a distribution that match the first $m$ orders of moments,  which is improved by the Gauss quadrature by a factor of two. 

{Algorithm~\ref{algo:quad} implements the Gauss quadrature rule using the Golub-Welsch method \cite{GW69}, with the moment vector as input. For further implementation details on truncated Gaussian distributions, see also \cite[Sec.~3.10]{burkardt2014truncated}.}

\setcounter{algorithm}{1}
\begin{algorithm}
\caption{Quadrature Rule}
\label{algo:quad}
\begin{algorithmic}
\STATE \textbf{Input:} A valid moment vector $\bfm=(m_1, \ldots, m_{2k-1})$
\STATE \textbf{Output:} A $k$-atomic distribution $P=\sum_{i=1}^k w_k\delta_{x_k}$.
\STATE Define the following degree-$k$ polynomial $P(x)$:
\[
P(x) = \det\begin{bmatrix}
1 & m_1 & \cdots & m_k \\
\vdots & \vdots & \ddots & \vdots \\
m_{k-1} & m_k & \cdots & m_{2k-1} \\
1 & x & \cdots & x^k
\end{bmatrix}
\]
\STATE Let the nodes $\{x_1, \ldots, x_k\}$ be the roots of the polynomial $P(x)$.
\STATE Let the weights $w = \{w_1, \ldots, w_k\}$ be defined as:
\[
w = \begin{bmatrix}
1 & 1 & \cdots & 1 \\
x_1 & x_2 & \cdots & x_k \\
\vdots & \vdots & \ddots & \vdots \\
x_1^{k-1} & x_2^{k-1} & \cdots & x_k^{k-1}
\end{bmatrix}^{-1}
\begin{bmatrix}
1 \\
m_1 \\
\vdots \\
m_{k-1}
\end{bmatrix}
\]
\end{algorithmic}
\end{algorithm}

\subsection{Orthogonal polynomials}
In this section, we give a brief introduction to certain classes of orthogonal polynomials that will be applied to prove the results in Appendix~\ref{app:lb}.
\label{app:ortho}
\begin{itemize}
  \item (Probabilists') Hermite polynomials 
  \begin{equation}
  \label{eq:hermite}
    H_{n}(x)=r ! \sum_{j=0}^{\lfloor n / 2\rfloor} \frac{(-1 / 2)^{j}}{j !(n-2 j) !} x^{n-2 j}.
  \end{equation}
  For $Z\sim N(0,1)$, it holds that 
      $\Expect[H_j(Z)H_k(Z)] =k! \delta_{jk}$ and
  $\Expect[H_k(x+Z)] = x^k$.
  \item Scaled Legendre polynomials
  \begin{equation}
  \label{eq:legendre}
    L_{n}(x)=\frac{\sqrt{2n+1}}{2^{n}} \sum_{k=0}^{\left\lfloor\frac{n}{2}\right\rfloor}(-1)^{k}\binom{n}{k} 
\binom{2n-2k}{n}  x^{n-2 k}
  \end{equation}
  with $\Expect[L_j(U)L_k(U)] =\delta_{jk}$ for $U\sim \Unif[-1,1]$.
\item Rogers-Szeg\"o polynomials \cite[Example 1.6.5 and Eq.~(1.6.51)]{BarrySimon2005a} 
  \begin{equation}
  \label{eq:rogers}
    R_n(x)=\frac{1}{\sqrt{(q)_n}}\sum_{j=0}^n(-1)^{n-j}
    \qbinom{n}{j}_q q^{(n-j) / 2} x^j
  \end{equation}
  which 
are orthogonal with respect to the ``wrapped Gaussian'' weight on the unit circle, namely
  $\Expect\left[R_k(e^{ia Z})\overline{R_j(e^{iaZ})}\right]=\delta_{jk}$, where $a^2=-\log q$ and $Z\sim N(0,1)$.
  Here in \eqref{eq:rogers}, 
  $(q)_{n}\triangleq (1-q)\left(1-q^{2}\right) \ldots\left(1-q^{n}\right),(0)_{n} \triangleq 1$,
and $\qbinom{n}{j}_q\triangleq\frac{(q)_{n}}{(q)_{j}(q)_{n-j}}$ is known as the $q$-binomial coefficient. By definition, $(q)_n$ is decreasing and converges to $(q)_\infty \triangleq \prod_{n=1}^\infty (1-q^n)$ as $n\to\infty$. $(q)_\infty$ is referred to as the \textit{Euler function}.
 \end{itemize}

\section{Proofs of lower bounds via orthogonal expansion}
\label{app:lb}
\subsection{Lower bound for Gaussian distributions}
\label{app:lb--subg}
 As mentioned in Remark~\ref{rem:lb-1}, the orthogonal expansion approach is applicable if the explicit expressions of the orthogonal polynomials are available. Specially, for the case where the mixing distribution is Gaussian, we state the result in Theorem~\ref{thm:app-lb-gauss}. Compared with the lower bound \eqref{eq:lb-gaussian},
 it recovers the order $\Theta(m/\sigma)$ in the exponent with a larger multiplicative constant.

The following lemma upper bounds the Euler function $(q)_{\infty}$, which is helpful to control the norm of the polynomial coefficient matrix.
\begin{lemma}
\label{lem:euler}
    Let $q\in (0,1)$ and $\exp(-t)=q$. For any $t_0>0$ and $t\in(0,t_0]$, there exists a constant $C$ which depends on $t_0$ such that
    \begin{equation*}
        \frac{1}{(q)_{\infty}}\leq C \exp \left(\frac{\pi^{2}}{6 t}\right).
    \end{equation*}
\end{lemma}
\begin{proof}
    Applying the asymptotic result in \cite[p. 57]{watson1936}, as $q\to 1_-$, 
    $$\frac{1}{(q)_{\infty}}= \sqrt{\frac{t}{2 \pi}} \exp \left(\frac{\pi^{2}}{6 t}-\frac{t}{24}\right)+o(1).$$
    Then, there exists a small $\delta_0>0$ such that $\frac{1}{(q)_{\infty}}\leq\exp \left(\frac{\pi^{2}}{6 t}\right)$ for $t\in (0,\delta_0)$, and the ratio $\frac{1}{(q)_{\infty}\exp \left(\frac{\pi^{2}}{6 t}\right)}$ is uniformly bounded in $[\delta_0,t_0]$. Hence, the desired result follows.
\end{proof}
\begin{theorem}
\label{thm:app-lb-gauss}
    Suppose that $m\geq 2\sigma$. There exists some universal constant $c$ such that 
    \label{thm:lb-1}
    \begin{equation*}
\appr(m,N(0,\sigma^2),\TV) \geq cm^{-3} \exp\left(-\left(4+\frac{\pi^2}{24}\right)\frac{m}{\sigma}\right).
  \end{equation*}
\end{theorem}
\begin{proof}
Let $\{R_n\}$ be the orthonormal Rogers-Szeg\"o polynomial \eqref{eq:rogers} with $a=\delta \sigma$ satisfying that 
 $\Expect\left[R_k(e^{i\delta \sigma Z})\overline{R_j(e^{i\delta \sigma Z})}\right]=\delta_{jk}$.
Let $\bfR=(R_{nj})_{n,j=0}^m$ be the associated coefficient matrix of $\{\varphi_n\}$ as defined in Section~\ref{sec:lb-general}.
To upper bound $\|\bfR\|_F$, we apply the explicit form of $R_n$ in \eqref{eq:rogers}, \ie,
\begin{equation}
    \label{eq:app4-8}
    R_n(w)=\frac{1}{\sqrt{(q)_n}}\sum_{j=0}^n(-1)^{n-j}
 \qbinom{n}{j}_q q^{(n-j) / 2} w^j
\end{equation}
where $q=\exp(-\delta^2\sigma^2)$. 

Set $\delta=\sqrt{\frac{4}{m\sigma}}$, with $q=\exp(-\frac{4\sigma}{m})$ and $t\triangleq-\log(q)=\frac{4\sigma}{m}$.
By our assumption, $t\in (0,2]$, $q\in [e^{-2},1)$. In this region, applying Lemma~\ref{lem:euler}, there exists a universal constant $C$ such that
\begin{equation*}
        \frac{1}{(q)_{\infty}}\leq C \exp \left(\frac{\pi^{2}}{6 t}\right)
        = C \exp \left(\frac{\pi^{2}m}{24\sigma }\right).
    \end{equation*}
By the $q$-binomial theorem \cite[17.2.35, 17.2.39]{NIST:DLMF}, for any $q\in[0,1]$, we have
\begin{equation*}
  \sum_{j=0}^{n} q^{j}\qbinom{n}{j}_{q^{2}}=\prod_{k=1}^n(1+q^k)
\leq\exp\left(\sum_{k=0}^\infty q^k\right)
= \exp\left(\frac{1}{1-q}\right).  
\end{equation*}
It follows from \eqref{eq:app4-8} that
\begin{align*}
    \|\bfR \|_F &\leq \sum_{n=0}^m\sum_{j=0}^n |R_{nj}|
    \leq \sum_{n=0}^m \frac{1}{\sqrt{(q)_n}}\sum_{j=0}^n \qbinom{n}{j}_q q^{ j/ 2}
    \leq \sum_{n=0}^m \frac{1}{\sqrt{(q)_n}}\exp\left(\frac{1}{1-\sqrt{q}}\right).
\end{align*}
Set $s=\sigma/m$ in $1-s- \exp(-2s)>0$ for $s\in[0,1/2]$, we have
$\frac{1}{1-\sqrt{q}}=\frac{1}{1-\exp(-2s)}\leq \frac{1}{s}=\frac{m}{\sigma}$.
Then
\begin{equation}
\label{eq:app4-9}
    \|\bfR \|_F^2\leq m^2\frac{\exp(2m/\sigma)}{(q)_\infty}
\leq Cm^2 \exp\left(\left(2+\frac{\pi^2}{24}\right)\frac{m}{\sigma}\right).
\end{equation}
Consequently, applying Propositions~\ref{prop:TV-eigen} and~\ref{prop:ortho} with \eqref{eq:app4-9}, we obtain
\begin{equation*}
\appr(m,N(0,\sigma^2),\TV)
    \geq 
    cm^{-3} \exp\left(-\left(4+\frac{\pi^2}{24}\right)\frac{m}{\sigma}\right)
\end{equation*}
for some universal constant $c$.
\end{proof}

\subsection{Completing the proof of Theorem~\ref{thm:lb-unif}}
\label{app:lb-unif-1}
We prove the lower bound \eqref{eq:lb-unif-1} for approximating uniform mixing distribution, and thereby complete the proof of Theorem~\ref{thm:lb-unif}. 
The framework resembles the approach in Section~\ref{sec:lb-general}, while we consider the $\chi^2$ divergence instead of working with the $\TV$ distance directly. 
First, we show that it suffices to derive the same lower bound (within a constant in the exponent) under the $\chi^2$-divergence.
Next, we relate the 
 $\chi^2$ variational lower bound to the spectrum of some weighted moment matrix, and then apply orthogonal expansion to bound the corresponding coefficient matrix as carried out in Proposition~\ref{prop:ortho}.  
 
 The following lemma shows a converse $f$-divergence inequality between two Gaussian mixtures, with respect to Lemma~\ref{lem:f-divs}.
 \begin{lemma}[{\cite[Theorem 21]{JPW23}}]
     \label{lem:H2-chi2}
    For any $P,Q\in \Pbdd{a}$ with $a\geq 2$, 
     \begin{equation*}
        \chi^{2}\left(f_{P} \| f_{Q}\right) \leq 2 \exp \left(50a^{2}\right) H^{2}\left(f_{P}, f_{Q}\right).
     \end{equation*}
 \end{lemma}
 
\begin{proposition}
    \label{prop:chi2tv}
     There exist universal positive constants $\zeta, \eta$, and $C^\prime$ such that the following holds.
     If $m\geq C^\prime (M\vee 1)^2$, then for any $P\in\Pbdd{M}$,
 \begin{equation}
     \label{eq:chi2tv}
     \appr(m,P,\TV)\geq \eta \appr(m,P,\chi^2)^{\zeta}.
 \end{equation}
\end{proposition}

\begin{proof}
    Suppose that $Q\in\calP_m$ satisfying $\TV(f_P,f_Q)= \appr(m,P,\TV)\triangleq \epsilon$. Let $Q_a$ be the conditional version of $Q$ on $I_a=[-a,a]$. Suppose that $a \geq 2(M\vee 1)$, and denote $U\sim Q$,  $V\sim P$. Then 
   \begin{equation}
       \label{eq:app4-13}
    \TV\left(f_{Q_a},f_P\right)\leq \TV\left(f_P, f_{Q}\right)+\TV\left(f_Q , f_{Q_a}\right) \leq \epsilon+\pbb\qth{|U|\geq a}.
   \end{equation}
    Since $P,Q_a\in\Pbdd{a}$, by Lemmas~\ref{lem:f-divs} and~\ref{lem:H2-chi2}, 
    \begin{equation}
    \label{eq:app4-14}
        \TV\left(f_{Q_a},f_P\right)\geq  \frac{1}{2}H^2\left(f_{Q_a},f_P\right)\geq \frac{e^{-50 a^2}}{4}\chi^2\left(f_{Q_a} \| f_P\right).
    \end{equation}
    Also, we have
    \begin{align*}
        \epsilon\geq\TV(f_P,f_{Q})&\geq \sup_{t} \pbb\qth{U+Z\geq t}-\pbb\qth{V+Z\geq t}\\
        &=\sup_{t}\Expect\qth{\Phi_c(t-U)}-\Expect\qth{\Phi_c(t-V)}\\
        &\geq\sup_{t\geq a}\pbb\qth{U\geq a}\Phi_c(t-a)-\Phi_c(t-M), 
    \end{align*}
     where the last inequality follows from Markov inequality and  $|V|\leq M$. Then, 
     \begin{equation}
         \label{eq:eq4-15}
         \pbb\qth{U\geq a}\leq\inf_{t\geq a}\frac{\epsilon+\Phi_c(t-M)}{\Phi_c(t-a)}.
     \end{equation}
     Likewise, the same upper bound holds for $\pbb\qth{U\leq -a}$. It follows from
     \eqref{eq:app4-13}, \eqref{eq:app4-14} and \eqref{eq:eq4-15} that, for $c_0=50$ and any $a\geq2(M\vee 1)$,
     \begin{align*}
         \chi^2\left(f_{Q_a} \| f_P\right)&\leq 4e^{c_0a^2}\TV\left(f_{Q_a},f_P\right)\\
         &\leq 4e^{c_0a^2} (\epsilon+\pbb\qth{|U|\geq a})\\
         &\leq 4e^{c_0a^2}\pth{ \epsilon+2\inf_{t\geq a}\frac{\epsilon+\Phi_c(t-M)}{\Phi_c(t-a)}}.
     \end{align*}
    By applying the Gaussian tail bound $\frac{t \phi(t)}{1+t^2}  \leq \Phi_c(t) \leq \frac{\phi(t)}{t}$ and setting $t=(2c_0+1) a$ for $a\geq 2(M \vee 1)$, we have
    \begin{align*}
        \frac{\epsilon+\Phi_c(t-M)}{\Phi_c(t-a)}
        &\leq \frac{1+(t-a)^2}{t-a}\frac{\epsilon}{\phi(t-a)}+\frac{1+(t-a)^2}{(t-a)(t-M)}\frac{\phi(t-M)}{\phi(t-a)}\\&\leq 4c_0a\frac{\epsilon}{\phi(2c_0a)} +2\frac{\phi((2c_0+0.5)a)}{\phi(2c_0a)}\\
        &\leq 4\sqrt{2\pi}c_0\epsilon e^{(2c_0^2+1)a^2} +2e^{-(c_0+1/8)a^2}.
    \end{align*}
Then, with constants 
$c_1=4+32\sqrt{2\pi}c_0$, $c_2=2c_0^2+c_0+1$, and $c_3=1/8$,
\begin{equation}
     \label{eq:eq4-17}
         \chi^2\left(f_{Q_a} \| f_P\right)\leq \inf_{a\geq 2(M\vee 1)} c_1(\epsilon e^{c_2a^2} + e^{- c_3a^2}).
     \end{equation}
By Theorem~\ref{thm:ub-bdd}, there exists a universal constant $C^\prime$ such that, for any $m\geq C^\prime (M\vee 1)^2$, 
$$\epsilon\leq  \exp(-(c_2+c_3)[2(M\vee 1)]^2).$$
Choose $a$ such that $\epsilon=e^{-(c_2+c_3)a^2}$ hold, and it follows that $a \geq 2(M\vee 1)$.
Then, \eqref{eq:eq4-17} implies
$$\chi^2\left(f_{Q_a} \| f_P\right)\leq 2c_1\epsilon^{\frac{c_3}{c_2+c_3}}=2c_1\appr(m,P,\TV)^{\frac{c_3}{c_2+c_3}}.$$
Since $Q_a\in\calP_m$ by definition, we obtain the desired result \eqref{eq:chi2tv} with $\zeta=\frac{c_2+c_3}{c_3}$ and $\eta=\frac{1}{(2c_1)^\zeta}$.
\end{proof}
\begin{theorem}
  \label{thm:lb-unif-1}
  There exists universal constants $C_1$ and $C^\prime$ such that the following holds. 
  When $m\geq C^\prime (M\vee 1)^2$,
\begin{equation*}
    \appr(m,\Unif[-M,M],\TV)\geq
    \exp\pth{-C_1m\log\pth{\frac{m}{M^2}}}.
\end{equation*}
\end{theorem}

\begin{proof}
We first prove a $\chi^2$ lower bound by applying the variational representation \eqref{eq:vr-2}, and finally upgrade it to the total variation
via Proposition~\ref{prop:chi2tv}.
We choose the Hermite polynomials $H_k$ in \eqref{eq:hermite} as the test functions. Denote $X\sim P=\Unif[-M,M]$, $Y=X+Z$, and $Y_m=X_m+Z$ for some $X_m\sim P_m \in \calP_m$. 
Then,
\begin{equation*}
\Expect\left[H_k(Y)\right]=\Expect\left[\Expect[H_k(X+Z)|X]\right]=\Expect X^k.
\end{equation*}
Likewise, $\Expect\left[H_k(Y_m)\right]=\Expect X_m^k$.
For $m\ge M^2$ and $k\in[2m]$, by \cite[7.374.7]{IntgTable}, we have
    \begin{align*}
    \var[H_k(Y)]
    \leq \Expect[H_k^2(Y)]
    =k!\sum_{l=0}^k \binom{k}{k-l} \frac{\Expect[X^{2l}]}{l!}
    \leq k!\sum_{l=0}^k 2^{k} \frac{M^{2l}}{l!}
    \leq k!(4e)^m.
    \end{align*}
Applying \eqref{eq:vr-2} yields
\begin{align}
\chi^{2}\left(f_{P_{m}} \| f_{P}\right)&\geq \sup _{k\in\naturals} 
\frac{\left(\mathbb{E}[X^k]-\mathbb{E}[X_m^k]\right)^{2}}{\var[H_k(Y)]}\\ &\geq 
\frac{1}{2m}\sum_{k=1}^{2m} \frac{\left(\mathbb{E}[X^{k}]-\mathbb{E}[X_m^{k}]\right)^{2}}{k!(4e)^m}\\
&= \frac{1}{2m(4e)^m}\sum_{k=1}^{2m} M^{2k} \frac{\left(\mathbb{E}[U^{k}]-\mathbb{E}[U_m^{k}]\right)^{2}}{k!}\label{eq:chi2-weighted}
\end{align}
where $U\sim \Unif[-1,1]$ and $U_m=\frac{X_m}{M}$.

We lower bound~\eqref{eq:chi2-weighted} by analyzing weighted moment matrix. Specifically, for $i,j\in\naturals$ satisfying $i+j=k$,
\begin{align*}
M^{2k} \frac{\left(\mathbb{E}[U^{k}]-\mathbb{E}[U_m^{k}]\right)^{2}}{k!}
&\ge \frac{M^{2(i+j)}}{2^{i+j}i!j!}\left(\mathbb{E}[U^{k}]-\mathbb{E}[U_m^{k}]\right)^{2}\\
&= \left(V_{ij}(U)-V_{ij}(U_m)\right)^{2},
\end{align*}
where 
$C_i\triangleq \frac{M^i}{\sqrt{2^i i!}}$ and $V_{ij}(U)\triangleq \Expect[(C_i U^i) (C_j U^j)]$.
The weighted moment matrix $\bfV_m(U)$ of order $m+1$ consists of $V_{ij}(U)$ in its $(i,j)\Th$ entry for $i,j=0,\dots,m$.
Let $\bfC_m=\operatorname{Diag}(C_{0},C_{1},\dots,C_{m})$.
Then $\bfV_m(U)=  \bfC_m \bfM_m(U) \bfC_m$. 
Notably, since $U_m$ is $m$-atomic, we also have $\operatorname{rank}(\bfV_m(U_m))\leq m$.
By Lemma~{\ref{lem:eym}}, we have 
\begin{align}
  \sum_{k=1}^{2m} M^{2k} \frac{\left(\mathbb{E}[U^{k}]-\mathbb{E}[U_m^{k}]\right)^{2}}{k!}
&\geq\frac{\|\bfV_m(U)-\bfV_m(U_m)\|^2_F}{m+1} \geq \frac{\lambda_{\min}^2(\bfV_m(U))}{m+1}.\label{eq:sum-lambda-min}
\end{align}
It remains to bound the smallest eigenvalue 
of $\bfV=\bfV_m(U)$, given by
\begin{equation}
        \lambda_{\min}(\bfV)=\min_{\bfx\in\reals^{m+1}\setminus\{\mathbf{0}\}} \frac{ \bfx^\top \bfV \bfx}{\|\bfx\|^2}.
    \end{equation}
Let $L_k$ be the orthonormal Legendre polynomial defined in \eqref{eq:legendre}, and $\pi_m(t)=\sum_{j=0}^m C_jx_j t^j$, with an expansion $\pi_m(t)=\sum_{k=0}^m c_kL_k(t)$.
Denote $\mathbf{c}=(c_0,\dots,c_m)^\top$
and  $\mathbf{t}=(1,t,\dots,t^m)^\top$. 
Then, by a similar derivation of Proposition~\ref{prop:ortho}, we have 
\begin{equation*}
\|\bfc\|_2^2=\Expect\left[\pi_m^2(U)\right] = \bfx^\top \bfV \bfx.
\end{equation*}

Let $\mathbf{L}_m=(L_{nj})_{n,j=0}^m$ be 
the associated coefficient matrix of $\{L_k\}$ of order $m+1$.
By definition,
\begin{equation*}
    \bfx^\top \bfC_m^\top \bft=  \pi_m(t)=
    \sum_{k=0}^m c_kL_k(t)= \bfc^\top   \bfL_m \bft, \quad \forall t,
\end{equation*}
which implies that $\mathbf{x} = \bfC_m^{-1}\mathbf{L}_m^\top \mathbf{c}$. Hence 
\begin{equation}
    \lambda_{\min}(\bfV)=\min_{\mathbf{c}\neq 0} \frac{\|\mathbf{c}\|_2^2}{\mathbf{c}^\top \mathbf{L}_m \bfC_m^{-2}\mathbf{L}_m^\top \mathbf{c}} 
    \geq \frac{1}{\|\mathbf{L}_m\bfC_m^{-1}\|_F^2}.\label{eq:lambda-min-coeffs}
\end{equation}

Applying the explicit formula of $L_m$ in \eqref{eq:legendre}, we have
\begin{align}
\|\mathbf{L}_m\bfC_m^{-1}\|^2_F 
\leq &\sum_{n=0}^m\sum_{k=0}^{\left\lfloor\frac{n}{2}\right\rfloor} \frac{2n+1}{2^{2n}}
   \frac{[(2n-2k)!]^2}{(k!)^2(n-2k)![(n-k)!]^2} \left(\frac{2}{M^2}\right)^{n-2k}\nonumber\\
\leq &(2m+1) \sum_{n=0}^m\sum_{k=0}^{\left\lfloor\frac{n}{2}\right\rfloor} \frac{1}{2^{2n}}
2^{4(n-k)}\frac{n!}{(k!)^2}\left(\frac{2}{M^2}\right)^{n-2k}\nonumber\\
= &(2m+1) \sum_{n=0}^m n!{\left(\frac{8}{M^2}\right)^n}\sum_{k=0}^{\left\lfloor\frac{n}{2}\right\rfloor} \left(\frac{(M^2/8)^k}{k!}\right)^2\nonumber\\
\leq& 3m^2 m! {\left(\frac{8}{M^2}\right)^m} \left(\sum_{k=0}^{\infty} \frac{(m/8)^k}{k!}\right)^2\nonumber\\ 
\leq&  Cm^{2.5}{\left(\frac{8m}{eM^2}\right)^m}\exp\pth{\frac{m}{4}}\label{eq:norm-weighted-coeffs}
\end{align}
for some constant $C$ when $m\geq M^2$. 
Combining \eqref{eq:chi2-weighted}, \eqref{eq:sum-lambda-min}, \eqref{eq:lambda-min-coeffs}, and \eqref{eq:norm-weighted-coeffs}, we obtain that
\begin{align*}
\appr(m,P,\chi^2)&\geq
\frac{1}{2m(4e)^m}\sup_{P_m\in\calP_m}\sum_{k=1}^{2m} M^{2k} \frac{\left(\mathbb{E}[U^{k}]-\mathbb{E}[U_m^{k}]\right)^{2}}{k!}\\
&\geq \frac{1}{Cm^{7}} \pth{\frac{e}{256}}^m\exp\left[-2m\log\left(\frac{m}{M^2}\right)-\frac{m}{2}\right]\\
&\geq c \exp\left[-2m\log\left(\frac{m}{M^2}\right)-6m\right] 
\end{align*}
for some small constant 
$c$.

Finally, by Proposition~\ref{prop:chi2tv}, there exists  universal constants $C_1$ and $C^\prime$ such that
\begin{equation*}
    \appr(m,P,\TV)\geq
    \exp\pth{-C_1m\log\pth{\frac{m}{M^2}}}, \quad m\geq C^\prime M^2. \qedhere
\end{equation*}
\end{proof}

\subsection{Proof of Proposition~\ref{prop:lb-unif-arc}}
\label{app:lb-unif-arc}
Given a distribution $P$ with density $f$, let $\{\varPhi_n\}$ be the associated set of \emph{monic} orthogonal polynomials on the unit circle introduced in \eqref{eq:opuc}. Then the orthonormal version is given by $\varphi_n(z)=\kappa_n\varPhi_n(z)\triangleq \sum_{j=0}^n R_{nj}z^j$ with $R_{nn}=\kappa_n$.
Then, $\bfR_m\triangleq(R_{nj})_{n,j=0}^m$ is the coefficient matrix of $\{\varphi_{n}\}$ of order $m+1$. 
The Frobenius norm of $\bfR_m$ is upper bounded by the following lemma:
\begin{lemma}
    \label{lem:arc-coef}
    For any distribution $P$ on $\reals$, under the notations above,
    \begin{equation*}
        \|\bfR_m\|_F^2\leq\sum_{n=0}^{m} 2^{2n} \kappa_n^2.
    \end{equation*}
\end{lemma}
\begin{proof}
     Let $a_{n}=-\overline{\varPhi_{n+1}(0)}, n\geq 0$. \cite[Eq.~(1.5.22)]{BarrySimon2005a} shows that $\kappa_n=\prod_{j=0}^{n-1}(1-\left|a_{j}\right|^2)^{-\frac{1}{2}}$.
     Then, by \cite[Eq.~(1.5.30)]{BarrySimon2005a}, it holds for all $|z|=1$ that
\begin{equation*}
    \left|\varphi_{n}(z)\right| 
    \leq \prod_{j=0}^{n-1}  \frac{2}{\sqrt{1-\left|a_{j}\right|^2}} \left|\varphi_{0}(z)\right| 
    = 2^n \kappa_n, 
    \quad n\in\naturals.
\end{equation*}
It follows that 
\begin{equation*}
    \sum_{j=0}^{n}\left|R_{nj}\right|^{2}=\frac{1}{2 \pi} \oint_{|z|=1}|\varphi_n(z)|^{2} \diff z\leq 2^{2n} \kappa_n^2.
\end{equation*}
Finally, the result follows from
$\|\bfR_m\|_F^2=\sum_{n=0}^{m}\sum_{j=0}^{n} \left|R_{nj}\right|^{2}$.
\end{proof}

It remains to construct a distribution supported on an arc of the unit circle and control the coefficients $\kappa_n$ in Lemma \ref{lem:arc-coef}.
To this end, we apply the following lemma \cite[Lemma 2.1]{Krasovsky2003} connecting the orthogonal system on an arc of the unit circle to the orthogonal system on an interval of the real line.  
Let $w$ be a positive,  \textit{even} weight function on $[-1, 1]$ normalized to $\int w(x) dx=1$, and $\{p_n\}$ be the set of monic orthogonal polynomials associated with $w$ with 
\[
    \int p_j(x)p_k(x) w(x) \diff x = h_j \indc{j=k}.    
\]
Let $f(\theta) =\frac{1}{2\gamma} w\left(\gamma^{-1} \cos \frac{\theta}{2}\right) \sin \frac{\theta}{2}$ if $\alpha \leq \theta \leq 2\pi-\alpha$, and $f(\theta) = 0$ otherwise, where $\gamma = \cos \frac{\alpha}{2}$ and $\alpha\in[0,\pi]$.
Since $f\geq 0$ and $\int f(x)\diff x=\int w(x)\diff x=1$, $f$ is a density.
Let $\{\varPhi_n\}$ be the set of monic orthogonal polynomials on the unit circle associated with $f$. 
The following statement holds:
\begin{lemma}
\label{lem:arc}    
Set $x = \gamma^{-1} 
\cos \frac{\theta}{2} = (2\gamma)^{-1}(z^{1/2} + z^{-1/2})$ with $z = e^{i\theta}$.
Then,
\begin{align*}
    \varPhi_{n}(z)&= \frac{(2 \gamma)^{n+1} z^{n / 2}}{z-1}\left(z^{1 / 2} p_{n+1}(x)-\frac{p_{n+1}\left(\gamma^{-1}\right)}{p_{n}\left(\gamma^{-1}\right)} p_{n}(x)\right) , \\
    \kappa_{n}^{-2}&=  2^{2n} \gamma^{2n+1} \frac{p_{n+1}\left(\gamma^{-1}\right)}{p_{n}\left(\gamma^{-1}\right)} h_n.
\end{align*}
\end{lemma}
Now, with Lemmas~\ref{lem:arc-coef} and~\ref{lem:arc}, we are ready to show the  lower bound for $\Pbdd{M}$:
\begin{proof}[Proof of Proposition~\ref{prop:lb-unif-arc}]
Let $b\in (0,\pi) $ and $\gamma=\sin\frac{b}{2}$. 
Recall that $\tilde{P}$ is defined as the distribution supported on $[-b,b]$ with a density function  
\begin{equation*}
    \tilde{f}(\theta)=
    \begin{cases}
    \frac{\sqrt{\gamma^2-\sin^2(\theta/2)}\cos(\theta/2)}{\pi \gamma^2},& \theta\in[-b,b],\\
    0,& \text{otherwise.}
\end{cases}
\end{equation*}
Let $\delta=b/M$ and $P\in\Pbdd{M}$ be the distribution of $X/\delta$ for $X\sim \tilde P$. 
To apply Proposition~\ref{prop:ortho}, it remains to evaluate the coefficients of the orthonormal polynomials associated with $\tilde P$, which can be related to  Chebyshev polynomials using Lemma~\ref{lem:arc}.
To this end, let $w(x)=\frac{2}{\pi}\sqrt{1-x^2}$ for $x\in[-1,1]$, and define the translated distribution $Q(\theta)\triangleq \tilde{P}(\theta-\pi)$ for $\theta\in [\pi-b,\pi+b]$, with density $\tilde{f}(\theta-\pi) = \frac{1}{2\gamma}w\left(\gamma^{-1} \cos \frac{\theta}{2}\right) \sin \frac{\theta}{2}$.
Let $U_n$ denote the degree-$n$ Chebyshev polynomial of the second kind \cite[Eq. (1.12.3)]{SzegoOrthopolys} and $p_n(x)=\frac{1}{2^n}U_{n}(x)$. 
Then, $p_n$ is monic with 
\begin{equation*}
    \int_{-1}^1 p_n(x) p_m(x) w(x) \diff x= h_n\delta_{mn},
    \quad h_n=\frac{1}{2^{2n}}.
\end{equation*}
Let $\{\Psi_n\}$ and $\{\tilde{\varPhi}_n\}$ be the sets of monic orthogonal polynomials on the unit circle associated with $Q$ and $\tilde{P}$ as defined in \eqref{eq:opuc}, and 
$\psi_n=\kappa_n\Psi_n$,
$\tilde{\varphi}_n= \tilde{\kappa}_n\tilde{\varPhi}_n$ be the orthonormal polynomials, respectively. By definition, $\tilde{\varPhi}_n(z)=(-1)^n\Psi_n(-z)$ and $\tilde{\kappa}_n=\kappa_n$. 
Let $\mathbf{\tilde{R}}_m$ be the coefficient matrix of $\{\tilde{\varPhi}_n\}$ of order $m+1$. By Lemmas~\ref{lem:arc-coef} and~\ref{lem:arc}, we obtain that
\begin{equation*}
\|\mathbf{\tilde{R}}_m\|_F^2\leq \sum_{n=0}^{m} 2^{2n} \tilde{\kappa}_{n}^2 = \sum_{n=0}^{m} 2^{2n} \kappa_n^2=\sum_{n=0}^{m}\frac{1}{\gamma^{2n+1}h_n}\frac{p_{n}\left(\gamma^{-1}\right)}{p_{n+1}\left(\gamma^{-1}\right)}.
\end{equation*}
Consider $r_n(x)\triangleq p_{n+1}(x)/p_n(x)$. By definition, $r_0(x)=x$. It follows from the recurrence relation $U_{n+1}(x)=2 x U_{n}(x)-U_{n-1}(x)$ \cite[Sec.4.5.]{SzegoOrthopolys} that 
\[
r_n(x)=x-\frac{1}{4r_{n-1}(x)}, \quad n\in\naturals.
\]
Note that $r_0(x)\ge x/2$ for $x\ge 0$, and if $r_{n-1}(x)\ge x/2$, then $r_{n}(x)\geq x-\frac{1}{2x}\geq \frac{x}{2}$ for $x\ge 1$. By induction, $r_n(x)\geq \frac{x}{2}$ for all $x\ge 1$ and $n\in\naturals$.
Since $\gamma^{-1}> 1$, we have
\begin{equation}
\label{eq:R_ub_unif}
    \|\tilde{\bfR}_m\|_F^2
    \leq  \sum_{n=0}^{m} 2^{2n+1}\gamma^{-2n}
    = 2\frac{(2/\gamma)^{2m+2}-1}{(2/\gamma)^{2}-1}.
\end{equation}
Assuming $m\geq \frac{eM^2}{2}$, choose $b={ \sqrt{\frac{M^2}{m}\log\frac{2m}{e M^2}}}$, which satisfies $b\leq \frac{\sqrt{2}}{e}$.
Then $\gamma=\sin\frac{b}{2}\in\qth{\frac{b}{4},\frac{b}{2}}$. 
Applying Propositions~\ref{prop:TV-eigen} and \ref{prop:ortho}, we obtain that
\begin{align*}
     \appr(m,P,\TV)&\geq 
     \frac{1}{2(m+1)\exp(m^2\delta^2/2) \|\tilde{\bfR}_m\|_F^2 }
     \geq
     \frac{(2/\gamma)^{2}-1}{8m[(2/\gamma)^{2}-(2/\gamma)^{-2m}]} \exp\pth{-\frac{m^2b^2}{2M^2}-2m\log\frac{2}{\gamma}}\\
     &\geq \frac{1-(b/4)^2}{8m} \exp\pth{-\frac{m^2b^2}{2M^2}-2m\log\frac{8}{b}}.
\end{align*}
We conclude the proof of~\eqref{eq:lb-unif-1} by plugging in the choice of $b$.
\end{proof}

\section{Proof of the approximation rates}
\label{app:rate}
\subsection{Proofs of main theorems}
\label{app:rate-main}
For convenience, we restate the dual formula of $\comp$ and $\appr$ in the following: 
\begin{equation}
    \label{eq:2-2}
\comp(\epsilon,\calP,d)=\inf\sth{m:\appr(m,\calP,d)\leq \epsilon},
\end{equation}
\begin{equation}
    \label{eq:2-3}
\appr(m,\calP,d)=\inf\sth{\epsilon:\comp(\epsilon,\calP,d)\leq m}. 
\end{equation}

\begin{proof}[Proof of Theorem~\ref{thm:main-bdd}] 
{For $d$ satisfying Assumption~\ref{as:f-div}}, denote $m=\comp(\epsilon,\Pbdd{M},d)$. Combining the result in Theorems~\ref{thm:ub-bdd} and~\ref{thm:lb-unif}, we have that $\log\frac{1}{\epsilon}\asymp m\log\frac{m}{M^2}$ when $m\geq C M^2$ for some constant $C$, and 
$\log\frac{1}{M^2\epsilon}\lesssim\frac{m^2}{M^2}\lesssim \log\frac{1}{\epsilon}$ when $3\sqrt{\kappa}M \leq m\leq   CM^2$. Also, the condition $M\leq \epsilon^{-c_1},c_1\in(0,\frac{1}{2})$ implies that $\log\frac{1}{M^2\epsilon}\asymp \log\frac{1}{\epsilon}$.
Following \eqref{eq:2-2}, we can rewrite the result as
\begin{equation*}
    m \asymp
      \begin{cases}
      \frac{\log\epsilon^{-1}}{\log\log\epsilon^{-1}-\log{M^2}}, & M\leq c\sqrt{\log\epsilon^{-1}};\\
      M\sqrt{\log\epsilon^{-1}}, &  c\sqrt{\log\epsilon^{-1}}<M\leq \epsilon^{-c_1},\\
  \end{cases}
\end{equation*}
 for some small constant $c$, which is equivalent to the desired result \eqref{eq:main01}.
 \end{proof}
\begin{proof}[Proof of Theorem~\ref{thm:main-tail}]
    {For $d$ satisfying Assumption~\ref{as:f-div}},
    denote $m=\comp(\epsilon,\calP_\alpha(\beta),d)$.
    First we prove the lower bound.
    By Theorem~\ref{thm:lb-tail-1} and Lemma~\ref{lem:f-divs},
    \begin{equation*}
        \frac{1}{\epsilon\beta^{\frac{4}{2+\alpha}}m^{\frac{2\alpha}{2+\alpha}}}
        \log \frac{1}{\epsilon\beta^{\frac{4}{2+\alpha}}m^{\frac{2\alpha}{2+\alpha}}} = 
        O_\alpha\pth{\frac{1}{\beta^2\epsilon}},
    \end{equation*}
    which implies that $m=\Omega_\alpha\qth{\beta \pth{\log \frac{1}{\beta^2\epsilon}}^{\frac{2+\alpha}{2\alpha}}}$.
    By the relation $ \beta \leq \epsilon^{-c_1}$ for $c_1\in(0,\frac{1}{2})$, we have 
    $m=\Omega_\alpha\qth{\beta \pth{\log \frac{1}{\epsilon}}^{\frac{2+\alpha}{2\alpha}}}$.

    Next, we show the upper bound by applying Theorem~\ref{thm:ub--subW}, where the condition $m\geq \Omega_\alpha(\beta)$ is satisfied by our lower bound. We have that
    \begin{equation*}
\log\frac{1}{\epsilon}\gtrsim_\alpha
  \begin{cases}
      m\log\frac{m^{\alpha-2}}{\beta^{2\alpha}}, & m^{\alpha-2}\gtrsim_\alpha \beta^{2\alpha};\\
      \pth{\frac{m}{\beta}}^{\frac{2\alpha}{2+\alpha}}, & m^{\alpha-2}\lesssim_\alpha \beta^{2\alpha},\\
  \end{cases}
\end{equation*}
    which by \eqref{eq:2-2} implies that
\begin{equation*}
m\lesssim_\alpha \frac{\log\frac{1}{\epsilon}}{\log\pth{1+\frac{1}{\beta}\pth{\log\frac{1}{\epsilon}}^{\frac{\alpha-2}{2\alpha}}}}.
\end{equation*}
Particularly, if $m^{\alpha-2}\beta^{-2\alpha}=O_\alpha(1)$ holds, we have 
    $m =\Theta_\alpha\pth{ \beta \pth{\log\frac{1}{\epsilon}}^{\frac{2+\alpha}{2\alpha}}}$.
\end{proof}
\subsection{Proofs for Section~\ref{sec:extension-moment}}
\label{app:main-moment}
The following two lemmas give upper and lower bounds for the approximation error, respectively:
\begin{proposition}
    \label{prop:ub-moment}
    Suppose that $m\geq (16e^3)^{\frac{9}{8\alpha}}\beta$ and $\beta\geq \frac{8\alpha}{e\log{16e^3}}$ hold. 
    Then, it holds for $d\in \{\TV,H^2,\KL,\chi^2\}$ that
    \begin{equation}
        \label{eq:ub-moment}
        \appr(m,\calM_\alpha(\beta),d) \leq O_\alpha\qth{\pth{\frac{\beta}{m}\sqrt{\log\frac{m}{\beta}}}^{\alpha}}.
    \end{equation}
\end{proposition}

\begin{proof}
Let $\kappa= 16e^3$ and $t=m \sqrt{\frac{\log\kappa }{4\kappa} \frac{1}{2\alpha\log(m/\beta)}}$, that satisfies
\begin{align*}
    3\sqrt{\kappa }t&~= m \sqrt{\frac{9\log\kappa }{8\alpha} \frac{1}{\log(m/\beta)}}\leq m; \\
    \kappa t^2&~= \kappa m \beta \frac{\log\kappa }{8\alpha\kappa} \frac{m/\beta}{\log(m/\beta)}\geq m\beta \frac{e\log \kappa}{8\alpha}\geq m.
\end{align*}
By Theorem~\ref{thm:ub-bdd}, we have  
\begin{equation*}
    \appr(m,\Pbdd{t},\chi^2)
    \le \exp\left(-\frac{\log\kappa }{4\kappa}\frac{m^2}{t^2}\right)=\pth{\frac{\beta}{m}}^{2\alpha}.
\end{equation*}
On the other hand, It follows from the tail bound \eqref{eq:orlicz_tail} that 
\begin{equation*}
    \sup_{P\in\calM_\alpha(\beta)}
    P(I_t^c)\leq 2\pth{\frac{\beta}{t}}^\alpha=2 \pth{\frac{\beta}{m}\sqrt{\frac{4\kappa}{\log\kappa} 2\alpha\log\frac{m}{\beta}}}^\alpha\asymp_\alpha\pth{{\frac{\beta}{m}\sqrt{\log\frac{m}{\beta}}}}^\alpha.
\end{equation*}
By Proposition~\ref{prop:ub-tail}, \eqref{eq:ub-moment} holds for $d=\chi^2$. For the total variation, we have
\begin{align*}
        \TV\left(f_{Q} , f_{P}\right)&\stepa{\leq} P(I_t)\TV\left(f_{Q} , f_{P_t}\right) +P(I_t^c) \TV\left(f_{Q} , f_{P_t^c}\right)\\
        &\stepb{\leq} \sqrt{\frac{1}{2}\chi^{2}\left(f_{Q} \| f_{P_t} \right)}  +P(I_t^c),
\end{align*}
where (a) applies Jensen's inequality, and (b) follows from Lemma~\ref{lem:f-divs}. Hence,
    $$\appr(m,\calP,\TV)\leq \inf_{t>0} \pth{\sqrt{\frac{1}{2}\appr(m,\Pbdd{t},\chi^2)}+\sup_{P\in\calP}P(I_t^c)},$$
and thus \eqref{eq:ub-moment} holds for $d=\TV$. Finally, \eqref{eq:ub-moment} holds for $d\in\{\TV,H^2,\KL,\chi^2\}$ by Lemma~\ref{lem:f-divs}.
\end{proof}

\begin{proposition}
    \label{thm:lb-tail-2}
    There exist $C_\alpha$ and $D_\alpha$ that only depend on $\alpha$ such that the following holds.
    If $m\geq D_\alpha\beta$,
    then, there exists $P\in \calM_\alpha(\beta)$ such that
    \begin{equation*}
        \appr(m,P,\TV) \geq \frac{C_\alpha}{m\log\frac{m}{\beta}}\pth{\frac{\beta}{m}}^{\alpha}.
    \end{equation*}
\end{proposition}  
\begin{proof}
Set $D_\alpha=\min\sth{x\geq e:\frac{1}{x}\pth{\frac{1}{\alpha\log x}}^{\frac{1}{\alpha}}\leq \frac{1}{3}\pth{\frac{1-\log^{-1} x}{1+\alpha}}^{\frac{1}{\alpha}}}$, which is a finite value since $\frac{1}{x}\pth{\frac{1}{\alpha\log x}}^{\frac{1}{\alpha}}\to 0$ and $\pth{\frac{1-\log^{-1} x}{1+\alpha}}^{\frac{1}{\alpha}}\to \pth{1+\alpha}^{-\frac{1}{\alpha}}>0$ as $x\to \infty$.
Suppose that $r= \frac{m}{\beta}\geq D_\alpha$.
Consider the truncated Pareto random variable $X \sim P$ with the density function 
\begin{equation}
    \label{eq:h_y}
    h_\alpha(x)= \frac{a}{x^{\alpha+1}}\indc{b\leq x \leq kb},
\end{equation}
where the parameters are set as
$k = r(\log r)^{\frac{1}{\alpha}}$, $a = \frac{\beta^\alpha}{\log k}$, and $b=[\alpha^{-1}a(1-k^{-\alpha})]^{\frac{1}{\alpha}}$.
Direct calculation shows that $h_\alpha$ is a well-defined density with $\|X\|_\alpha=\beta$, \ie, $P\in \calM_\alpha(\beta)$. 
Note that 
\begin{equation*}
a\in \qth{\frac{\alpha}{1+\alpha}\frac{\beta^\alpha}{\log r},\frac{\beta^\alpha}{\log r}}; \quad b\in\qth{\pth{\frac{1-\log^{-1} D_\alpha}{1+\alpha}}^{\frac{1}{\alpha}}\frac{\beta}{(\log r)^{1/\alpha}},\pth{\frac{1}{\alpha}}^{\frac{1}{\alpha}}\frac{\beta}{(\log r)^{1/\alpha}}}.
\end{equation*}
Choose $\delta=\frac{\pi D_\alpha (\alpha\log D_\alpha)^{1/\alpha}}{m}$.
Then, we have 
\begin{equation*}
    b \leq \pth{\frac{1}{\alpha}}^{\frac{1}{\alpha}}\frac{\beta}{(\log r)^{1/\alpha}}\leq \frac{m}{D_\alpha}\pth{\frac{1}{\alpha\log D_\alpha}}^{1/\alpha}=\frac{\pi}{\delta};
\end{equation*}
\begin{equation*}
kb\geq \pth{\frac{1-\log^{-1} D_\alpha}{1+\alpha}}^{\frac{1}{\alpha}}\frac{\beta r(\log r)^{\frac{1}{\alpha}}}{(\log r)^{1/\alpha}}= 
    m \pth{\frac{1-\log^{-1} D_\alpha}{1+\alpha}}^{\frac{1}{\alpha}} \geq  \frac{3m}{D_\alpha}\pth{\frac{1}{\alpha\log D_\alpha}}^{1/\alpha}=\frac{3\pi}{\delta}.
\end{equation*}
With $Y\triangleq \delta X \sim g$, we have 
\begin{align*}
    \appr(m,P,\TV)
    \stepa{\geq} \sup_{\delta>0} \frac{\pi\min_{0\leq\theta\leq 2\pi}g^\wrap(\theta)}{(m+1)\exp(m^2\delta^2/2)}
    \stepb{\geq} \frac{\pi h_{\alpha}\pth{\frac{3\pi}{\delta}}}{2m\delta\exp(m^2\delta^2/2)},
\end{align*}
where (a) follows from Proposition~\ref{prop:TV-eigen} and  \eqref{eq:eigen-lb-0}, and (b) holds by the inequality $$g^\wrap(\theta)\geq \inf_{\theta\in[\delta b,\delta b+2\pi)}g(\theta)=\frac{1}{\delta} h_\alpha\pth{b+\frac{2\pi}{\delta}}\geq \frac{1}{\delta}h_{\alpha}\pth{\frac{3\pi}{\delta}},\quad \theta\in [0,2\pi].$$
Since $\frac{3\pi}{\delta}\in [b,kb]$, we have $h_{\alpha}\pth{\frac{3\pi}{\delta}}\gtrsim_\alpha \frac{1}{\log r}\frac{\beta^\alpha}{m^{\alpha+1}}$, 
and the desired result follows. 
\end{proof}
\label{app:rate-extension}
\subsection{Proofs for Section~\ref{sec:extension-other}}
\label{app:highD-ub-lb}
To begin with, we introduce some basic concepts of moment tensors. 
For a random variable $U\in \reals^d$, define its order-$\ell$ moment tensor in $(\reals^d)^{\otimes\ell}$ as
\begin{equation*}
    M_{\ell}(U) \triangleq \mathbb{E}[\underbrace{U \otimes \cdots \otimes U}_{\ell \text { times }}].
\end{equation*}
For example, $M_1(U)=\Expect[U]$, and $M_2(U-\Expect[U])$ is the covariance matrix of $U$.
Then, for a multi-index $\bfj=(j_1,\ldots,j_\ell)\in [d]^\ell$, the $\bfj\Th$ entry of $M_{\ell}(U)$ is $m_{\bfj}(U)\triangleq \Expect[U_{j_1}\cdots U_{j_\ell}]$.
Also, with $U\sim P$, we write $M_{\ell}(P) =M_{\ell}(U)$. 
The Frobenius norm of a tensor $T\in (\reals^d)^{\otimes\ell}$ is defined as $\|T\|_F\triangleq \sqrt{\inner{T,T}}$ , where the tensor inner
product is $\inner{S,T}=\sum_{\bfj\in [d]^\ell} S_{j_{1}, \ldots, j_{\ell}} T_{j_{1}, \ldots, j_{\ell}}$.
The following result is an extension of Lemma~\ref{lem:mm1} to multivariate distributions.

\begin{lemma}
    \label{lem:mm1-highD}
    Suppose that $P,Q$ are supported on the $d$-dimensional $\ell_2$-ball of radius $R$. If $L>4(dR)^2$ and $M_\ell(P)=M_\ell(Q)$ for $\ell\in[L-1]$, then
\begin{equation*}
  \chi^{2}\left(f_P \| f_Q\right) \leq 4\exp\pth{\frac{R^{2}}{2}}\pth{\frac{4ed^2R^2}{L}}^L.
\end{equation*}
\end{lemma}

\begin{proof}
Denote $\mu=\Expect[Q]$ and let $U\sim P$, $V\sim Q$, $U-\mu\sim {P}^\prime_\mu$, and $V-\mu\sim {Q}^\prime_\mu$. 
Following the proof argument of \cite[Theorem 4.2]{Doss20}, we have that 
\begin{align}
\label{eq:chi2-moment-highd}
        \chi^{2}\left(f_{P} \| f_{Q}\right)=
        \chi^{2}\left(f_{P^\prime_\mu} \| f_{Q^\prime_\mu}\right) 
        \leq e^{\frac{R^{2}}{2}} \sum_{\ell \geq 1} \frac{\left\|M_{\ell}(P^\prime_\mu)-M_{\ell}\left(Q^\prime_\mu\right)\right\|_{\mathrm{F}}^{2}}{\ell !} d^{\ell}.
\end{align}
Note that $M_{\ell}(P)=M_{\ell}(Q)$ for all $\ell\in[L-1]$ implies that $M_{\ell}(P_\mu')=M_{\ell}(Q_\mu')$ for all $\ell\in[L-1]$. 
Since ${P}^\prime_\mu,{Q}^\prime_\mu\in\Pbdd{2R,d}$, we have $\|M_{\ell}(P^\prime_\mu)\|_F^2,\|M_{\ell}(Q^\prime_\mu)\|_F^2 \le d^\ell (2R)^{2\ell} $.
It follows from \eqref{eq:chi2-moment-highd} that
\begin{align*}
    \chi^{2}\left(f_{P} \| f_{Q}\right)
 \le e^{\frac{R^{2}}{2}} \sum_{\ell \geq L} 
\frac{d^{2\ell}[2\left(2R\right)^\ell]^{2}}{\ell !}
\leq 4 e^{\frac{R^{2}}{2}} \pth{\frac{4ed^2R^2}{L}}^L,
\end{align*}
where the last inequality follows from the Poisson tail bound $\pbb[X\geq L]\leq e^{-4d^2R^2}\left(\frac{4e d^2R^2}{L}\right)^{L}$ for $X\sim\operatorname{Poisson}(4d^2R^2)$ and $L>4d^2R^2$ 
{\cite[Theorem 4.4]{mitzenmacher2005probability}}.
\end{proof}

\begin{proof}[Proof of Proposition~\ref{prop:highdim-ub}]
Note that there are $\binom{d+\ell-1}{\ell}$ distinct entries in $M_\ell(P)$. 
Hence, the total number of distinct entries in $M_1(P),\ldots,M_{L-1}(P)$ is 
$C_{L,d}\triangleq\sum_{i=1}^{L-1} \binom{d+\ell-1}{\ell} =\binom{d+L-1}{L-1}$.
Denote $T_{L}(P)\in \reals^{C_{L,d}}$ as the tuple consisting of these distinct entries.
Consider the following convex set of tuples
$$\sth{ T_{L}(P): P\in\Pbdd{M,d}}
\subseteq \reals^{C_{L,d}}.$$
By Carath\'eodory's theorem, 
there exists a distribution $P^\prime$ supported on no more than $C_{L,d}\leq L^d$ atoms such that $m_\bfj(P)=m_\bfj(P^\prime)$ for all $\bfj\in[d]^\ell, \ell \in [L-1]$.
By Lemma~\ref{lem:mm1-highD}, we have
\begin{align*}
      \chi^{2}\left(f_{P^\prime} \| f_{P}\right)
      \le 4\exp\left(-L\log\frac{L}{M^2}+\frac{M^2}{2}
      +\log(4ed^2)L\right).
    \end{align*}
Consequently, when $L\geq \kappa_d M^2$ for $\kappa_d= 64e^3d^4$,
\begin{equation}
    \label{eq:calE-d-global}
    \appr(L^d,\Pbdd{M,d},\chi^{2})\leq \exp\left(-\frac{L}{2}\log\frac{L}{M^2}\right).
\end{equation}
Suppose that $6\sqrt{3\kappa_d}M\leq L\leq \kappa_d M^2$ holds. Set $K=\floor{\frac{6\kappa_d M^2}{L}}\geq 6$.
Denote $r=\frac{M}{K}$. Let $\calB = \{ B_i=B(u_i,r) : i\in [N] \}$ be a $r$-covering of $B_d(M)$ under the Euclidean distance, where $N\leq (3K)^d$ by \cite[Corollary 4.2.13]{HDP}. Define $\tilde B_i = \{u \in B_d(M): i=\min\{j:j\in \argmin_{k\in[N]} \|u_k-u\|\}\} $. Since $\|u_i-u\|=\min_{j\in[N]} \|u_j-u\|\leq r$ for any $u_i\in \tilde B_i$, we have $\tilde B_i \subset B(u_i,r)$. By definition, $\{\tilde{B}_i\}_{i=1}^N$ consists of a partition of $B_d(M)$ (that is, each $\tilde{B}_i$ is disjoint and $B_d(M)=\cup_{i=1}^N \tilde{B}_i$). 
Let $P_{k}$ denote the conditional distribution of $P$ on $\tilde{B}_k$.
Let $m=L^d$, $\tilde{m}= \floor{\frac{m}{(3K)^d}} = \floor{\pth{\frac{L}{3K}}^d}$, and $\tilde{L}= \floor{\tilde{m}^{1/d}}$.  
Note that the condition $\tilde{L} \ge \kappa_d \pth{\frac{M}{K}}^2$ holds by
\begin{align*}
    K^2\tilde{L}&
    \stepa{\geq} K^2\floor{\frac{L}{3K}}
    \geq  \floor{\frac{6\kappa_d M^2}{L}}^2 \floor{\frac{L^2}{18\kappa_d M^2}}
\stepb{\geq} \pth{\frac{6}{7}\frac{6\kappa_d M^2}{L}}^2
\pth{\frac{6}{7}\frac{L^2}{18\kappa_d M^2}} 
\geq \kappa_d M^2,
\end{align*}
where (a) follows from
$\floor{x^d}\geq\floor{x}^d$ for $d\in \naturals$ and $x\geq 0$, and (b) applies
$\floor{x}\geq \frac{c}{c+1} x$ for all $x\geq c\in\naturals$ and $\min\sth{\frac{L^2}{18\kappa_d M^2},\frac{6\kappa_d M^2}{L}}\geq 6$. 
Then~\eqref{eq:calE-d-global} implies that, for each $P_{k}$, there exists $\tilde P_{k}$ supported on at most $\tilde m$ atoms 
 such that
\begin{align*}
       \chi^2(f_{\tilde P_{k}}\|f_{P_k})
\leq \exp\left(-\frac{\tilde{L}}{2}\log\frac{\tilde{L}}{\left(\frac{M}{K}\right)^2}\right)
\leq \exp\left(-\frac{L^2}{42\kappa_d M^2}\log\kappa_d\right).
\end{align*}
Finally, define $P_m \triangleq \sum_{k=1}^N P(\tilde{B}_k) \tilde P_{k}$ supported on at most $\tilde{m}N\leq m$ atoms.
Note that $P=\sum_{k=1}^N P(\tilde{B}_k) P_{k}$ since $\{\tilde{B}_k\}_{k=1}^N$ is a partition of $B_d(M)$. By Jensen's inequality,
\begin{equation*}
    \chi^{2}\left(f_{P_{m}} \| f_{P}\right)
    \leq \sum_{k=1}^N P(\tilde{B}_k) \chi^2(f_{\tilde P_{k}}\|f_{P_{k}})
    \leq \exp\pth{-\frac{\log\kappa_d }{42\kappa_d}\frac{L^2}{M^2} }.
    \qedhere
\end{equation*}
\end{proof}

To prove the lower bound, we first introduce the multilevel Toeplitz matrices, which are generalizations of Toeplitz matrices that naturally arise in multidimensional Fourier analysis \cite{Pestana2019,TYRTYSHNIKOV1998multiToeplitz}. 
Let $\bfu_L:\reals^d\mapsto \complex^{L^d}$ denote the multivariate trigonometric functions ordered by the $L$-base representation
\[
(\bfu_L(\theta))_k = \exp\pth{i \sum_{n=0}^{d-1} j_n\theta_n},
\quad \text{if}~k= \sum_{n=0}^{d-1} j_n L^n, j_n\in\{0,1,\dots,L-1\}.
\]
For a random variable $X$ (or the probability distribution $P$ thereof) supported on $\reals^d$, define
$$\bfA_{L}(X)=\bfA_{L}(P)=\Expect[\bfu_L(X)\bfu_L^\star(X)].$$
By definition, $\bfA_{L}(P)$ has a nested block structure with $L\times L$ blocks, and each block contains $L\times L$ smaller blocks.
For example, when $d=2$ and $L=2$,
$$
\bfA_{L}(P)=\pth{
\begin{matrix}
    a_{(0,0)} & a_{(0,-1)} & a_{(-1,0)} & a_{(-1,-1)}\\
    a_{(0,1)} & a_{(0,0)} & a_{(-1,1)} & a_{(-1,0)}\\
    a_{(1,0)} & a_{(1,-1)} & a_{(0,0)} & a_{(0,-1)}\\
    a_{(1,1)} & a_{(1,0)} & a_{(0,1)} & a_{(0,0)}
\end{matrix}
},
$$
where $a_{\bfj}=\Expect \exp(i X^\top \bfj )$ denotes the trigonometric moment.

 Let $X\sim P$ and $X_m\sim P_m$ for any $P_m\in\calP_m$. Note that  the rank of $\bfA_{L}(P_m)$ is at most $m$. 
Applying~\eqref{eq:vr-1} yields that, for any $\delta>0$, 
\begin{align*}
    \TV\left(f_{P_{m}} , f_{P}\right) 
&\geq \frac{1}{2} \sup _{\omega\in \reals^d} \exp\pth{-\frac{\|\omega\|^2}{2}} \left|\mathbb{E}\qth{\exp\pth{i \langle\omega,X\rangle }}-\mathbb{E}\qth{\exp\pth{i\langle\omega,X_m\rangle }}\right| \\
&\geq  \frac{1}{2} \exp\pth{-\frac{d\delta^2(L-1)^2}{2}} \max_{\bfj\in \{-L+1,\dots,L-1\}^d}  \left|\mathbb{E}\qth{\exp\pth{i \langle\bfj,\delta X\rangle }}-\mathbb{E}\qth{\exp\pth{i\langle\bfj,\delta X_m\rangle }}\right| \\
&\geq \frac{1}{2L^d} \exp\pth{-\frac{d\delta^2(L-1)^2}{2}}\|\bfA_{L}(\delta X)-\bfA_{L}(\delta X_m)\|_{F}. 
\end{align*}
By Lemma~\ref{lem:eym}, if $L^d\geq m+1$,
\begin{equation}
    \label{eq:TV-eigen-d}
    \appr(m,P,\TV) \geq \sup_{\delta>0}\frac{1}{2L^d} \exp\pth{-\frac{d\delta^2(L-1)^2}{2}} \lambda_{\min}(\bfA_{L}(\delta X)).
\end{equation}
Then, we extend the wrapped density approach {and the orthogonal expansion approach described in Section~\ref{sec:spectral}}  to multi-dimensional case.

\paragraph{{Wrapped density.}}
Given a density function $g$ supported on a subset of $\reals^d$,  
define the corresponding wrapped density $g^{\wrap}$ on $[-\pi,\pi]^d$ as 
\begin{equation*}
    g^{\wrap}(\theta)=\begin{cases}
        \sum_{j\in \integers^d} g(\theta-2\pi j), & \theta\in[-\pi,\pi]^d;\\
        0, &\text{otherwise}.
    \end{cases}
\end{equation*}
Denote $Y=\delta X\sim g$ and $Y^{\wrap}\sim g^{\wrap}$.
Be definition, $\bfA_L(Y)=\bfA_L(Y^{\wrap})$.
The following result is a straightforward corollary of \cite[Lemma 4.1]{Pestana2019},
which can be viewed as a multi-dimensional version of \eqref{eq:eigen-lb-0} and follows from the same argument as \eqref{eq:eigen-lb-0why}.
\begin{lemma}
    \label{lem:eigen-lb-d}
    Let $g$ be a density function on $\reals^d$ 
    that is symmetric about $0$.
    Then, for $Y\sim g$,
   $$ \lambda_{\min}(\bfA_L(Y))  = \lambda_{\min}(\bfA_L(Y^{\wrap})) 
    \geq \min \{(2\pi)^d g^{\wrap}(\theta) : \theta\in [-\pi,\pi]^d\}.$$
\end{lemma}

\paragraph{Orthogonal expansion.}
We recall some notations in Section~\ref{sec:spectral}.
Let $\{\varphi_n\}$ be the orthonormal polynomials on the unit circle associated with $Y=\delta X$, where $X\sim P$ is a one-dimensional random variable and $\delta>0$.
 Denote $\bfR_{L-1,\delta}=(R_{jk})_{j,k=0}^{L-1}$ as the associated coefficient matrix of $\{\varphi_k\}$ of order $L$. 
For each multi-index $\boldsymbol{\ell}=(\ell_0,\ldots,\ell_{d-1})\in \integers_{\geq 0}^{d}$, we define the $\boldsymbol{\ell}\Th$ multivariate orthogonal polynomial as 
\begin{equation*}
    \varphi_{\boldsymbol{\ell}}(\bfw)= \prod_{j=0}^{d-1} \varphi_{\ell_j} (w_j), \quad w\in\complex^d.
\end{equation*}
Then, $\{\varphi_{\boldsymbol{\ell}}(\bfw)\}$ is a basis of $d$-variate polynomial space.
Let $\bfY=(Y_0,...,Y_{d-1})$ be $d$ \iid\ copies of $Y$. By the orthogonality and independence,
\begin{equation}
    \label{eq:orthonormal-multivariate}
    \Expect\qth{ \varphi_{\boldsymbol{\ell}}(e^{i\bfY})\overline{\varphi_{\boldsymbol{\ell'}}(e^{i\bfY})} }= \prod_{j=0}^{d-1} \Expect\qth{ \varphi_{\ell_j}(e^{iY_j})\overline{\varphi_{\ell_j'}(e^{iY_j})} }=\indc{\boldsymbol{\ell}=\boldsymbol{\ell}'}.
\end{equation}
 
The next proposition lower bounds the minimum eigenvalue of multilevel Toeplitz matrices as an extension of Proposition~\ref{prop:ortho}. 
\begin{proposition}
    \label{prop:ortho-multi}
    Let $X\sim P$, $\delta>0$, and $\bfY=(Y_0,...,Y_{d-1})$, where each $Y_j$ is an \iid\ copy of $Y=\delta X$.
    Then, with the above notations,
\begin{equation*}
    \lambda_{\min}(\bfA_L(\bfY)) \geq \frac{1}{\|\bfR_{L-1,\delta}\|_F^{2d}}.
\end{equation*}
\end{proposition}

\begin{proof} 
 Let $\bfw=(w_0,\ldots,w_{d-1})^\top\in \complex^{d}$ with $w_j=e^{i y_j}$.
 Denote $\boldsymbol{\ell}_k=(\ell_{k,0},\ldots,\ell_{k,d-1})^\top\in\{0,1,\dots,L-1\}^d$ as the $L$-base representation vector for $k= \sum_{j=0}^{d-1} \ell_{k,j} L^j\in \{0,1,\dots,L^d-1\}$.
 For any $\bfx=(x_0,\ldots,x_{L^d-1})^\top \in \complex^{L^d} \backslash \{\mathbf{0}\}$, 
 let $\pi(\bfw)=\sum_{k=0}^{L^d-1} x_k\bfw^{\boldsymbol{\ell}_k}$ with $\bfw^{\boldsymbol{\ell}_k}\triangleq \prod_{j=0}^{d-1} w_j^{\ell_{k,j}}$. 
Expand $\pi(\bfw)$ under the orthogonal basis $\{\varphi_{\boldsymbol{\ell}_k}\}$ as 
\[
\pi(\bfw)
=\sum_{k=0}^{L^d-1} x_k\bfw^{\boldsymbol{\ell}_k}
=\sum_{k=0}^{L^d-1} c_k\varphi_{\boldsymbol{\ell}_k}(\bfw).
\]
Letting $\bfc=(c_0,\ldots,c_{L^d-1})^\top$ and $\bfA \equiv\bfA_L(\bfY)$, the orthogonality of $\{\varphi_{\boldsymbol{\ell}_k}\}$ in~\eqref{eq:orthonormal-multivariate} implies that 
\begin{equation*}
\|\bfc\|^2=\Expect\left[|\pi(e^{i\bfY})|^2\right] 
=\bfx^\star \bfA \bfx,
\end{equation*}
where the last equality follows from the fact that $A_{kk'}=\Expect\exp(i(\boldsymbol{\ell}_k-\boldsymbol{\ell}_{k'})^\top \bfY)$.

Let $\bfR\equiv\bfR_{L-1,\delta}$ be the coefficient matrix of $\{\varphi_k\}$ of order $L$. Note that 
\begin{align*}
    \pi(\bfw)&=\sum_{k=0}^{L^d-1} c_k\varphi_{\boldsymbol{\ell}_k}(\bfw)\\
    &=\sum_{k=0}^{L^d-1} c_k \qth{\prod_{j=0}^{d-1}  \pth{ \sum_{t=0}^{\ell_{k,j}} R_{\ell_{k,j}t}w_j^{t} } }\\
    &= \sum_{k'=0}^{L^d-1} \qth{\sum_{k:\boldsymbol{\ell}_{k'}\leq \boldsymbol{\ell}_{k}} c_k \pth{\prod_{j=0}^{d-1} R_{\ell_{k,j}\ell_{k',j}}
    }}\bfw^{\boldsymbol{\ell}_{k'}},
\end{align*}
where $\boldsymbol{\ell}_{k'}\leq \boldsymbol{\ell}_{k}$ denotes $\ell_{k',j}\leq \ell_{k,j}$ for all $j=0,\ldots,d-1$. 
By comparing the coefficients, we have that 
$\bfx^\top = \bfc^\top \bfS$, where $\bfS=(S_{kk'})_{k,k'=0}^{L^d-1}\in\reals^{L^d \times L^d}$
is defined as $S_{kk'}=\prod_{j=0}^{d-1} R_{\ell_{k,j}\ell_{k',j}}$ if $\boldsymbol{\ell}_{k'}\leq \boldsymbol{\ell}_{k}$,
and $S_{kk'}=0$ otherwise. 
Then, by definition, 
\begin{align*}
    \|\bfS\|_F^2  &= \sum_{k=0}^{L^d-1} \sum_{k': \boldsymbol{\ell}_{k'}\leq \boldsymbol{\ell}_{k}} \prod_{j=0}^{d-1} R_{\ell_{k,j}\ell_{k',j}}^2 \\
    & = \sum_{0\leq \ell'_0 \leq \ell_0 \leq L-1} \cdots \sum_{0\leq \ell'_{d-1} \leq \ell_{d-1} \leq L-1} \prod_{j=0}^{d-1} R_{\ell_{j}\ell'_{j}}^2 \\
    & = \prod_{j=0}^{d-1} \pth{\sum_{0\leq \ell_j' \leq \ell_j \leq L-1} R_{\ell_{j}\ell'_{j}}^2}= \|\bfR\|_F^{2d}.
\end{align*}
Hence, $\|\bfx\|\leq \|\bfc\| \|\bfS\|\leq \|\bfc\|\|\bfS\|_F = \|\bfc\| \|\bfR\|_F^{d}$.
Finally, applying \eqref{eq:rayleigh}, we have
\begin{equation*}
    \lambda_{\min}(\bfA)
    =\min_{\bfx\ne \mathbf{0}} \frac{ \bfx^\star \bfA \bfx}{\|\bfx\|^2}
    \geq  \min_{\bfx\ne \mathbf{0}}\frac{\|\bfc\|^2}{\|\bfc\|^2\|\bfR\|_F^{2d}} =\frac{1}{\|\bfR\|_F^{2d}}.
    \qedhere
\end{equation*}
\end{proof}

\begin{proof}[Proof of Proposition~\ref{prop:highD-lb}]
Let $L = \ceil{(m+1)^{1/d}}$ satisfying $m+1\le L^d\leq 2^d(L-1)^d \leq 2^d m$. 
Let $X\sim P= \Unif([-M/\sqrt{d},M/\sqrt{d}]^d)$ and $\delta=\frac{\pi\sqrt{d}}{M}$. 
Then, $\bfA_{L}(\delta X)=I_{L^d}$ and thus {the conclusion \eqref{eq:highD-lb-bdd-1}} follows from  \eqref{eq:TV-eigen-d}.

{
   Next we prove \eqref{eq:highD-lb-bdd-2}. Assume that $L\geq \frac{eM^2}{2d}$. Set $b={ \sqrt{\frac{M^2}{Ld}\log\frac{2Ld}{e M^2}}}\in (0,\frac{\sqrt{2}}{e})$ and $\delta=b\frac{\sqrt{d}}{M}$. Let $Y\sim \tilde{P}$ be as defined in \eqref{eq:f_test_unif} that is supported on $[-b,b]$, and denote $P$ as the distribution of $\frac{Y}{\delta}$ supported on $[-\frac{M}{\sqrt{d}},\frac{M}{\sqrt{d}}]$.
   Consider the test distribution $P^{\otimes d}\in \Pbdd{M,d}$. 
   Let $\bfR$ be the order-$L$ associated coefficient matrix of the orthogonal system with $\tilde{P}$. Similar to \eqref{eq:R_ub_unif}, we have $\|\bfR\|_F^2 \leq \exp\pth{O(L\log\frac{1}{b})}$. Applying \eqref{eq:TV-eigen-d} and Proposition~\ref{prop:ortho-multi}, we have  
   \begin{align*}
        \appr(m,P,\TV) &\geq \frac{1}{2L^d} \exp\qth{- \Omega\pth{d\delta^2L^2+dL\log\frac{1}{b}}} \\
        &\geq \frac{1}{2^{d+1}m}\exp\qth{- \Omega\pth{dL\log\frac{dL}{M^2}}}.
   \end{align*}
   The desired result follows.
}
\end{proof}
Next, by applying similar analyses for the upper and lower bounds, we present the extension for approximating multivariate Gaussian distributions.
\begin{proposition}
     Suppose that $\sigma\geq \sigma_0>0$. There exist $C_0,C_1$ which depend on $d$ and $\sigma_0$ such that, when $m^{1/d}\geq C_0 \sigma$,
 \[
 \appr(m,N(0,\sigma^2 I_d),\chi^2)\leq \exp\pth{- C_1 \frac{m^{\frac{1}{d}}}{\sigma}}.
 \]
Furthermore, for any $m\in\naturals$ and $\sigma>0$,
 \[
     \appr(m,N(0,\sigma^2 I_d),\TV)\geq \frac{1}{2^{\frac{d}{2}+1}m^{\frac{1}{2}}\sigma^{\frac{d}{2}}}\exp\pth{-\frac{d\pi m^{\frac{1}{d}}}{\sigma} }.
 \]
\end{proposition}
\begin{proof}
    Let $X\sim P=N(0,\sigma^2 I_d)$.
    There exist sufficiently large $C_0$ and $C'$ depending on $d$ and $\sigma_0$ such that, if $\tilde m \triangleq m^{1/d} \ge C_0 \sigma$ and $t=C'\sqrt{\tilde m\sigma}$, then $6\sqrt{3\kappa_d}t \le \tilde m \le  \kappa_d t^2$ and $t\geq 2\sigma \sqrt{d}$ hold.
    Let $A=B_d(t)$ and $P_A$ denote the conditional distribution of $P$ on $A$.
     Applying Lemma~\ref{lem:chi2-truncate}, we have 
     \begin{align*}
         \appr(m,P,\chi^2)&\leq \frac{2}{P(A)}\pth{\appr(m,P_A,\chi^2)+P(A^c)}\\
         &\stepa{\leq}  \frac{2}{P(A)}\qth{\exp\pth{-\frac{\log\kappa_d }{42\kappa_d}\frac{\tilde m^{2}}{t^2} }+\exp\pth{-\frac{1}{2}\pth{\frac{t}{\sigma}-\sqrt{d}}^2}}\\
         &\leq \frac{2}{P(A)}\qth{\exp\pth{-\frac{\log\kappa_d }{42\kappa_d}\frac{\tilde m^{2}}{t^2} }+\exp\pth{-\frac{t^2}{8\sigma^2}}}\\
         &\leq \exp\pth{- C_1 \frac{\tilde m}{\sigma}},
     \end{align*}
     where (a) applies the  tail bound of the $\chi^2$ distribution $P[\|X/\sigma\|_2 \geq \sqrt{d}+\sqrt{2x}]\leq \exp(-x)$ \cite[Lemma 1]{Massart2000}, and $C_1$ depends on $d$ and $\sigma$. 

For the lower bound, applying \eqref{eq:TV-eigen-d} and Lemma~\ref{lem:eigen-lb-d} with $L = \ceil{(m+1)^{1/d}}$ and $\delta=\sqrt{\frac{\pi}{\sigma (L-1)}}$,
we obtain
\begin{align*}
    \appr(m,N(0,\sigma^2 I_d),\TV)&\geq  \frac{(2\pi)^d}{2L^d} \exp\pth{-\frac{d\delta^2 (L-1)^2}{2}} \min_{\theta\in [-\pi,\pi]^d} \frac{1}{(2\pi )^{d/2}(\delta \sigma)^d}e^{- \frac{\Norm{\theta}^2}{2\delta^2\sigma^2}}
    \\
    &= \frac{(2\pi)^{\frac{d}{2}}}{2(L\delta\sigma)^d} \exp\pth{-\frac{d\delta^2 (L-1)^2}{2}-\frac{d\pi^2}{2\delta^2\sigma^2} }\\
    &\geq \frac{1}{2(L\sigma)^{\frac{d}{2}}}\exp\pth{-\frac{d\pi (L-1)}{\sigma} }\\
    &\geq \frac{1}{2^{\frac{d}{2}+1}m^{\frac{1}{2}}\sigma^{\frac{d}{2}}}\exp\pth{-\frac{d\pi m^{1/d}}{\sigma} }. \qedhere
\end{align*}
\end{proof}

\section{Proof of the convergence rates of NPMLE}
\label{app:proof-mle-rate}
\subsection{Unconstrained NPMLE}
\begin{lemma}
    \label{lem:npmle_uc}
    Let $\calQ_0$ be the collection of all distributions on $\reals$. Then, for any $\epsilon\in (0,1/2]$,
    \begin{align*}
        \comp ( \epsilon , \calQ_0 , L_{\infty,[-M,M]} ) \lesssim M\sqrt{\log\epsilon^{-1}} \vee \log\epsilon^{-1}.
    \end{align*}
\end{lemma}
\begin{proof}
    Set $t = M+4\sqrt{\log{\frac{1}{\epsilon}}}$ and $m=12t\sqrt{\kappa\log\frac{1}{\epsilon}}$, where $\kappa$ is defined in Theorem~\ref{thm:ub-bdd}. Fix any $P\in\calQ_0$, and
    define $P_{I_t}$ and $P_{I_t^c}$ as $P$ conditioned on $I_t=[-t,t]$ and $I_t^c$, respectively. Then, we have  
    \[
    f_P=P(I_t)f_{P_{I_t}} + P(I_t^c)f_{P_{I_t^c}}.
    \]
    Applying Theorem~\ref{thm:ub-bdd} with $3\sqrt{\kappa}t \leq m\leq  \kappa t^2$ yields that $\appr(m,\Pbdd{t},\chi^2)\leq \epsilon^{8}$. By Lemma~\ref{lem:f-divs}, there exists $\tilde P\in\calP_m$ such that $\TV(f_{\tilde P},f_{P_{I_t}})\leq \epsilon^{4}$. Since $\|f_P^\prime\|_\infty\leq \|\phi^\prime\|_\infty =(2\pi e)^{-1/2}$, $\TV(f_{\tilde P},f_{P_{I_t}})\leq \epsilon^4$ then implies $\Norm{f_{\tilde P}-f_{P_{I_t}}}_\infty \leq (\frac{2}{\pi e})^{1/4}\epsilon^2$.  
    Let $P_{m+1}=P(I_t)\tilde P + P(I_t^c)\delta_t \in \calP_{m+1}$. Then, we have
    \begin{align*}
        |f_P(x)-f_{P_{m+1}}(x)|
        &= \abs{P(I_t)f_{P_{I_t}}(x) + P(I_t^c)f_{P_{I_t^c}}(x) - 
        P(I_t)f_{\tilde P}(x) - P(I_t^c)\phi(x-t)
        }\\
        &\le \abs{f_{P_{I_t}}(x)-f_{\tilde P}(x)} + \abs{f_{P_{I_t^c}}(x)} + \abs{\phi(x-t)}.
    \end{align*}
    Since $P_{I_t^c}$ is supported on $\{|x|\geq t\}$, we have $f_{P_{I_t^c}}(x)\leq \phi (t-M)\leq \epsilon^8$ for any $|x|\le M$. Hence,
       \begin{align*}
       \sup_{|x|\le M}|f_P(x)-f_{P_{m+1}}(x)| & \le  \pth{\frac{2}{\pi e}}^{1/4}\epsilon^2 + 2\epsilon^{8} \leq \epsilon.
       \end{align*}    
Consequently, $\comp ( \epsilon , \calQ_0 , L_{\infty,[-M,M]} )\leq m+1$, and the desired result follows. 
\end{proof}

\subsection{Constrained NPMLE}
We introduce some notations for the metric entropy. 
For $\epsilon>0$, an $\epsilon$-net of a set $\calF$ with respect to a metric $d$ is a set $\calN$ such that for all $f\in\calF$, there exists $g\in\calN$ such that $d(g,f)\le \epsilon$. 
The minimum cardinality of $\epsilon$-nets is denoted by $N(\epsilon,\calF,d)$.
Given a family $\calP$ of distributions, define $\calF_\calP\triangleq \{f_P:P\in\calP\}$.

The following lemma bounds the $L_\infty$-metric entropy of the Gaussian mixture densities with compactly supported and sub-Weibull mixing distributions. Note that this is an improvement of the previous result \cite[Lemma 2]{Zhang08} which deals with \textit{truncated} $L_\infty$-norm on a compact interval. 
Combined with our main result in Theorem \ref{thm:main-bdd}, this also improves the \textit{untruncated} $L_\infty$-entropy estimate in \cite[Lemma 2]{GV07} which gives 
$\log N(\epsilon,\calF_{\Pbdd{M}},L_{\infty}) \lesssim M \log \frac{1}{\epsilon} \log \frac{M}{\epsilon}$.
The key idea is to use the smoothness of the Gaussian mixtures to relate $L_\infty$-error (pointwise) to $L_1$-error (TV), the latter of which can be further bounded by moment matching.

\begin{lemma}
    \label{lmm:npmle-metric-entropy}
    Let $0<\epsilon<1$.
    There exists a universal constant $C$ such that 
    \[
        \log N(\epsilon,\calF_{\Pbdd{M}},L_{\infty}) \le C m^\star(\epsilon,\Pbdd{M}, \TV)\log \frac{1}{\epsilon}.
    \]
    Additionally, there exists a constant $C_\alpha$ depending only on $\alpha$ such that
    \[
        \log N(\epsilon,\calF_{\calP_\alpha(\beta)},L_{\infty}) \le C_\alpha m^\star(\epsilon, \calP_\alpha(\beta), \TV)\log \frac{1}{\epsilon}.
    \]
\end{lemma}
\begin{proof}
    First, we prove the result for $\calP=\Pbdd{M}$. 
    Let $m=\comp(\epsilon,\calP,\TV)$ and $t=M$. 
    Let $\calS_m\triangleq \{(p_1,\dots,p_m)|\sum_{i=1}^m p_i=1, p_i\geq 0\}$ and  $\calN_m\subseteq \calS_m$ be the smallest $\epsilon$-net of $\calS_m$ under the $L_1$-distance.
    Let $\calL\triangleq\{0,\pm\epsilon,\dots,\pm\floor{\frac{t}{\epsilon}}\epsilon\}$.
    Define the following set of finite mixture densities 
    $$\calC\triangleq\left\{\sum_{j=1}^m w_j N(\theta_j,1): (w_1,\dots,w_m)\in \calN_m, \theta_1 \le \dots \le \theta_m, \{\theta_j\}_{j=1}^m\subseteq \calL\right\}.$$
    The cardinality of $\calC$ is upper bounded by 
    \begin{equation*}
        |\calC|\leq \binom{m+|\calL|-1}{m}|\calN_m| \stepa{\leq}
    \frac{\pth{\frac{2t}{\epsilon}+m}^m}{m!}  2m\left(1+\frac{1}{\epsilon}\right)^{m-1} 
    \leq 2m\pth{\frac{2t}{\epsilon}+m}^m\pth{\frac{2e}{\epsilon m}}^m,
    \end{equation*}
    where (a) holds by $\binom{n}{m}\leq \frac{n^m}{m!}$ 
    and the fact that $|\calN_m| \leq 2m\left(1+\frac{1}{\epsilon}\right)^{m-1}$ \cite[Corollary 27.4]{PW-it}.

    Then, we prove the covering property of $\calC$.
    By definition, for any $P\in\calP$, there exists $P_m=\sum_{i=1}^m w_j\delta_{\theta_j}$ with $\theta_1 \le \dots \le \theta_m$ and $|\theta_j|\le t$ for all $j$ such that $\TV(f_{P_m},f_P)\leq \epsilon$. 
    Let $\theta_j^\prime=\theta_j\frac{\floor{|\theta_j|/\epsilon}}{|\theta_j|/\epsilon} \in \calL$ and consider $P_m'=\sum_{i=1}^m w_j\delta_{\theta_j'}$. 
    By Jensen's inequality and the fact that $\TV(N(0, 1), N(\epsilon, 1))  \leq \sqrt{\frac{1}{2\pi}}\epsilon$,
    we have $\TV(f_{P_m},f_{P_m^\prime})\leq  \sqrt{\frac{1}{2\pi}}\epsilon$. 
    Finally, we construct $P_m''=\sum_{j=1}^m w^\prime_j\delta_{\theta_j^\prime}\in\calC$ by choosing $w^\prime\in\calN$ so such $L_1(w,w^\prime)\le \epsilon$.
    By triangle inequality, $\TV(f_{P_m^{\prime}},f_{P_m^{\prime\prime}})\leq \frac{1}{2}L_1(w,w^\prime) \leq\frac{1}{2} \epsilon.$
    Hence, $\calC$ is a $2\epsilon$-net of $\calF_\calP$ under $\TV$, and thus 
     \begin{align}
        \label{eq:appr-entropy-cover-2}
     \log N(2\epsilon,\calF_\calP,\TV)
     \leq m\log\pth{\frac{4et}{\epsilon^2m}+\frac{2e}{\epsilon }}+\log 2m\stepa{\leq} O\pth{m \log\frac{1}{\epsilon}},
     \end{align}
     where (a) holds since $m \gtrsim M=t$ by Theorem~\ref{thm:lb-unif}.
    
    Finally, since $\|f_P^\prime\|_\infty\leq \|\phi^\prime\|_\infty =(2\pi e)^{-1/2}$, 
    $L_\infty(f_P,f_Q)\geq (\frac{8}{\pi e})^{1/4}\sqrt{\epsilon}$ implies $\TV(f_P,f_Q)\geq 2\epsilon$. 
    Then, 
    \begin{equation}
            \label{eq:appr-entropy-Linf}
            \log N(\pth{8/(\pi e)}^{1/4}\sqrt{\epsilon},\calF_\calP,L_\infty)\leq \log N(2\epsilon,\calF_\calP,\TV)\leq O\pth{m \log\frac{1}{\epsilon}},
    \end{equation}
    and equivalently, $\log \calN(\epsilon,\calF_\calP,L_\infty)\leq O\pth{m \log\frac{1}{\epsilon}}$.

    The proof for $\calP=\calP_\alpha(\beta)$ is analogous, where we need to choose $t$ depending on $m$ and $\beta$ as in the proof of Theorem~\ref{thm:ub--subW}. 
    Specifically, let
    $t=t(m,\beta)\asymp_\alpha m^{\frac{2}{2+\alpha}}\beta^{\frac{\alpha}{2+\alpha}}$ if $m^{\alpha-2}\lesssim_\alpha \beta^{2\alpha},$ and $t\asymp_\alpha\beta \pth{m\log \frac{m^{\alpha-2}}{\beta^{2\alpha}}}^\frac{1}{\alpha}$ otherwise. 
    Plugging $t=t(m,\beta)$ into  \eqref{eq:appr-entropy-cover-2} and noting that $m \gtrsim_\alpha \beta$ by Theorem~\ref{thm:inapprox}, we likewise obtain that
    $\log \calN(\epsilon,\calF_\calP,L_\infty)\lesssim_\alpha {m \log\frac{1}{\epsilon}}$.
\end{proof} 
\begin{proof}[Proof of Theorem~\ref{thm:rate-families}]
We first prove the result for $P^\star\in  \Pbdd{M}$. 
Let $\epsilon = n^{-2(c \vee 1)}$, $\delta=s\epsilon_n $. 
Theorem~\ref{thm:main-bdd} implies that $\epsilon_n^2 \asymp \frac{m^\star}{n} \log \frac{1}{\epsilon}$, where $m^\star\triangleq m^\star(\epsilon,\calP,\TV)$. 
Consider the smallest $\epsilon$-net denoted by $\calN$ of $$\calF \triangleq \{f_P:P\in\calP, H(f_P, f_{P^\star})\ge \delta \}\subseteq \calF_\calP$$ under the $L_\infty$-distance. Without loss of generality, assume $\calN\subseteq \calF$ \cite[Exercise~4.2.9]{HDP}. 
By Lemma~\ref{lmm:npmle-metric-entropy}, we have $H_\epsilon= \log |\calN|\lesssim m^\star \log\frac{1}{\epsilon}$.
It follows from the definition of the $\epsilon$-net that, if $H(f_{\hat P}, f_{P^\star}) \ge \delta$, then there exists $g\in\calN$ such that $f_{\hat P}(x)\le g(x)+\epsilon$ for all $x\in\reals$.
By the optimality of $\hat P$, 
\begin{align}
     0&\le \sum_{i=1}^n \log \frac{f_{\hat P}(X_i)}{f_{P^\star}(X_i)}    
    \le \max_{g\in\calN}\sum_{i=1}^n \log \frac{g+\epsilon }{f_{P^\star}}(X_i). 
    \label{eq:npmle-decompose}
\end{align} 
For a fixed function $g\in\calN$, applying Chernoff bound yields that 
\[
\prob{\sum_{i=1}^n \log \frac{g+\epsilon }{f_{P^\star}}(X_i)\ge 0}
\le \exp\pth{  n \log \Expect\sqrt{\frac{g+\epsilon }{f_{P^\star}}} },
\]
where
\begin{equation}
\Expect\sqrt{\frac{g+\epsilon }{f_{P^\star}}}
\leq 
\Expect\sqrt{\frac{g }{f_{P^\star}}}
+ \sqrt{\epsilon} \Expect\sqrt{\frac{1 }{f_{P^\star}}}
= 1-\frac{1}{2}H^2(g,f_{P^\star}) + 
\sqrt{\epsilon}
\int_{\reals} \sqrt{f_{P^\star}(x)} \diff x.
    \label{eq:sqrtintegral}
\end{equation}
Since $g\in\calN$, we have  $H(g,f_{P^\star})\ge \delta$.
Additionally, applying Cauchy-Schwarz inequality yields  $\Expect\sqrt{1/f_{P^\star}}=(\int_{|x|\le 2M}+\int_{|x|> 2M}) \sqrt{f_{P^\star}(x)} \diff x \leq c_0(\sqrt{M}\vee 1) $ for some universal constant $c_0$.
Then, we have $\Expect\sqrt{(g+\epsilon)/f_{P^\star}} \le 1 - \frac{\delta^2}{2} +  c_0\sqrt{\epsilon}(\sqrt{M}\vee 1)$.
Also note that $n\epsilon_n^2\asymp m^\star \log \frac{1}{\epsilon}\gtrsim M\vee 1\geq \sqrt{M}\vee 1$ according to Theorem~\ref{thm:lb-unif}.
Hence, by the union bound, there exist absolute constants $c_1,s^\star>0$ such that for any $s>s^\star$,
\begin{align}
    \prob{\max_{g\in\calN}\sum_{i=1}^n \log \frac{g+\epsilon }{f_{P^\star}}(X_i) \ge 0 }
    & \le \exp\pth{ -n \pth{\frac{\delta^2}{2} - c_0\sqrt{\epsilon}(\sqrt{M}\vee 1)}  + H_\epsilon} \nonumber\\
    & \le \exp\pth{- c_1 s^2 m^\star \log n}. \label{eq:npmle-1}
\end{align}
Consequently, we conclude~\eqref{eq:npmle-rate}.

The proof for $P^\star \in \calP_\alpha(\beta)$ follows from a similar argument.
By using $\epsilon_n$ specified in~\eqref{eq:rate-tail}, the condition $\epsilon_n^2 \gtrsim \frac{m^\star}{n} \log \frac{1}{\epsilon}$ is satisfied by Theorem~\ref{thm:main-tail}. 
To bound the integral $\int_{\reals}\sqrt{f_P(x)} \diff x$, 
note that  $\int_{|x|\le \beta}\sqrt{f_P(x)} \diff x \leq \sqrt{2\beta}$ and 
\begin{align*}
    \int_{|x|>  \beta} \sqrt{f_P(x)} \diff x 
    &\leq \int_{|x|> \beta} \sqrt{\Expect \phi(x-\theta)\pth{\indc {|x-\theta|\leq |x|/2}+\indc{|x-\theta|> |x|/2 }}} \diff x \nonumber\\
    &\leq \int_{|x|> \beta}  \sqrt{ \phi(0)  \pbb[|\theta|\geq |x|/2]}  +\sqrt{\phi(x/2)}\diff x
    \leq c_\alpha(\beta\vee 1),
\end{align*}
where the last inequality holds by the tail probability bound~\eqref{eq:subweibull-tail} for some $c_\alpha>0$ that depends on $\alpha$. Then, Theorem~\ref{thm:inapprox} implies that $\comp\log\frac{1}{\epsilon} \gtrsim \beta \vee 1 \gtrsim_\alpha  \int \sqrt{f_P(x)} \diff x$.
Following the similar derivation, \eqref{eq:npmle-1} holds for some $c_1>0$ and $s>s^\star$, where $c_1, s^\star$ may depend on $\alpha$. 
\end{proof}

\begin{remark}
The preceding proof is simpler than existing arguments 
(e.g., the proofs of \cite[Theorem 1]{Zhang08} and its  multivariate extension \cite[Theorem 2.1]{Saha19}), which rely on truncated $L_\infty([-M,M])$ metric entropy of the mixture density. As such, in \eqref{eq:npmle-decompose} one needs to take into account the contribution from those $|X_i| \geq M$.
In comparison, we directly apply the global $L_\infty$-entropy in 
Lemma \ref{lmm:npmle-metric-entropy}, which avoids truncating the sample, and directly bound the contribution of $\int \sqrt{f_{P^\star}}$ in \eqref{eq:sqrtintegral} which can be afforded since the mixing distributions here are light-tailed.    
\end{remark}
\bibliographystyle{alpha}
\bibliography{sample}

\end{document}